\documentclass[12pt]{amsart}
\usepackage{SSdefn}
\usepackage{eucal}

\newcommand{\circled}[1]{\textcircled{\tiny #1}}

\DeclareMathOperator{\Cl}{Cl}
\DeclareMathOperator{\sdim}{sdim}

\DeclareMathOperator{\chr}{char}

\newcommand{\veeotimes}{\mathbin{\check{\boxtimes}}}
\newcommand{\hattimes}{\mathbin{\hat{\boxtimes}}}

\newcommand{\kar}{\mathrm{kar}}
\newcommand{\add}{\mathrm{add}}
\newcommand{\bone}{\mathbf{1}}

\renewcommand{\Vec}{\mathbf{Vec}}
\newcommand{\QVec}{\mathbf{QVec}}
\newcommand{\Queer}{\mathbf{Queer}}
\newcommand{\SVec}{\mathbf{SVec}}

\newcommand{\SSet}{\mathbf{SSet}}

\newcommand{\ttimes}{\rlap{$\times$}\hbox{$\hskip 1.5pt\times$}}

\newcommand{\fsA}{\mathtt{A}}
\newcommand{\fsB}{\mathtt{B}}
\newcommand{\fsC}{\mathtt{C}}
\newcommand{\fsD}{\mathtt{D}}
\newcommand{\fnC}{\mathtt{c}}
\newcommand{\fnD}{\mathtt{d}}

\makeatletter
\def\thickhline{%
  \noalign{\ifnum0=`}\fi\hrule \@height \thickarrayrulewidth \futurelet
   \reserved@a\@xthickhline}
\def\@xthickhline{\ifx\reserved@a\thickhline
               \vskip\doublerulesep
               \vskip-\thickarrayrulewidth
             \fi
      \ifnum0=`{\fi}}
\makeatother

\newlength{\thickarrayrulewidth}
\title{Supersymmetric monoidal categories}

\author{Steven V Sam}
\address{Department of Mathematics, University of California, San Diego, CA}
\email{\href{mailto:ssam@ucsd.edu}{ssam@ucsd.edu}}
\urladdr{\url{http://math.ucsd.edu/~ssam/}}
\thanks{SS was supported by NSF grant DMS-1812462.}

\author{Andrew Snowden}
\address{Department of Mathematics, University of Michigan, Ann Arbor, MI}
\email{\href{mailto:asnowden@umich.edu}{asnowden@umich.edu}}
\urladdr{\url{http://www-personal.umich.edu/~asnowden/}}
\thanks{AS was supported by NSF grant DMS-1453893.}

\makeatletter
\@namedef{subjclassname@2020}{%
  \textup{2020} Mathematics Subject Classification}
\makeatother

\subjclass[2020]{%
  18M05, %   	Monoidal categories, symmetric monoidal categories  
  05E10. %   	Combinatorial aspects of representation theory
  }

\date{February 13, 2021}

\begin{document}

\begin{abstract} 
We develop the idea of a supersymmetric monoidal supercategory, following ideas of Kapranov. Roughly, this is a monoidal category in which the objects and morphisms are $\bZ/2$-graded, equipped with isomorphisms $X \otimes Y \to Y \otimes X$ of parity $\vert X \vert \vert Y \vert$ on homogeneous objects. There are two fundamental examples: the groupoid of spin-sets, and the category of queer vector spaces equipped with the half tensor product; other important examples can be derived from these (such as the category of linear spin species). There are also two general constructions. The first is the exterior algebra of a supercategory (due to Ganter--Kapranov). The second is a construction we introduce called Clifford eversion. This defines an equivalence between a certain 2-category of supersymmetric monoidal supercategories and a corresponding 2-category of symmetric monoidal supercategories. We use our theory to better understand some aspects of the queer superalgebra, such as certain factors of $\sqrt{2}$ in the theory of Q-symmetric functions and Schur--Sergeev duality.
\end{abstract}

\maketitle
\tableofcontents

\section{Introduction}

The categorical analog of a commutative algebra is a symmetric (or braided) linear monoidal category. These categories have played a fundamental role in category theory and representation theory, among other places. The purpose of this paper is to develop a categorical analog of supercommutative superalgebras, following ideas of Kapranov \cite[\S 3.4]{supersym}; we call these supersymmetric (or superbraided) monoidal supercategories. In the rest of this introduction, we explain the basic idea behind the definition, some of our results and applications, and compare our contributions to the existing discussions of supercategories in the literature.

\subsection{The definition}

Recall that a supercommutative superalgebra is a $\bZ/2$-graded algebra such that $xy=(-1)^{\vert x \vert \vert y \vert} yx$ holds for all homogeneous elements $x$ and $y$. Here we write $\vert x \vert \in \bZ/2$ for the degree of a homogeneous element $x$. Thus we expect that in a superbraided monoidal category, there should be even and odd homogeneous objects, and there should be an isomorphism between $X \otimes Y$ and $(-1)^{\vert X \vert \vert Y \vert}$ ``times'' $Y \otimes X$ for homogeneous objects $X$ and $Y$.

The first puzzle here is what it means to negate an object. This has been well-understood for some time. Let $V$ be a finite dimensional super vector space. (Recall that a super vector space is simply a $\bZ/2$-graded vector space.) The super dimension of $V$ is defined to be $\sdim(V) = \dim(V_0)-\dim(V_1)$. Let $\Pi(V)$ be the super vector space given by $\Pi(V)_i=V_{i+1}$, i.e., $\Pi(V)$ is obtained from $V$ by interchanging the two graded pieces. Then $\sdim(\Pi(V))=-\sdim(V)$. Thus $\Pi(V)$ can be thought of as the negative of $V$.

The identity map $V \to \Pi(V)$ is an odd degree isomorphism. This map characterizes $\Pi(V)$ uniquely up to even isomorphism: if $W$ is any super vector space equipped with an odd degree isomorphism $V \to W$, then there is a unique even isomorphism $W \to \Pi(V)$ making the triangle commute. We can thus speak of $\Pi(V)$ for any object $V$ in a category in which the morphisms are $\bZ/2$-graded. Precisely, we define a \textbf{supercategory} to be a category enriched over super vector spaces, so that all $\Hom$ spaces are super vector spaces. Given an object $V$ in a supercategory, we define $\Pi(V)$ to be an object equipped with an odd degree isomorphism $V \to \Pi(V)$, if such an object and morphism exist. (We assume in what follows that they always do.) We can thus regard $\Pi(V)$ as the negative of $V$.

Following our initial intuition, suppose that we have a $\bZ/2$-supercategory (i.e., one in which the objects are $\bZ/2$-graded) equipped with a monoidal structure $\otimes$. Then we see that a superbraiding should be a natural isomorphism
\begin{displaymath}
\beta_{X,Y} \colon X \otimes Y \to \Pi^{\vert X \vert \vert Y \vert}(Y \otimes X),
\end{displaymath}
defined for all homogeneous objects $X$ and $Y$.

Of course, as with ordinary braidings, our superbraiding should satisfy some version of the hexagon axioms. It turns out that the usual hexagon axioms are inconsistent in this setting due to the signs that arise from composition in supercategories, but one can attempt to resolve the inconsistencies by allowing the hexagons to commute up to scalars. This ends up working if we assume the ground field has an eighth root of unity; however, the scalars that come up are not at all obvious. By working with $\bZ/4$-supercategories (or $\bZ/q$-supercategories for any $q$ divisible by~4), one can get by with much nicer systems of scalars that only involve signs (and thus make sense over general fields). In both cases, we introduce a notion of factor system (Definition~\ref{defn:factors}) to axiomatize the conditions on the scalars that appear in the hexagon axioms. One of the main contributions of this article is identifying these conditions and verifying them in important examples of interest.

We have thus arrived at the main definition: a superbraiding is a natural isomorphism $\beta$ as above such that the hexagon axioms hold up to scalars determined by a choice of factor system. See \S \ref{ss:type1} and \S \ref{ss:type2} for details.

\begin{figure}[!h]
\begin{tabular}{ccc}
\thickhline\\[-11pt]
$0$-category &\;& $1$-category \\[2pt]
\hline\\[-10pt]
Vector space && Supercategory \\
Commutative algebra && Symmetric monoidal supercategory \\
Super vector space && $\bZ/q$-supercategory \\
Commutative superalgebra && Symmetric monoidal $\bZ/q$-supercategory \\
Supercommutative superalgebra && Supersymmetric monoidal $\bZ/q$-supercategory \\[3pt]
\thickhline
\end{tabular}
\caption{Some analogies relevant to this paper. Here $q$ should be even.}
\end{figure}

\subsection{The 2-categorical perspective} \label{ss:intro-2cat}

One can define the notion of supercommutative superalgebra in elementary terms, as we did above. There is also a more abstract approach: one endows the monoidal category of super vector spaces with a new symmetry $\tau$ (namely $x \otimes y \mapsto (-1)^{\vert x \vert \vert y \vert} y \otimes x$), and realizes supercommutative superalgebras as the \emph{commutative} algebra objects in this category. The more abstract approach has several benefits: for example, any properties of commutative algebras in symmetric tensor categories can immediately be applied to supercommutative superalgebras, and it also enables one to easily consider similar structures, such as Lie superalgebras.

In our discussion of superbraided supercategories above, we have taken the elementary approach. We now explain a version of the abstract approach. Let $\sC$ be the $2$-category of $\bZ/q$-supercategories, with $q$ even. This $2$-category has a natural monoidal structure $\boxtimes$, given by a standard product construction for supercategories. Given two $\bZ/q$-supercategories $\cA$ and $\cB$, we define $\rT \colon \cA \boxtimes \cB \to \cB \boxtimes \cA$ to be the equivalence given by $X \boxtimes Y \mapsto \Pi^{\vert X \vert \vert Y \vert}(Y) \boxtimes X$, where $X$ and $Y$ are homogeneous objects of $\cA$ and $\cB$. One should think of the equivalence $\rT$ as analogous to the isomorphism $\tau$ from the previous paragraph. We show that $\rT$ (together with a choice of factor system) defines a symmetry on $\sC$ in \S\ref{s:2cat}. One can then regard supersymmetric monoidal supercategories as the \emph{symmetric} monoidal objects in the symmetric monoidal 2-category $(\sC, \rT)$.

\subsection{Spin-symmetric groups} 

The most fundamental example of a symmetric monoidal category is the groupoid of finite sets under disjoint union. A skeletal version $\cS$ of this category can be constructed explicitly, as follows. There is one object $[n]$ for each natural number $n$. We put $\End_{\cS}([n])=\fS_n$ and $\Hom_{\cS}([n],[m])=\emptyset$ for $n \ne m$. The monoidal operation $\amalg$ is defined on objects by $[n] \amalg [m]=[n+m]$ and on morphisms using the natural homomorphism $\fS_n \times \fS_m \to \fS_{n+m}$. Finally, the symmetry isomorphism $[n] \amalg [m] \to [m] \amalg [n]$ is given by the element $\tau_{n,m} \in \fS_{n+m}$ defined by
\begin{displaymath}
\tau_{n,m}(i) = \begin{cases}
i+m & \text{if $1 \le i \le n$} \\
i-n & \text{if $n+1 \le i \le n+m$} \end{cases}.
\end{displaymath}
The hexagon axioms amount to some identities satisfied by the $\tau_{n,m}$ elements.

The above picture has an analog in the setting of supersymmetric monoidal categories in which the symmetric groups are replaced by spin-symmetric groups. Recall that the $n$th spin-symmetric group $\wt{\fS}_n$ is a certain central extension of $\fS_n$ by $\bZ/2$; see \S \ref{ss:spin}. We define a skeletal category $\tilde{\cS}$ as follows. There is again one object $[n]$ for each natural number $n$. We put $\End_{\tilde{\cS}}([n])=\bk[\tilde{\fS}_n]/(c+1)$, where $c \in \wt{\fS}_n$ is the distinguished central element, and define $\Hom_{\tilde{\cS}}([n],[m])=0$ if $n \ne m$. The monoidal operation $\amalg$ on $\tilde{\cS}$ is defined on objects by $[n] \amalg [m]=[n+m]$, and on morphisms using the natural map $\wt{\fS}_n \times \wt{\fS}_m \to \wt{\fS}_{n+m}$. Finally, the supersymmetry is induced by a particular choice of lift $\tilde{\tau}_{n,m} \in \wt{\fS}_{n+m}$ of the element $\tau_{n,m}$ defined in the previous paragraph. In \S \ref{ss:catS}, we verify that the axioms of a supersymmetric monoidal category hold in this case; this amounts to establishing some identities satisfied by the $\tilde{\tau}_{n,m}$ elements.

The definition and study of the elements $\tilde{\tau}_{n,m}$ (in \S \ref{ss:calculations}), while totally elementary, is one of the most important parts of this paper. For a particular choice of $n$ and $m$, there are two lifts of $\tau_{n,m}$ to $\wt{\fS}_{n+m}$, and neither is preferred. We find, however, that the lifts $\tilde{\tau}_{n,m}$ we define have important compatibility properties as $n$ and $m$ vary; in other words, there is a preferred \emph{system} of lifts (up to the action of a single $\bZ/2$). This strikes us as an important discovery. In fact, as discussed in \S \ref{ss:intro-evert} below, we show that essentially every supersymmetry can, in a sense, be constructed from these elements.

The category $\cS$ is not merely the simplest example of a symmetric monoidal category: it is also universal. Indeed, suppose that $\cC$ is a symmetric monoidal category. Given an object $X$ of $\cC$, the symmetric group $\fS_n$ naturally acts on $X^{\otimes n}$. This construction induces a symmetric monoidal functor $\cS \to \cC$ via $[n] \mapsto X^{\otimes n}$. This observation shows that $(\cS, [1])$ is the universal symmetric monoidal category equipped with an object, in the sense that giving a symmetric monoidal functor $\cS \to \cC$ is equivalent to giving an object of $\cC$.

An analogous result holds for $\tilde{\cS}$. Indeed, if $\cC$ is a supersymmetric monoidal $\bZ$-category and $X$ is a degree~1 object in $\cC$ then we show that the spin-symmetric group $\wt{\fS}_n$ naturally acts on $X^{\otimes n}$. This leads to the following theorem:

\begin{theorem}
$(\tilde{\cS},[1])$ is the universal supersymmetric monoidal $\bZ$-supercategory equipped with an object of degree $1$.
\end{theorem}

Another important example of a supersymmetric category comes from the representation theory of spin-symmetric groups. A {\bf linear spin-species} is a sequence $(V_n)_{n \ge 0}$ where $V_n$ is a representation of $\wt{\fS}_n$ (that interacts with super structures appropriately). The category of such sequences, denoted $\Rep(\wt{\fS}_{\ast})$, admits a monoidal structure defined by inducing from Young subgroups. The $\tilde{\tau}_{n,m}$ elements can be used to construct a supersymmetry on this category. The details are carried out in \S \ref{s:spinrep}.

\subsection{The half tensor product}

The second most fundamental example of a symmetric monoidal category is the category $\Vec$ of vector spaces under tensor product. The category $\SVec$ of super vector spaces is \emph{not} supersymmetric: it is a ``plain'' symmetric monoidal supercategory since the isomorphism $V \otimes W \to W \otimes V$ has even degree. (The same applies to any category where the tensor product and symmetry are inherited from $\SVec$, such as the category of representations of a Lie superalgebra.) However, there is a supersymmetric analog to $\Vec$, that we now discuss.

Let $V$ be a super vector space. A {\bf queer structure} on $V$ is an odd degree automorphism $\nu \colon V \to V$ that squares to the identity. Suppose that $U$ and $V$ are super vector spaces equipped with queer structures $\mu$ and $\nu$. Then $\mu \otimes \nu$ is an even degree automorphism of $U \otimes V$ that squares to $-1$. The {\bf half tensor product} of $U$ and $V$, denoted $2^{-1}(U \otimes V)$, is defined to be the $\zeta_4$-eigenspace of $\mu \otimes \nu$. (Here $\zeta_4=\sqrt{-1}$.) The choice of $\zeta_4$ does not matter much, as $\mu \otimes 1$ (or $1 \otimes \nu$) interchanges the $+\zeta_4$ and $-\zeta_4$ eigenspaces.

This construction appears frequently in the representation theory of superalgebras. Suppose that $A$ is a $\bC$-superalgebra and $M$ is a finite dimensional simple $A$-module. The super version of Schur's lemma states that $\End_A(M)$ is either 1-dimensional (``type M'') or 2-dimensional (``type Q''); in the latter case, the algebra $\End_A(M)$ is generated over $\bC$ by a queer structure. Now suppose that $B$ is a second $\bC$-superalgebra and $N$ is a finite dimensional simple $B$-module. If either $M$ or $N$ has type M then $M \otimes N$ is a simple $A \otimes B$-module. If both $M$ and $N$ have type Q then $2^{-1}(M \otimes N)$ is a simple module, and $M\otimes N \cong 2^{-1}(M \otimes N) \otimes \bC^{1|1}$. (See \cite[\S 3.1.3]{chengwang} for details.)

Returning to the general situation, let $(U,\mu)$ and $(V,\nu)$ be as above. A simple computation shows that the natural symmetry $\tau_{U,V} \colon U \otimes V \to V \otimes U$ carries the $\zeta_4$-eigenspace of $\mu \otimes \nu$ into the $(-\zeta_4)$-eigenspace of $\nu \otimes \mu$. In particular, $\tau_{U,V}$ does \emph{not} map $2^{-1}(U \otimes V)$ into $2^{-1}(V \otimes U)$. However, the composition $(\nu \otimes 1) \circ \tau_{U,V}$ does map $2^{-1}(U \otimes V)$ to $2^{-1}(V \otimes U)$. This is an odd degree isomorphism, which suggests a lurking supersymmetry.

In \S \ref{s:queer}, we define a $\bZ/2$-supercategory $\Queer$ by taking the degree~0 objects to be super vector spaces and the degree~1 objects to be super vector spaces equipped with a queer structure. We define a monoidal operation $\odot$ on this category by using the half tensor product on two degree~1 objects and the ordinary tensor product on other pairs of objects. We show that this does indeed define a monoidal structure; this is not obvious, as the associator maps are non-trivial. We further show that the construction in the previous paragraph (scaled by an appropriate root of unity) defines a supersymmetry on this monoidal structure. Thus $\Queer$ is an example of a supersymmetric monoidal $\bZ/2$-supercategory. We believe this places the half tensor product construction in its proper theoretical framework.

\subsection{Exterior algebras}

Let $\cA$ be a $\bC$-linear category. One can then define the $n$th symmetric power of $\cA$, denoted $\Sym^n(\cA)$, to be the category of $\fS_n$-equivariant objects in the $n$th power of $\cA$ (for an appropriate notion of power). The sum of these categories, denoted $\Sym(\cA)$, is naturally a symmetric monoidal category, and is the universal such category to which $\cA$ maps.

Now suppose that $\cA$ is a supercategory. Ganter--Kapranov \cite{ganter} defined the $n$th exterior power of $\cA$, denoted $\lw^n(\cA)$, to be the category of appropriately $\wt{\fS}_n$-equivariant objects in the $n$th power of $\cA$. We show that the sum of these categories, denoted $\lw(\cA)$, is naturally a supersymmetric monoidal supercategory, and has a universal property similar to the symmetric power. This construction provides a large source of supersymmetric categories. See \S \ref{ss:exterior} for details.

\subsection{Clifford eversion} \label{ss:intro-evert}

Let $\cA$ be a nice $\bZ/q$-supercategory, with $q$ even; ``nice'' means that $\cA$ is additive, Karoubian, and admits a $\Pi$ structure. In \S \ref{s:clifford}, we define a new nice $\bZ/q$-supercategory $\Cl(\cA)$ called the \emph{Clifford eversion} of $\cA$ by considering certain kinds of Clifford modules in $\cA$. (This construction is essentially a generalization of the category $\Queer$ discussed above, see \S \ref{ss:qcom}.) If $\cA$ has a symmetric monoidal structure then we show that $\Cl(\cA)$ naturally has a supersymmetric monoidal structure; this construction again hinges on the $\tilde{\tau}_{n,m}$ elements. Moreover, we prove:

\begin{theorem}
Clifford eversion defines an equivalence between the $2$-category of nice symmetric monoidal $\bZ/q$-supercategories and the $2$-category of nice supersymmetric monoidal $\bZ/q$-supercategories.
\end{theorem}

Let $\sC$ be the 2-category of nice $\bZ/q$-supercategories. This category has a natural monoidal structure that admits an obvious symmetry $\Sigma$; it also admits the less obvious symmetry $\rT$ discussed in \S\ref{ss:intro-2cat}. In fact, Clifford eversion induces an equivalence of symmetric monoidal 2-categories $(\sC, \Sigma) \to (\sC, \rT)$. This recovers the above theorem, but is a more fundamental statement. However, we only partially prove this statement since the notion of equivalence of symmetric monoidal 2-categories is rather complicated.

Let $\cA$ be nice symmetric monoidal $\bZ/q$-supercategory with an exact structure, and let $\rK_+(\cA)$ be the quotient of the Grothendieck group by the relations $[\Pi(X)]=[X]$. We show (Proposition~\ref{prop:K-ring}) that there is a natural isomorphism of rings
\begin{displaymath}
\rK_+(\cA)[1/\sqrt{2}] \cong \rK_+(\Cl(\cA))[1/\sqrt{2}].
\end{displaymath}
This helps explain the appearance of certain factors of $\sqrt{2}$ in the theory of Q-symmetric functions.

\subsection{Sergeev--Schur duality}

(We take the base field to be $\bC$ in this section.) Classical Schur--Weyl duality connects the representation theory of the general linear group and the symmetric groups. Let $\Rep^{\pol}(\GL)$ be the category of polynomial representations of the infinite general linear group; this is the tensor category generated by the standard representation $\bV$. The symmetric group $\fS_n$ acts on $\bV^{\otimes n}$ by permuting tensor factors, and so we obtain a functor $\Rep(\fS_n) \to \Rep^{\pol}(\GL)$ by $M \mapsto M \otimes_{\bC[\fS_n]} \bV^{\otimes n}$. Let $\Rep(\fS_{\ast})$ be the category of sequences $(V_n)_{n \ge 0}$ where $V_n$ is a representation of the symmetric group $\fS_n$; such sequences are sometimes called \emph{linear species}. Then one formulation of Schur--Weyl duality states that the above functor extends to an equivalence $\Rep(\fS_{\ast}) \to \Rep^{\pol}(\GL)$. This equivalence respects various additional structures on the two categories. In particular, it is an equivalence of symmetric monoidal categories, where $\Rep^{\pol}(\GL)$ is endowed with the usual tensor product and $\Rep(\fS_{\ast})$ is given the induction tensor product, defined by
\begin{displaymath}
(V \otimes W)_n = \bigoplus_{i+j=n} \Ind_{\fS_i \times \fS_j}^{\fS_n}(V_i \otimes W_j).
\end{displaymath}
This equivalence is extremely useful when studying algebra within the tensor category $\Rep(\fS_{\ast})$, as it provides a bridge to more classical theory.

Sergeev--Schur duality is an analog of Schur--Weyl duality in which the general linear group is replaced by the queer Lie superalgebra. Precisely, let $\bV=\bC^{\infty|\infty}$ and fix an odd-degree isomorphism $\alpha \colon \bV \to \bV$ squaring to~2. The infinite queer Lie superalgebra $\fq$ is defined to be the subalgebra of $\fgl(\bV)$ that preserves $\alpha$. The symmetric group $\fS_n$ acts on $\bV^{\otimes n}$ by permuting tensor factors. However, there are additional natural endomorphisms that commute with $\fq$: namely, we can apply $\alpha$ to any tensor factor. To make this more precise, let $\Cl_n$ denote the $n$th Clifford algebra; this has $n$ anti-commuting generators that each square to~2. Define the Hecke--Clifford algebra $\cH_n$ to be the semi-direct product $\fS_n \ltimes \Cl_n$. Then $\cH_n$ acts on $\bV^{\otimes n}$ and commutes with $\fq$. As in the previous paragraph, we obtain an induced functor $\Rep(\cH_*) \to \Rep^{\pol}(\fq)$. One formulation of Sergeev--Schur duality is that this is an equivalence (and respects the natural symmetric monoidal structures).

There is a close relationship between the algebra $\cH_n$ and the spin-symmetric groups: we have an isomorphism $\cH_n \cong \bk[\wt{\fS}_n]/(c+1) \otimes \Cl_n$. Thus an $\cH_n$-module can be viewed as a Clifford module in the category $\Rep(\wt{\fS}_n)$ of linear spin species. We thus see that $\Rep(\cH_*)$ can be identified with the Clifford eversion of the category $\Rep(\wt{\fS}_{\ast})$ of sequences of $\wt{\fS}_n$-representations. We can thus rephrase the above discussion as follows:

\begin{theorem}
The category $\Rep^{\pol}(\fq)$ is equivalent to the Clifford eversion of the category $\Rep(\wt{\fS}_{\ast})$ of linear spin species.
\end{theorem}

We show that this equivalence respects natural symmetric monoidal structures on the two categories. We view the above theorem as an elegant expression of Sergeev--Schur duality, as it provides a direct link between the queer algebra and the spin-symmetric groups. See \S \ref{s:sergeev} for details.

The Grothendieck ring of $\Rep^{\pol}(\fq)$ is naturally identified with the ring $\Gamma$ of odd symmetric functions. This ring has a basis consisting of the Schur Q-functions $Q_\lambda$. However, they do not represent classes of simple objects in general: one must adjust them by appropriate powers of $\sqrt{2}$ (which depend on $\lambda$ as well as which category is being used). One of our goals was to better understand the nature of these powers of $\sqrt{2}$, and the effect of Clifford eversion on Grothendieck rings is one piece of that puzzle. Unfortunately, there is one more piece which is not explained by our general framework: some simple modules of $\bk[\wt{\fS}_n]/(c+1)$ carry an action of $\Cl_1$, while others do not, and this also features into the powers of $\sqrt{2}$ that appear.

\subsection{Relation to previous work}

Kapranov proposed a sketch for the definition of a supersymmetric monoidal supercategory in \cite[\S 3.4]{supersym}. A large portion of this article is devoted to providing the details for this sketch.

Versions of (monoidal) supercategories have been considered in \cite{brundan}, \cite{KKT}, \cite{ellis}. However, these sources do not treat the notion of (super)braiding or (super)symmetry for a monoidal supercategory, which as far as we know, is treated in detail for the first time here. The Clifford twist in \cite{KKT} is similar to our notion of Clifford eversion, but is not the same (for instance, the Clifford twist is an involution even if the coefficient field does not contain $\sqrt{-1}$).

\subsection{Open problems}

We mention a few open problems:
\begin{enumerate}
\item In Remark~\ref{rmk:factor}, we observe a connection between even factor systems and degree 3 group cohomology. However, this connection is not complete, and it does not apply to odd factor systems (which is the relevant case for supersymmetries). It would be helpful to have a better understanding of these issues.
\item Is there a conceptual a priori explanation for why the 2-categories of symmetric and supersymmetric monoidal supercategories are equivalent?
\item Is there a ``natural'' construction of the groupoid of spin-sets? That is, is there a natural category where the objects are some kind of ``enhanced'' finite sets whose automorphism groups are spin-symmetric groups? This might clarify a number of our constructions (such as the $\tilde{\tau}_{n,m}$ elements). Is there a non-groupoid version?
\item We have provided explanations for most of the proof of Theorem~\ref{thm:evert-2cat}, but it remains incomplete.
\end{enumerate}

\subsection{Notation}

We fix a field $\bk$ throughout which we assume does not have characteristic~2. One could allow $\bk$ to be a commutative $\bZ[1/2]$-algebra, but we stick to fields for simplicity. We fix a compatible system of primitive 2-power roots of unity $\zeta_{2^k}$ in $\ol{\bk}$; by compatible, we mean that $\zeta_{2^{k+1}}^2 = \zeta_{2^k}$.

In this article, ``superalgebra'' means an algebra with a $\bZ/2$-grading and a ``commutative superalgebra'' refers to a commutative algebra with a $\bZ/2$-grading (which does not affect the definition of commutative). This is sometimes used in the literature to mean a skew-commutative algebra; we reserve the phrase ``supercommutative superalgebra'' for this, i.e., an algebra $A$ with a $\bZ/2$-grading such that $xy = (-1)^{|x| |y|} yx$ for homogeneous elements $x,y \in A$.

\section{Supercategories} \label{s:supercat}

\subsection{Super vector spaces} \label{ss:super-vector-spaces}

A {\bf super vector space} over $\bk$ is a $\bZ/2$-graded $\bk$-vector space. Given a super vector space $V$, we write $V_i$ for the subspace of homogeneous elements of degree $i$. For $v \in V$, we write $\vert v \vert=i$ to indicate that $v$ is a homogeneous element of degree $i$. For $n \in \bZ$, we let $V[n]$ be the super vector space with $V[n]_i=V_{n+i}$.

Let $V$ and $W$ be two super vector spaces and let $f \colon V \to W$ be a linear map. We say that $f$ is {\bf homogeneous} of parity $n$ if $f(V_i) \subseteq W_{i+n}$. We write $\Hom(V,W)$ for the set of all linear maps, and $\Hom(V,W)_i$ for those of parity $i$. Every linear map decomposes uniquely as a sum of even and odd maps, and so we have a decomposition
\begin{displaymath}
\Hom(V,W) = \Hom(V,W)_0 \oplus \Hom(V,W)_1
\end{displaymath}
In this way, we regard $\Hom(V,W)$ itself as a super vector space. We let $\SVec$ be the category of super vector spaces and all linear maps, and $\SVec^{\circ}$ the category of super vector spaces and even linear maps.

The tensor product of two super vector spaces is just the tensor product of the underlying vector spaces, graded in the usual manner. Given homogeneous linear maps $f_1 \colon V_1 \to W_1$ and $f_2 \colon V_2 \to W_2$, we define $f_1 \otimes f_2 \colon V_1 \otimes V_2 \to W_1 \otimes W_2$ to be the linear map given on homogeneous elements by the formula
\begin{displaymath}
(f_1 \otimes f_2)(v_1 \otimes v_2) = (-1)^{\vert v_1 \vert \vert f_2 \vert} f_1(v_1) \otimes f_2(v_2).
\end{displaymath}
We caution that $\otimes$ is \emph{not} a functor $\SVec \times \SVec \to \SVec$; this will be explained in detail in \S \ref{ss:svec-mon}. However, $\otimes$ is a functor $\SVec^{\circ} \times \SVec^{\circ} \to \SVec^{\circ}$, and naturally endows $\SVec^{\circ}$ with the structure of a monoidal category.

Let $V$ and $W$ be super vector spaces. We define isomorphisms $\sigma, \tau \colon V \otimes W \to W \otimes V$ for homogeneous elements $v \in V$ and $w \in W$ as follows:
\begin{displaymath}
\sigma(v \otimes w)=w \otimes v, \qquad
\tau(v \otimes w)=(-1)^{\vert v \vert \vert w \vert} w \otimes v.
\end{displaymath}
Both $\sigma$ and $\tau$ define symmetric structures on $\otimes$. However, the symmetric structure $\sigma$ is uninteresting in that it does not interact with the grading. In what follows, we regard $\SVec^{\circ}$ as a symmetric monoidal category using $\tau$, unless explicitly said otherwise.

\subsection{Supercategories} \label{ss:scat}

A {\bf supercategory} is a category $\cA$ enriched in the category $\SVec^{\circ}$ of super vector spaces. Thus $\cA$ is a category and $\Hom_{\cA}(X,Y)$ has the structure of a super vector space for all objects $X$ and $Y$ such that composition defines an even map
\begin{displaymath}
\Hom_{\cA}(Y,Z) \otimes \Hom_{\cA}(X,Y) \to \Hom_{\cA}(X,Z), \qquad
f \otimes g \mapsto f \circ g.
\end{displaymath}
Given a supercategory $\cA$, we let $\cA^{\circ}$ be the category with the same collection of objects but only the even morphisms; that is, $\Hom_{\cA^{\circ}}(X,Y)=\Hom_{\cA}(X,Y)_0$.

Let $\Lambda$ be an abelian group (or, for now, any set). A {\bf homogeneous $\Lambda$-supercategory} is a collection $\cA=\{\cA_i\}_{i \in \Lambda}$, where each $\cA_i$ is a supercategory. We refer to $\cA_i$ as the category of homogeneous objects of degree $i$. Informally, an inhomogeneous $\Lambda$-supercategory is a supercategory $\cA$ equipped with subcategories $\cA_i$ for $i \in \cA$ such that $\cA$ is a direct sum of the $\cA_i$'s, in some sense. This concept will appear in some examples, but we do not give a rigorous definition to avoid details of what direct sum means here (there are some subtleties if, e.g., $\cA$ is non-additive or if $\Lambda$ is infinite). In the inhomogeneous case, any construction will be determined by what it does on homogeneous objects, so we can typically restrict our attention to them. In what follows, ``$\Lambda$-supercategory'' will by default refer to the homogeneous version, and we almost always take $\Lambda$ to be $\bZ$ or $\bZ/q$.

Suppose that $\cA$ is a $\Lambda$-supercategory and we have a function $f \colon \Lambda \to \Lambda'$. Then we can convert $\cA$ into a $\Lambda'$-supercategory by defining $\cA_i = \coprod_{j \in f^{-1}(i)} \cA_j$ for $i \in \Lambda'$. Here the coproduct is taken in the sense of linear categories: the $\Hom$ space between objects in different factors is~0. In particular, we can convert a $\bZ$-supercategory into a $\bZ/q$-supercategory, or a $\bZ/q$-supercategory into a $\bZ/q'$-supercategory if $q'$ divides $q$.

\begin{example}
Let $G$ be a supergroup over $\bC$ containing a central $\bG_m$. Then $\cA=\Rep(G)$ is an (inhomogeneous) $\bZ$-supercategory, by taking $\cA_i$ to be the subcategory of objects where $\bG_m$ acts with weight $i$.
\end{example}

\subsection{Superfunctors}

Let $\cA$ and $\cB$ be supercategories. A {\bf superfunctor} $F \colon \cA \to \cB$ is a functor (in the usual sense) such that the maps
\begin{displaymath}
F \colon \Hom_{\cA}(X,Y) \to \Hom_{\cB}(F(X), F(Y))
\end{displaymath}
are even maps of super vector spaces. If $\cA$ and $\cB$ are $\Lambda$-supercategories then a {\bf $\Lambda$-superfunctor} is a collection of superfunctors $F=\{F_i \colon \cA_i \to \cB_i \}_{i \in \Lambda}$.

Let $F$ and $G$ be two superfunctors from $\cA$ to $\cB$. A {\bf homogeneous natural transformation} $\eta \colon F \to G$ of parity $p \in \bZ/2$ is a rule that assigns to every object $X$ of $\cA$ a homogeneous morphism $\eta_X \colon F(X) \to G(X)$ of degree $p$ in $\cB$ such that if $f \colon X \to Y$ is a homogeneous morphism in $\cA$ then $G(f) \circ \eta_X = (-1)^{p \vert f \vert} \eta_Y \circ F(f)$. A general natural transformation is a sum of an even and an odd natural transformation. The definition is the same for $\Lambda$-superfunctors: the grading does not affect the definition at all.

\subsection{$\Pi$-structures} \label{ss:pi}

Let $\cA$ be a supercategory. A {\bf $\Pi$-structure} on an object $X$ of $\cA$ is an odd degree isomorphism $\xi_X \colon X \to \Pi(X)$, for some object $\Pi(X)$. A {\bf $\Pi$-structure} on $\cA$ is a choice of $\Pi$-structure on each object of $\cA$. Such a structure need not exist. Note that if $\cA$ is $\Lambda$-graded and $X$ is homogeneous then $\Pi(X)$ is necessarily homogeneous of the same degree.

Suppose we have a $\Pi$-structure on $\cA$. For a homogeneous morphism $f \colon X \to Y$ in $\cA$, define $\Pi(f)$ to be the unique morphism making the following diagram commute up to $(-1)^{\vert f \vert}$
\begin{displaymath}
\xymatrix@C=4em{
X \ar[r]^f \ar[d]_{\xi_X} & Y \ar[d]^{\xi_Y} \\
\Pi(X) \ar@{..>}[r]^{\Pi(f)} & \Pi(Y) }.
\end{displaymath}
By uniqueness, it is clear that $\Pi$ is a superfunctor $\cA \to \cA$, and that $\xi \colon \id_{\cA} \to \Pi$ is an odd natural isomorphism. (If $\cA$ is a $\Lambda$-supercategory then $\Pi$ is a $\Lambda$-superfunctor.) Furthermore, taking $f=\xi_X$, we find $\Pi(\xi_X)=-\xi_{\Pi(X)}$.

For a non-negative integer $n$, we let $\Pi^n$ be the $n$-fold iterate of $\Pi$. For $n < m$, we define $\xi^{n,m}_X \colon \Pi^n(X) \to \Pi^m(X)$ to be the composition $\xi_{\Pi^{m-1}(X)} \circ \cdots \circ \xi_{\Pi^{n+1}(X)} \circ \xi_{\Pi^n(X)}$; for $n > m$, we define $\xi^{n,m}_X=(\xi^{m,n}_X)^{-1}$; and for $n=m$, we let $\xi^{n,m}_X$ be the identity. The map $\xi^{n,m}_X$ is an isomorphism of parity $m-n$. In particular, we see that $\xi^{0,2} \colon \id_{\cA} \to \Pi^2$ is an even degree isomorphism. For $q$ even, we define $\Pi^n$ for $n \in \bZ/q$ to be the colimit of the system
\begin{displaymath}
\xymatrix@C=4em{
\cdots \ar[r] & \Pi^{n-q} \ar[r]^{\xi^{n-q,n}} & \Pi^n \ar[r]^{\xi^{n+q,n}} \ar[r] & \Pi^{n+q} \ar[r] & \cdots. }
\end{displaymath}
This colimit exists since all the maps are even isomorphisms.

Suppose that $(\Pi, \xi)$ and $(\Pi', \xi')$ are two $\Pi$-structures on $\cA$. Define $i_X \colon \Pi(X) \to \Pi'(X)$ to be the morphism $\xi'_X \circ \xi_X^{-1}$. Then $i_X$ is an even degree isomorphism, and defines an even degree isomorphism of functors $i \colon \Pi \to \Pi'$ compatible with $\xi$ and $\xi'$. We thus see that any two $\Pi$-structures are canonically (and, in fact, uniquely) isomorphic by an even degree natural transformation.

\begin{proposition} \label{prop:Pi-functor}
Let $F \colon \cA \to \cB$ be a superfunctor between supercategories with $\Pi$-structures. Then there is a canonical even isomorphism $F \circ \Pi = \Pi \circ F$.
\end{proposition}

\begin{proof}
  Let $X$ be an object of $\cA$. Then $\xi_{F(X)} F(\xi_X)^{-1}$ gives an even-degree isomorphism $F(\Pi(X)) \cong \Pi(F(X))$. We now check that this is a natural isomorphism. Let $f \colon X \to Y$ be a homogeneous morphism. Then
  \[
    \xi_Y f = (-1)^{|f|} \Pi(f) \xi_X, \qquad \xi_{F(Y)} F(f) = (-1)^{|f|} \Pi(F(f)) \xi_{F(X)}.
  \]
  We apply $F$ to the first equation and rearrange to get $F(f) F(\xi_X)^{-1} = (-1)^{|f|} F(\xi_Y)^{-1} F(\Pi(f))$. Post-compose this equation with $\xi_{F(Y)}$ and use the second equation to get
  \[
  \xi_{F(Y)} F(\xi_Y)^{-1} F(\Pi(f)) = \Pi(F(f)) \xi_{F(X)} F(\xi_X)^{-1},
\]
which shows that we have defined a natural isomorphism $F \circ \Pi \cong \Pi \circ F$.
\end{proof}

\subsection{Product categories} \label{ss:scat-prod}

Let $\cA$ and $\cB$ be two supercategories. We let $\cA \veeotimes \cB$ be the product category, in the sense of enriched categories. Precisely, the objects of $\cA \veeotimes \cB$ are pairs $(A,B)$ with $A \in \cA$ and $B \in \cB$; we write $A \boxtimes B$ in place of $(A,B)$. The $\Hom$ spaces are defined by
\begin{displaymath}
\Hom_{\cA \veeotimes \cB}(A \boxtimes B, A' \boxtimes B') = \Hom_{\cA}(A,A') \otimes \Hom_{\cB}(B,B').
\end{displaymath}
Given homogeneous morphisms $f \otimes g \colon A \boxtimes B \to A' \boxtimes B'$ and $f' \otimes g' \colon A' \boxtimes B' \to A'' \boxtimes B''$, their composition is defined by
\begin{displaymath}
(f' \otimes g') \circ (f \otimes g) = (-1)^{\vert f \vert \vert g' \vert} (f' \circ f) \otimes (g' \circ g).
\end{displaymath}
The sign here comes from the symmetry of the tensor product in $\SVec$. (To define the product of enriched categories in general, the enriching category needs a symmetric monoidal structure, and here we are using the symmetry $\tau$ on $\SVec^{\circ}$.)

Let $\cA$ and $\cB$ be two homogeneous $\Lambda$-supercategories. We define $\cA \boxtimes \cB$ to be the homogeneous $\Lambda$-supercategory with
\begin{displaymath}
(\cA \boxtimes \cB)_i=\coprod_{j+k=i} \cA_j \veeotimes \cB_k.
\end{displaymath}
We note that $\veeotimes$ does not distribute over direct sums of supercategories (even finite direct sums), and so does not give the expected result on inhomogeneous supercategories; this is why we have used a different symbol in the $\Lambda$-graded case.

If either $\cA$ or $\cB$ admits a $\Pi$-structure then so does $\cA \boxtimes \cB$. Indeed, suppose $\cA$ does, and let $A \boxtimes B$ be a given object of $\cA \boxtimes \cB$. If $\xi \colon A \to \Pi(A)$ is an odd degree isomorphism then $\xi \otimes \id_B \colon A \boxtimes B \to \Pi(A) \boxtimes B$ is an odd degree isomorphism, and so the claim is verified.

\subsection{Grothendieck groups}

Let $\cA$ be a supercategory equipped with an exact structure on $\cA^{\circ}$. We can then consider the Grothendieck group $\rK(\cA^{\circ})$ of $\cA^{\circ}$. For $\eps \in \{+,-\}$, we define $\rK_{\eps}(\cA)$ to be the quotient of $\rK(\cA^{\circ})$ by the relations $[X]=\eps [Y]$ whenever there is an odd degree isomorphism $X \to Y$. If $\cA$ admits a $\Pi$-structure then $\rK_{\eps}(\cA)$ is simply the quotient of $\rK(\cA^{\circ})$ by the relations $[\Pi(X)]=\eps [X]$. If $\cA$ admits a $\Lambda$-grading then $\rK(\cA^{\circ}) = \bigoplus_{i \in \Lambda} \rK(\cA_i^{\circ})$, and so $\rK(\cA^{\circ})$ is canonically $\Lambda$-grading. This induces a $\Lambda$-grading on $\rK_{\eps}(\cA)$.

\section{Monoidal structures} \label{s:monoidal}

\subsection{Monoidal structures} \label{ss:svec-mon}

Let $\cA$ be a $\Lambda$-supercategory. A {\bf monoidal structure} on $\cA$ consists of:
\begin{itemize}
\item a $\Lambda$-superfunctor $\otimes \colon \cA \boxtimes \cA \to \cA$;
\item a unit object $\bone$;
\item an even natural isomorphism $\alpha_{X,Y,Z} \colon X \otimes (Y \otimes Z) \to  (X \otimes Y) \otimes Z$ satisfying the pentagon axiom, i.e., the following diagram is commutative for all objects $A,B,C,D$:
  \[
    \xymatrix{
      & (A \otimes B) \otimes (C \otimes D) \ar[rd]^-{\alpha_{A \otimes B,C,D}} & \\
      A \otimes (B \otimes (C \otimes D)) \ar[d]_-{\id_A \otimes \alpha_{B,C,D}} \ar[ru]^-{\alpha_{A,B,C\otimes D}} & &
      ((A \otimes B) \otimes C) \otimes D \\
      A \otimes ((B \otimes C) \otimes D) \ar[rr]^-{\alpha_{A,B \otimes C,D}} &&
      (A \otimes (B \otimes C)) \otimes D \ar[u]_-{\alpha_{A,B,C} \otimes \id_D}
      }
  \]
\item even natural isomorphisms $\lambda_X \colon \bone \otimes X \to X$ and $\rho_X \colon X \otimes \bone \to X$ satisfying the triangle axiom, i.e., the following diagram is commutative for all objects $A,B$:
  \[
    \xymatrix{
      A \otimes (\bone \otimes B) \ar[rr]^-{\alpha_{A,\bone,B}} \ar[rd]_-{\id_A \otimes \lambda_B} & & (A \otimes \bone) \otimes B \ar[ld]^-{\rho_A \otimes \id_B} \\
      &  A \otimes B
    }
  \]
\end{itemize}
We note that if $A$ and $B$ are homogeneous objects of $\cA$ then $A \otimes B$ is homogeneous and $\vert A \otimes B \vert = \vert A \vert + \vert B \vert$. A {\bf monoidal $\Lambda$-supercategory} is a $\Lambda$-supercategory equipped with a monoidal structure.

\begin{example}
The supercategory $\SVec$ is monoidal under tensor product. (Note here that we are using all of $\SVec$, and not just the subcategory $\SVec^{\circ}$.)
\end{example}

\begin{remark}
  Suppose that $\cA$ is a monoidal supercategory and we have homogeneous morphisms
\begin{displaymath}
f_1 \colon A_1 \to A_2, \quad
f_2 \colon A_2 \to A_3, \quad
f_1' \colon A_1' \to A_2', \quad
f_2' \colon A_2' \to A_3'.
\end{displaymath}
Then the diagram
\begin{displaymath}
\xymatrix{
A_1 \otimes A_1' \ar[rr]^{f_2f_1 \otimes f_2'f_1'} \ar[rd]_{f_1 \otimes f_1'} && A_3 \otimes A_3' \\
& A_2 \otimes A_2' \ar[ru]_{f_2 \otimes f_2'} }
\end{displaymath}
commutes up to $(-1)^{\vert f_1 \vert \vert f_2' \vert}$, that is, we have the identity
\begin{displaymath}
(f_2 \otimes f_2') \circ (f_1 \otimes f_1') = (-1)^{\vert f_1 \vert \vert f_2' \vert} (f_2 f_1 \otimes f_2' f_1').
\end{displaymath}
This follows from the definition of the composition law in the category $\cA \boxtimes \cA$. From this, we see that $\otimes$ does \emph{not} define a functor $\cA \times \cA \to \cA$ (in general). However, it does define a functor $\cA^{\circ} \times \cA^{\circ} \to \cA^{\circ}$, and $\cA^{\circ}$ is naturally a monoidal category.
\end{remark}

Let $\cA$ be a monoidal $\Lambda$-supercategory admitting a $\Pi$-structure. Let $A$ and $B$ be objects of $\cA$. Then we have odd degree isomorphisms $\xi_{A \otimes B} \colon A \otimes B \to \Pi(A \otimes B)$ and $\xi_A \otimes \id_B \colon A \otimes B \to \Pi(A) \otimes B$. We thus have a canonical even degree isomorphism $\Pi(A \otimes B) \to \Pi(A) \otimes B$. Similarly, we have a canonical even degree isomorphism $\Pi(A \otimes B) \to A \otimes \Pi(B)$. We thus see that we have canonical even degree isomorphisms $\Pi(A) \cong \Pi(\bone) \otimes A \cong A \otimes \Pi(\bone)$.

\begin{definition} \label{defn:mon-func}
Let $\cA$ and $\cB$ be monoidal $\Lambda$-supercategories. A {\bf monoidal $\Lambda$-superfunctor} $F \colon \cA \to \cB$ is a $\Lambda$-superfunctor $F$ equipped with even isomorphisms $\Phi_{X,Y} \colon F(X) \otimes F(Y) \to F(X \otimes Y)$ and $\phi \colon F(\bone_{\cA}) \to \bone_{\cB}$ such that the following diagrams commute for all objects $X,Y,Z$ of $\cA$
\[
  \xymatrix@C=4em{
    (F(X) \otimes F(Y)) \otimes F(Z) \ar[r]^-{\Phi_{X,Y} \otimes 1} \ar[d]_-{\alpha_{F(X), F(Y), F(Z)}} & F(X \otimes Y) \otimes F(Z) \ar[r]^-{\Phi_{X \otimes Y,Z}} & F((X\otimes Y) \otimes Z) \ar[d]^-{F(\alpha_{X,Y,Z})} \\
    F(X) \otimes (F(Y) \otimes F(Z)) \ar[r]^-{1 \otimes \Phi_{Y,Z}} & F(X) \otimes F(Y \otimes Z) \ar[r]^-{\Phi_{X,Y\otimes Z}} & F(X \otimes (Y \otimes Z)) }
\]
\[
  \xymatrix{ \bone \otimes F(X) \ar[r]^-{\lambda_{F(X)}} \ar[d]_-{\phi \otimes 1} & F(X) \\ F(\bone) \otimes F(X) \ar[r]^-{\Phi_{\bone, X}} & F(\bone \otimes X) \ar[u]_-{F(\lambda_X)} } \qquad
    \xymatrix{ F(X) \otimes \bone \ar[r]^-{\rho_{F(X)}} \ar[d]_-{1 \otimes \phi} & F(X) \\ F(X) \otimes F(\bone) \ar[r]^-{\Phi_{X,\bone}} & F(X \otimes \bone) \ar[u]_-{F(\rho_X)} }
  \]
  We will generally suppress $\Phi$ and $\phi$ from the notation.
\end{definition}

\begin{definition} \label{defn:mon-trans}
Given two monoidal $\Lambda$-superfunctors $(F,\Phi,\phi)$ and $(G, \Gamma, \gamma)$ between $\Lambda$-supercategories $\cA$ and  $\cB$, a {\bf monoidal natural transformation} $\eta \colon F \to G$ is an even natural transformation $\eta$ such that the following two diagrams commute for all objects $X,Y$ of $\cA$:
\[
  \xymatrix@C=4em{F(X) \otimes F(Y) \ar[r]^-{\eta_X \otimes \eta_Y} \ar[d]_-{\Phi_{X,Y}} & G(X) \otimes G(Y) \ar[d]^-{\Gamma_{X,Y}} \\ F(X \otimes Y) \ar[r]^-{\eta_{X \otimes Y}} & G(X \otimes Y) }, \qquad
  \xymatrix@C=4em{ \bone \ar[r]^-\gamma \ar[d]_-\phi & G(\bone) \\ F(\bone) \ar[ur]_-{\eta_\bone}}.
\]
We thus have a 2-category of monoidal supercategories, and in particular, a notion of monoidal equivalence.
\end{definition}

\subsection{The equivalences $\Sigma$ and $\rT$} \label{ss:sigma-tau}

Let $\cA$ and $\cB$ be $\bZ/q$-graded supercategories, with $q$ even. Define an equivalence
\begin{displaymath}
\Sigma_{\cA,\cB} \colon \cA \boxtimes \cB \to \cB \boxtimes \cA
\end{displaymath}
as follows. On homogeneous objects, we put $\Sigma_{\cA,\cB}(A \boxtimes B)=B \boxtimes A$. Given homogeneous morphisms $f \colon A \to A'$ and $g \colon B \to B'$ between objects of $\cA$ and $\cB$, we define $\Sigma_{\cA,\cB}(f \otimes g)=(-1)^{\vert f \vert \vert g \vert} g \otimes f$.

Now suppose that $\cA$ and $\cB$ admit $\Pi$-structures. We define an equivalence
\begin{displaymath}
\rT_{\cA,\cB} \colon \cA \boxtimes \cB \to \cB \boxtimes \cA
\end{displaymath}
as follows. On homogeneous objects, we put $\rT_{\cA,\cB}(A \boxtimes B)=\Pi^{\vert A \vert \vert B \vert}(B) \boxtimes A$. Given homogeneous morphisms $f \colon A \to A'$ and $g \colon B \to B'$ of homogeneous objects, we put $\rT_{\cA,\cB}(f \otimes g) = (-1)^{\vert f \vert \vert g \vert} \Pi^{\vert A \vert \vert B \vert}(g) \otimes f$.

We regard $\Sigma$ and $\rT$ as analogous to the two symmetries $\sigma$ and $\tau$ we defined on $\SVec$ in \S \ref{ss:super-vector-spaces}. They will give rise to the notion of symmetric and supersymmetric monoidal supercategories, just as $\sigma$ and $\tau$ give rise to the notion of commutative and supercommutative superalgebras.

\subsection{Factor systems} \label{ss:factor}

We next want to define braidings and symmetries, and their super analogs. For this, we will require certain systems of scalars that we now introduce.

\begin{definition} \label{defn:factors}
\addtocounter{equation}{-1}
  \begin{subequations}
Let $q$ be an even integer and let $p \in \bZ/2$. A {\bf B-factor system} of parity $p$ is a pair $\omega=(\omega_1, \omega_2)$ of functions
\begin{displaymath}
\omega_1,\omega_2 \colon (\bZ/q)^3 \to \bk^{\times}
\end{displaymath}
satisfying the following conditions for all $a,b,c,d \in \bZ/q$ (in what follows, the semicolon in, e.g.,  $\omega_1(a; b,c)$, indicates that the $b$ and $c$ parameters play similar roles, while the $a$ parameter is somewhat different):
\begin{align}
\label{eq:condition1}
\omega_1(a; b,c) \omega_1(a; b+c,d) &= \omega_1(a; b,c+d) \omega_1(a; c,d) \\
\label{eq:condition2}
\omega_2(a+b,c; d) \omega_2(a,b; d) &= \omega_2(b,c; d) \omega_2(a,b+c; d) \\
\label{eq:condition3}
\omega_1(b; c,d) \omega_1(a; c,d) \omega_1(a+b; c,d)^{-1} &= (-1)^{abcd p} \omega_2(a,b; c) \omega_2(a,b; d) \omega_2(a,b; c+d)^{-1} \\
\label{eq:condition4}
\omega_1(a; b,c) \omega_1(a; c,b)^{-1} &= \omega_2(b,a; c) \omega_2(a,b; c)^{-1}.
\end{align}
An {\bf S-factor system} of parity $p$ is a triple $\omega=(\omega_1, \omega_2, \omega_{\sharp})$ where $(\omega_1, \omega_2)$ is a B-factor system of parity $p$ and $\omega_{\sharp} \colon (\bZ/q)^2 \to \bk^{\times}$ is a function satisfying the following conditions:
\begin{align}
\label{eq:condition5}
\omega_1(a; b,c) \omega_2(b,c; a) &= \omega_{\sharp}(a,b) \omega_{\sharp}(a,c) \omega_{\sharp}(a,b+c)^{-1} \\
\label{eq:condition6}
\omega_1(c; a,b) \omega_2(a,b; c) &=\omega_{\sharp}(a,c) \omega_{\sharp}(b,c) \omega_{\sharp}(a+b,c)^{-1}.
\end{align}
We say that $\omega$ is {\bf symmetric} if $\omega_{\sharp}(a,b)=\omega_{\sharp}(b,a)$.
\end{subequations}
\end{definition}

\begin{remark}
A few remarks:
\begin{enumerate}
\item The B- and S- prefixes stand for ``braiding'' and ``syllepsis'' (see \S \ref{s:2cat} for more).
\item The parity $p$ of the factor system only intervenes in \eqref{eq:condition3}.
\item Let $\Gamma_p$ be the set of all S-factor systems of parity $p$ for $\bZ/q$. Then $\Gamma=\Gamma_0 \oplus \Gamma_1$ is a $\bZ/2$-graded abelian group. \qedhere
\end{enumerate}
\end{remark}

\begin{example} \label{ex:trivial}
Let $\omega_1$, $\omega_2$, and $\omega_{\sharp}$ be the constant functions taking the value~1. Then $\omega$ is an even S-factor system. We refer to this as the {\bf trivial factor system}.
\end{example}

\begin{example} \label{ex:theta} 
Suppose that $q$ is divisible by $4$. Then $n \mapsto \binom{n}{2}$ is a well-defined function $\bZ/q \to \bZ/2$. Put
\begin{displaymath}
\fsA_1(a; b,c) = 1, \qquad
\fsA_2(a,b; c) = (-1)^{\binom{c}{2} ab}, \qquad
\fsA_{\sharp}(a,b) = (-1)^{\binom{a}{2} \binom{b}{2}}.
\end{displaymath}
Then $\fsA$ is an odd symmetric S-factor system. This is perhaps the most important example for this paper.
\end{example}

\begin{example} \label{ex:varpi}
Suppose that $\zeta_4 \in \bk$. Define
\begin{displaymath}
\fsB_1(a; b, c) = 1, \qquad
\fsB_2(a, b; c) = \begin{cases}
\zeta_4 & \text{if $a$, $b$, and $c$ are all odd} \\
1 & \text{otherwise} \end{cases}.
\end{displaymath}
Then $(\fsB_1, \fsB_2)$ is an odd B-factor system. Now suppose that $\zeta_8 \in \bk$, and put
\begin{displaymath}
\fsB_{\sharp}(a,b) = \begin{cases}
\zeta_8 & \text{if $a$ and $b$ are both odd} \\
1 & \text{otherwise} \end{cases}.
\end{displaymath}
Then $(\fsB_1, \fsB_2, \fsB_{\sharp})$ is an odd symmetric S-factor system. This factor system has the advantage over the $\fsA$ factor system in that it does not require $q$ to be divisible by~4.
\end{example}

\begin{example} \label{ex:eta}
Define
\begin{displaymath}
\fsC_1(r;s,t) = (-1)^{rst}, \qquad \fsC_2(r,s;t)=(-1)^{rst}, \qquad \fsC_{\sharp}(r,s)=1.
\end{displaymath}
Then $\fsC$ is an even symmetric S-factor system. 
\end{example}

\begin{example} \label{ex:kappa}
Define
\begin{displaymath}
\fsD_1(r;s,t) =  (-1)^{rst}, \qquad \fsD_2(r,s;t)=(-1)^{rst}, \qquad \fsD_{\sharp}(r,s)=(-1)^{rs}.
\end{displaymath}
Then $\fsD$ is an even symmetric S-factor system.
\end{example}

\begin{definition} \label{def:cobound}
Let $\phi \colon (\bZ/q)^2 \to \bk^{\times}$ be an arbitrary function. Put
\begin{align*}
\omega_1(r;s,t) &= \phi(r,s+t) \phi(r,s)^{-1} \phi(r,t)^{-1}, \\
\omega_2(r,s;t) &= \phi(r+s,t) \phi(r,s)^{-1} \phi(s,t)^{-1}, \\
\omega_{\sharp}(r,s) &= \phi(r,s) \phi(s,r).
\end{align*}
Then $(\omega_1, \omega_2, \omega_{\sharp})$ is a symmetric even S-factor system, called the {\bf coboundary} of $\phi$. We denote this system by $\partial(\phi)$.
\end{definition}

\begin{example} \label{ex:cobound1}
Suppose $4 \mid q$. Then $\fsC$ is the coboundary of the function $\fnC$ given by
\begin{displaymath}
\fnC(r,s)=(-1)^{\binom{rs}{2}}. \qedhere
\end{displaymath}
\end{example}

\begin{example} \label{ex:cobound2}
Suppose $q$ is an arbitrary even integer and $\zeta_4 \in \bk$. Then $\fsD$ is the coboundary of the function $\fnD$ given by 
\begin{displaymath}
\fnD(x,y) = \begin{cases}
1 & \text{if $x$ or $y$ is even} \\
\zeta_4 & \text{if $x$ and $y$ are odd} \end{cases}. \qedhere
\end{displaymath}
\end{example}

We now give a somewhat more involved example of a coboundary. 

\begin{proposition} \label{prop:type2-ab}
Set $\epsilon(n)=-1$ if $n \equiv 3 \pmod{4}$ and $\epsilon(n)=1$ otherwise. 
\begin{enumerate}[\rm \indent (a)]
\item Suppose that $\zeta_4 \in \bk$ and that $q$ is divisible by~$8$. Define
\begin{displaymath}
\mathtt{a}(n,m)=\epsilon(n)^m \zeta_4^{-\binom{n}{2} m}.
\end{displaymath}
Then $\fsA = \fsB \cdot \partial(\mathtt{a})$ holds as B-factor systems.
\item Suppose that $\zeta_{16} \in \bk$ and that $q$ is divisible by~$16$. Define
\begin{displaymath}
\mathtt{a}'(n,m)=\epsilon(n)^m \zeta_4^{-\binom{n}{2} m} \zeta_{16}^{nm}
\end{displaymath}
Then $\fsA = \fsB \cdot \partial(\mathtt{a}')$ holds as S-factor systems.
\end{enumerate}
\end{proposition}

\begin{proof}
(a) The first component of the equation $\fsA=\fsB \cdot \partial(\mathtt{a})$ amounts to the identity
  \[
    \epsilon(a)^{b+c} \zeta_4^{-\binom{a}{2} (b+c)} = \epsilon(a)^b \zeta_4^{-\binom{a}{2} b} \epsilon(a)^c \zeta_4^{-\binom{a}{2} c},
  \]
  which is clear. The second component amounts to
  \[
    \fsB_2(a,b;c) \epsilon(a+b)^c \zeta_4^{-\binom{a+b}{2} c} = (-1)^{\binom{c}{2} ab} \epsilon(b)^c \zeta_4^{-\binom{b}{2} c} \epsilon(a)^c \zeta_4^{-\binom{a}{2} c}.
  \]
The quotient of the left side by the right side is
  \begin{align*}
(-1)^{\binom{c}{2} ab} \fsB_2(a,b;c) \epsilon(a+b)^c \epsilon(a)^{c} \epsilon(b)^{c} \zeta_4^{-abc},
  \end{align*}
  which now only depends on $a,b,c$ modulo $4$. We will show it is 1 in all cases. If $c$ is even, this reduces to $((-1)^{\binom{c}{2}} \zeta_4^{-c})^{ab} = 1$. If $c$ is odd, then this simplifies to $\fsB_2(a,b;1) \epsilon(a+b) \epsilon(a) \epsilon(b) \zeta_4^{-ab}$. If $a=0$, we get $\epsilon(b)^2 = 1$. If $a=1$, we get $\fsB_2(1,b;1) \epsilon(b+1) \epsilon(b) \zeta_4^{-b}$. If $a=2$, we get $\epsilon(b+2) \epsilon(b) (-1)^b$. Finally, if $a=3$, we get $-\fsB_2(1,b;1) \epsilon(b-1) \epsilon(b) \zeta_4^b$. In all cases, we see that the expressions are 1 for all choices of $b$ by inspection.

(b) Since $\mathtt{a}'$ is equal to $\mathtt{a}$ times a quantity that is bilinear, the first two components of $\partial(\mathtt{a}')$ agree with those of $\partial(\mathtt{a})$. In particular, the equation $\fsA=\fsB \cdot \partial(\mathtt{a}')$ holds on the first two components by part~(a). On the $\sharp$ component, it amounts to the following identity
  \[
    \epsilon(b)^a \zeta_4^{-\binom{b}{2}a} \zeta_{16}^{ab} \epsilon(a)^b \zeta_4^{-\binom{a}{2} b} \zeta_{16}^{ab} (-1)^{\binom{a}{2} \binom{b}{2}} = \fsB_{\sharp}(a,b).
  \]
The quotient of the left side by the right side is
  \[
    \epsilon(b)^a \zeta_4^{-\binom{b}{2}a - \binom{a}{2}b} \zeta_{8}^{ab} \epsilon(a)^b (-1)^{\binom{a}{2} \binom{b}{2}} \fsB_{\sharp}(a,b)^{-1}.
  \]
  We will show this is always 1. Suppose that $a=2c$ is even (we will assume that $a,b$ are integers). Then the above simplifies to
  \[
    (-1)^{\binom{b}{2} c} \zeta_4^{-2bc^2 + bc} \zeta_4^{bc} (-1)^{\binom{b}{2} c} = 1.
  \]
  Since the expression is symmetric in $a$ and $b$, it remains to prove that the expression is 1 when $a=2c+1$ and $b=2d+1$ are odd. In that case, it simplifies to
  \begin{align*}
    &\ \quad    \epsilon(b) \zeta_4^{-(2d^2 + d)a} \zeta_4^{-(2c^2+c)b} (-1)^{cd} \zeta_4^{c+d} \zeta_8 \epsilon(a) (-1)^{cd}\zeta_8^{-1}\\
    &= \epsilon(b) (-1)^d \zeta_4^{-ad} (-1)^c \zeta_4^{-bc} \zeta_4^{c+d} \epsilon(a)\\
    &= \epsilon(a) (-1)^c \epsilon(b) (-1)^d = 1. \qedhere
  \end{align*}
\end{proof}

\begin{remark} \label{rmk:factor}
Let $\rC^r(G, A)$ denote the group of $r$-cochains of a group $G$ with values in the abelian group $A$. Note that $\rC^0(G,A)=A$ is the set of functions from a single point to $A$. These groups form a cochain complex that computes the group cohomology of $G$. Let $\rC^{r,s} = \rC^r(\bZ/q, \bk^{\times}) \otimes \rC^s(\bZ/q, \bZ)$. We have a natural map $i \colon \rC^{r,s} \to \rC^{r+s}(\bZ/q, \bk^{\times})$ given on pure tensors by
\begin{displaymath}
i(c \otimes c')(g_1, \ldots, g_{r+s}) = c(g_1, \ldots, g_r)^{c'(g_{r+1}, \ldots, g_{r+s})}.
\end{displaymath}
One easily sees that $i$ is an isomorphism (e.g., by considering cochains supported at a single point). Thus $\rC^{1,2} \oplus \rC^{2,1}$ can be identified with the set of pairs $(\omega_1, \omega_2)$ of functions $(\bZ/q)^3 \to \bk^{\times}$. For such a pair to form an even B-factor system, it must be killed by the differential of the total complex (this accounts for the first three conditions in the definition of factor system). This suggests that even B-factor systems are closely related to $\rH^3(\bZ/q \times \bZ/q, \bk^\times)$, and that there could be a more fundamental definition including $\rC^{0,3}$ and $\rC^{3,0}$. We note that 3-cocycles of $\bZ/q$ can be used to twist the associativity constraint in a monoidal $\bZ/q$-supercategory, which might be relevant. We expect that even S-factor systems are similarly related to third abelian cohomology (as defined in \cite{MacLane}). It is not clear to us exactly how odd factor systems relate to cohomology.
\end{remark}

\subsection{Braidings} \label{ss:braiding}

Let $\cA$ be a monoidal $\bZ/q$-supercategory with $q$ even (we allow $q=0$). Consider an even isomorphism of superfunctors $\beta \colon \otimes \to \otimes \circ \Sigma$. Explicitly, $\beta$ assigns to every pair $(A,B)$ of objects an even isomorphism $\beta_{A,B} \colon A \otimes B \to B \otimes A$ such that if $f \colon A \to A'$ and $g \colon B \to B'$ are homogeneous morphisms then the diagram
\begin{displaymath}
\xymatrix@C=4em{
A \otimes B \ar[d]_{f \otimes g} \ar[r]^{\beta_{A,B}} & B \otimes A \ar[d]^{g \otimes f} \\
A' \otimes B' \ar[r]^{\beta_{A',B'}} & B' \otimes A' }
\end{displaymath}
commutes up to $(-1)^{\vert f \vert \vert g \vert}$. Consider the two hexagons (H1) and (H2):
\begin{displaymath}
\begin{gathered}
\xymatrix@C=3em{
& A \otimes (B \otimes C) \ar[r]^-{\beta_{A,B\otimes C}} & (B \otimes C) \otimes A \ar[dr]^-{\alpha_{B,C,A}^{-1}} \\
(A \otimes B) \otimes C \ar[ur]^-{\alpha_{A,B,C}^{-1}} \ar[dr]_-{\beta_{A,B} \otimes 1_C} & & & B \otimes (C \otimes A)\\
& (B \otimes A) \otimes C \ar[r]_-{\alpha^{-1}_{B,A,C}} & B \otimes (A \otimes C) \ar[ur]_-{1_B \otimes \beta_{A,C}}
}
\end{gathered}
\tag{H1}
\end{displaymath}
\begin{displaymath}
\begin{gathered}
\xymatrix@C=3em{
& (A \otimes B) \otimes C \ar[r]^-{\beta_{A \otimes B,C}} & C \otimes (A \otimes B) \ar[dr]^-{\alpha_{C,A,B}} \\
A \otimes (B \otimes C) \ar[ur]^-{\alpha_{A,B,C}} \ar[dr]_-{1_A \otimes \beta_{B,C}} & & & (C \otimes A) \otimes B\\
& A \otimes (C \otimes B) \ar[r]_-{\alpha_{A,C,B}} & (A \otimes C) \otimes B \ar[ur]_-{\beta_{A,C} \otimes 1_B} 
}.
\end{gathered}
\tag{H2}
\end{displaymath}

\begin{definition} 
We say that $\beta$ is a {\bf braiding} with respect to the even B-factor system $\omega$ if in (H1) the bottom path is equal to $\omega_1(\vert A \vert; \vert B \vert, \vert C \vert)$ times the top path, and in (H2) the bottom path is equal to $\omega_2(\vert A \vert, \vert B \vert; \vert C \vert)$ times the top path, whenever $A$, $B$, and $C$ are homogeneous objects.
\end{definition}

\begin{definition}
We say that $\beta$ is a {\bf symmetry} with respect to the even S-factor system $\omega$ if it is a braiding and further satisfies the condition $\beta_{B,A} \circ \beta_{A,B}=\omega_{\sharp}(\vert A \vert, \vert B \vert) \id_{A \otimes B}$ for all homogeneous objects $A$ and $B$.
\end{definition}

\begin{remark}
By taking $\omega$ to be the trivial factor system (Example~\ref{ex:trivial}), one recovers the usual definition of symmetry: the diagrams (H1) and (H2) commute, and $\beta_{B,A} \beta_{A,B}=\id_{A \otimes B}$.
\end{remark}

\begin{remark} \label{rmk:convert}
Suppose $\beta$ is a symmetric with respect to $\omega$, and let $\phi \colon (\bZ/q)^2 \to \bk^{\times}$ be an arbitrary function. Define $\beta'$ by $\beta'_{X,Y}=\phi(\vert X \vert, \vert Y \vert) \beta_{X,Y}$. Then $\beta'$ is a symmetry with respect to $\omega \cdot \partial \phi$, where $\partial(\phi)$ is as in Definition~\ref{def:cobound}.
\end{remark}

\begin{remark}
The reason that even factor systems are the appropriate systems of scalars will be made clear in \S \ref{ss:explanation}.
\end{remark}

\begin{definition}
Let $\cA$ and $\cB$ be braided monoidal $\bZ/q$-supercategories. A {\bf braided monoidal $\Lambda$-superfunctor} $F \colon \cA \to \cB$ is a monoidal $\Lambda$-superfunctor $(F,\Phi,\phi)$ such that the following diagram commutes for all objects $X,Y$ of $\cA$
\[
  \xymatrix{  F(X) \otimes F(Y) \ar[rr]^-{\beta_{F(X), F(Y)}} \ar[d]_-{\Phi_{X,Y}}& & F(Y) \otimes F(X) \ar[d]^-{\Phi_{Y,X}} \\
    F(X \otimes Y) \ar[rr]^-{F(\beta_{X,Y})} && F(Y \otimes X) }.
\]
A natural transformation of braided monoidal superfunctors is simply one of monoidal superfunctors, i.e., the braiding is not used in the definition.
\end{definition}

\subsection{Type~I superbraidings} \label{ss:type1}

Suppose $\cA$ is a monoidal $\bZ/q$-supercategory, with $q$ even (we allow $q=0$), admitting a $\Pi$-structure. Consider an even isomorphism of functors $\beta \colon \otimes \to \otimes \circ \rT$. Explicitly, $\beta$ assigns to every pair $(A,B)$ of homogeneous objects an even isomorphism $\beta_{A,B} \colon A \otimes B \to \Pi^{\vert A \vert \vert B \vert}(B) \otimes A$ such that if $f \colon A \to A'$ and $g \colon B \to B'$ are nonzero homogeneous morphisms of homogeneous objects then the diagram 
\begin{displaymath}
\xymatrix@C=4em{
  A \otimes B \ar[d]_{f \otimes g} \ar[r]^-{\beta_{A,B}} & \Pi^{\vert A \vert \vert B \vert}(B) \otimes A \ar[d]^-{\Pi^{|A||B|}(g) \otimes f} \\
A' \otimes B' \ar[r]^-{\beta_{A',B'}} & \Pi^{\vert A \vert \vert B \vert}(B') \otimes A' }
\end{displaymath}
commutes up to $(-1)^{\vert f \vert \vert g \vert}$. Note that the existence of nonzero $f,g$ forces $|A|=|A'|$ and $|B|=|B'|$; if $f$ or $g$ is 0, then the power of $\Pi$ may need to change, but the resulting diagram will commute for trivial reasons.

The hexagons (H1) and (H2) admit obvious modifications (H1$'$) and (H2$'$) in this setting, as follows. Let $A$, $B$, and $C$ be homogeneous objects. Then (H1$'$) is the diagram
\begin{displaymath}
\begin{gathered}
\xymatrix@C=.75em{
& A \otimes (B \otimes C) \ar[r]^-{\psi_1} &
(\Pi^{b+c}(B) \otimes C) \otimes A \ar[dr]^-{\alpha^{-1}} \\
(A \otimes B) \otimes C \ar[ur]^-{\alpha^{-1}} \ar[dr]_-{\beta \otimes \id} & & &
\Pi^{b+c}(B) \otimes (C \otimes A) \\
& (\Pi^b(B) \otimes A) \otimes C \ar[r]_-{\alpha^{-1}}
& \Pi^b(B) \otimes (A \otimes C) \ar[ur]_-{\psi_2}
}.
\end{gathered}
\tag{H1$'$}
\end{displaymath}
Here $b=\vert A \vert \vert B \vert$, $c=\vert A \vert \vert C \vert$, $\psi_1$ is the composition
\begin{displaymath}
\xymatrix@C=5em{
A \otimes (B \otimes C) \ar[r]^-{\beta_{A,B \otimes C}} &
\Pi^{b+c}(B \otimes C) \otimes A \ar[r]^-{\sim} &
(\Pi^{b+c}(B) \otimes C) \otimes A },
\end{displaymath}
$\psi_2$ is the composition
\begin{displaymath}
\xymatrix@C=5em{
\Pi^b(B) \otimes (A \otimes C) \ar[r]^-{\id \otimes \beta_{A,C}} &
\Pi^b(B) \otimes (\Pi^c(C) \otimes A) \ar[r]^-{\xi \otimes \xi \otimes \id} &
\Pi^{b+c}(B) \otimes (C \otimes A) },
\end{displaymath}
and we have omitted indices for readability; recall $\xi=\xi^{n,m}$ is the isomorphism $\Pi^n \to \Pi^m$. The diagram (H2$'$) is:
\begin{displaymath}
\begin{gathered}
\xymatrix@C=.75em{
& (A \otimes B) \otimes C \ar[r]^-{\beta} &
\Pi^{a+b}(C) \otimes (A \otimes B) \ar[dr]^-{\alpha} \\
A \otimes (B \otimes C) \ar[ur]^-{\alpha} \ar[dr]_-{\id \otimes \beta} & & &
(\Pi^{a+b}(C) \otimes A) \otimes B \\
& A \otimes (\Pi^b(C) \otimes B) \ar[r]_-{\alpha} &
(A \otimes \Pi^b(C)) \otimes B \ar[ur]_-{\beta \otimes \id}
}
\end{gathered}
\tag{H2$'$}
\end{displaymath}
where $a=\vert A \vert \vert C \vert$ and $b=\vert B \vert \vert C \vert$, and we have again omitted indices.

It is not correct to ask for the above hexagons to commute: in fact, this leads to inconsistencies (see \S \ref{ss:explanation}). Instead, we make the following definition:

\begin{definition}
We say that $\beta$ is a {\bf type~I superbraiding} with respect to the odd B-factor system $\omega$ if in (H1$'$) the bottom path is equal to $\omega_1(\vert A \vert; \vert B \vert, \vert C \vert)$ times the top path, and in (H2$'$) the bottom path is equal to $\omega_2(\vert A \vert, \vert B \vert; \vert C \vert)$ times the top path.
\end{definition}

\begin{definition}
We say that $\beta$ is a {\bf type~I supersymmetry} with respect to the odd S-factor system $\omega$ if it is a type~I superbraiding with respect to $\omega$, and for all homogeneous objects $A$ and $B$, the diagram
\begin{displaymath}
\xymatrix@C=4em{
A \otimes B \ar[r]^{\beta_{A,B}} \ar[rd]_{\omega_{\sharp}(a,b) \xi \otimes \xi} &
\Pi^{ab}(B) \otimes A \ar[d]^{\beta_{\Pi^{ab}(B),A}} \\
& \Pi^{ab}(A) \otimes \Pi^{ab}(B) }
\end{displaymath}
commutes, where $a=\vert A \vert$ and $b=\vert B \vert$.
\end{definition}

\begin{remark}
If $\beta$ satisfies the conditions to be a supersymmetry then the two hexagon axioms are equivalent conditions: starting from either hexagon, invert all of the arrows and relabel the objects to get the other hexagon.
\end{remark}

\begin{example} \label{ex:type1a}
We define a {\bf type~Ia superbraiding} (respectively, {\bf supersymmetry}) to be a type~I superbraiding (respectively, supersymmetry) with respect to the factor system $\fsA$ from Example~\ref{ex:theta}.
\end{example}

\begin{example} \label{ex:type1b}
We define a {\bf type~Ib superbraiding} (respectively, {\bf supersymmetry}) to be a type~I superbraiding (respectively, supersymmetry) with respect to the factor system $\fsB$ from Example~\ref{ex:varpi}.
\end{example}

\begin{remark}
One can convert between type~Ia and type~Ib superbraidings when $\fsB \fsA^{-1}$ is a coboundary, in the sense of Definition~\ref{def:cobound} (see Remark~\ref{rmk:convert}). Proposition~\ref{prop:type2-ab} gives conditions that ensure $\fsB \fsA^{-1}$ is a coboundary.
\end{remark}

\begin{definition}
Let $\cA$ and $\cB$ be monoidal $\bZ/q$-supercategories equipped with type~I superbraidings with respect to the same factor systems. A {\bf type~I superbraided monoidal superfunctor} is a monoidal superfunctor $F \colon \cA \to \cB$ such that the following diagram commutes
\[
  \xymatrix@C=4em{  F(X) \otimes F(Y) \ar[rr]^-{\beta_{F(X), F(Y)}} \ar[d]_-{\Phi_{X,Y}} && \Pi^{|X||Y|} (F(Y) \otimes F(X)) \ar[d]^-{\Pi^{|X||Y|}(\Phi_{Y,X})} \\
    F(X \otimes Y) \ar[r]^-{F(\beta_{X,Y})} & F(\Pi^{|X||Y|}(Y \otimes X)) \ar@{=}[r] & \Pi^{|X||Y|}(F(Y \otimes X))}
\]
where the equality is the canonical isomorphism coming from Proposition~\ref{prop:Pi-functor}. A natural transformation of such functors is simply one of monoidal superfunctors.
\end{definition}

\subsection{Type~II superbraidings} \label{ss:type2}

The definition of type~I superbraiding followed the general procedure of replacing $\Sigma$ by $\rT$, and therefore has some conceptual clarity. However, these superbraidings can be somewhat inconvenient due to the $\Pi$'s that arise; for instance, if $X$ is an odd-degree object then a type~I superbraiding does not induce an endomorphism of $X \otimes X$, but rather a map $X \otimes X \to \Pi(X) \otimes X$. Of course, one can convert an even degree map $X \otimes Y \to \Pi^{\vert X \vert \vert Y \vert}(Y) \otimes X$ into a map $X \otimes Y \to Y \otimes X$ of degree $\vert X \vert \vert Y \vert$. This leads to an alternate formulation of the concept of superbraiding in which there are no $\Pi$'s, but the maps may not be even. We now give the details.

Let $\cA$ be a monoidal $\bZ/q$-supercategory, with $q$ even (we allow $q=0$); we do not require a $\Pi$-structure. Consider an isomorphism $\beta \colon \otimes \to \otimes \circ \Sigma$ such that if $A$ and $B$ are homogeneous then $\beta_{A,B}$ is homogeneous of degree $\vert A \vert \vert B \vert$. Explicitly, $\beta$ assigns to every pair of homogeneous objects $(A,B)$ an isomorphism $\beta_{A,B} \colon A \otimes B \to B \otimes A$ of degree $\vert A \vert \vert B \vert$ such that if $f \colon A \to A'$ and $g \colon B \to B'$ are homogeneous morphisms of homogeneous objects then the diagram
\begin{displaymath}
\xymatrix@C=4em{
A \otimes B \ar[d]_{f \otimes g} \ar[r]^{\beta_{A,B}} & B \otimes A \ar[d]^{g \otimes f} \\
A' \otimes B' \ar[r]^{\beta_{A',B'}} & B' \otimes A' }
\end{displaymath}
commutes up to $(-1)^{\vert A \vert \vert B \vert (\vert f \vert+\vert g \vert)+\vert f \vert \vert g \vert}$. The first part of the sign, namely $(-1)^{\vert A \vert \vert B \vert(\vert f \vert + \vert g \vert)}$, comes from the condition that $\beta$ is a natural transformation, while the second part of the sign, namely $(-1)^{\vert f \vert \vert g \vert}$, comes from the definition of $\Sigma$.

\begin{definition}
We say that $\beta$ is a {\bf type~II superbraiding} with respect to the B-factor system $\omega$ if in (H1) the bottom path is equal to $\omega_1(\vert A \vert; \vert B \vert, \vert C \vert)$ times the top path, and in (H2) the bottom path is equal to $\omega_2(\vert A \vert, \vert B \vert; \vert C \vert)$ times the top path.
\end{definition}

\begin{definition}
We say that $\beta$ is a {\bf type~II supersymmetry} with respect to the S-factor system $\omega$ if it is a type~II superbraiding with respect to $\omega$ and for all homogeneous objects $A$ and $B$ we have $\beta_{B,A} \beta_{A,B}= \omega_{\sharp}(\vert A \vert, \vert B \vert) \cdot \id_{A \otimes B}$
\end{definition}

\begin{remark}
We define type~IIa/IIb superbraidings/supersymmetries analogously to the type~I case (see Examples~\ref{ex:type1a} and~\ref{ex:type1b}).
\end{remark}

\begin{definition}
Let $\cA$ and $\cB$ be monoidal $\bZ/q$-supercategories equipped with type~II superbraidings with respect to the same factor systems. A {\bf type~II superbraided monoidal functor} is a monoidal superfunctor $F \colon \cA \to \cB$ such that the following diagram commutes
\[
  \xymatrix@C=6em{  F(X) \otimes F(Y) \ar[r]^-{\beta_{F(X), F(Y)}} \ar[d]_-{\Phi_{X,Y}} & F(Y) \otimes F(X) \ar[d]^-{\Phi_{Y,X}} \\
    F(X \otimes Y) \ar[r]^-{F(\beta_{X,Y})} & F(Y \otimes X) }
\]
for all objects $X$ and $Y$. A natural transformation of such functors is simply one of monoidal superfunctors.
\end{definition}

\subsection{Equivalence between type~I and type~II superbraidings} \label{ss:equivalence-type}

Fix a monoidal $\bZ/q$-supercategory (with $q$ even) $\cA$ with a $\Pi$-structure. The following proposition shows how to convert between the two types of superbraidings.

\begin{proposition} \label{prop:type12}
Suppose that $\beta$ is a type~II superbraiding with respect to the odd B-factor system $\omega$. For homogeneous objects $A$ and $B$, define $\beta'_{A,B}$ to be the composition
\begin{displaymath}
\xymatrix@C=6em{
A \otimes B \ar[r]^-{\beta_{A,B}} &
B \otimes A \ar[r]^-{\xi^{0,\vert A \vert \vert B \vert}_B \otimes \id_A} &
\Pi^{\vert A \vert \vert B \vert}(B) \otimes A }.
\end{displaymath}
Then $\beta'$ is a type~I superbraiding with respect to the odd B-factor system $\fsC \omega$, where $\fsC$ is as in Example~\ref{ex:eta}. Moreover, this construction defines a bijection between the two types of superbraidings.

Now suppose that $\beta$ is a type~II supersymmetry with respect to the odd S-factor system $\omega$. Then $\beta'$ is a type~I supersymmetry with respect to the odd S-factor system $\fsD \omega$, where $\fsD$ is as in Example~\ref{ex:kappa}. Again, this construction is bijective.
\end{proposition}

We note that $\fsC=\fsD$ as B-factor systems, and so one could use $\fsD$ in the first part of the proposition as well.

\begin{proof}
First we check that $\beta'$ is a natural transformation, i.e., that the following square
\[
\xymatrix@C=4em{  A \otimes B \ar[r]^-{\beta'_{A,B}} \ar[d]_-{f \otimes g} & \Pi^{|A||B|}(B) \otimes A \ar[d]^-{\Pi^{|A||B|}(g) \otimes f} \\
  A' \otimes B' \ar[r]^-{\beta'_{A',B'}} & \Pi^{|A||B|}(B') \otimes A' }
\]
commutes up to $(-1)^{|f||g|}$ for all nonzero homogeneous morphisms $f \colon A \to A'$ and $g \colon B \to B'$ between homogeneous objects. By definition, this expands to the diagram
\[
\xymatrix@C=4em{  A \otimes B \ar[r]^-{\beta_{A,B}} \ar[d]_-{f \otimes g} & B \otimes A \ar[d]^{g \otimes f} \ar[r]^-{\xi \otimes 1} & \Pi^{|A||B|}(B) \otimes A \ar[d]^-{\Pi^{|A||B|}(g) \otimes f} \\
  A' \otimes B' \ar[r]^-{\beta_{A',B'}} & B' \otimes A' \ar[r]^-{\xi \otimes 1} & \Pi^{|A||B|}(B') \otimes A' }
\]
Since $\beta$ is a type II superbraiding, the left square commutes up to $(-1)^{|A||B|(|f|+|g|) + |f||g|}$. By definition of how $\Pi$ acts on morphisms, the right square commutes up to $(-1)^{|A||B|(|f|+|g|)}$, and so the original square commutes up to $(-1)^{|f||g|}$, as desired.

Now we check that the hexagon axiom (H1$'$) holds. Consider the following diagram (for simplicity we ignore the associators)
\[
  \xymatrix{ A \otimes B \otimes C \ar[r]^-{\omega_1 \cdot \beta_{A,B \otimes C}} \ar[d]_-{\beta_{A,B} \otimes 1} & B \otimes C \otimes A \ar[r]^-{\xi \otimes 1} \ar[rdd]_-{\xi \otimes \xi \otimes 1} \ar[dd]_-{\xi \otimes 1 \otimes 1} & \Pi^{|A||B|+|A||C|}(B \otimes C) \otimes A \ar[d]^-{\cong} \\
    B \otimes A \otimes C \ar[d]_-{\xi \otimes 1 \otimes 1} \ar[ur]_-{1 \otimes \beta_{A,C}} & & \Pi^{|A||B|+|A||C|}(B) \otimes C \otimes A \\
  \Pi^{|A||B|}(B) \otimes A \otimes C \ar[r]^{1 \otimes \beta_{A,C}} & \Pi^{|A||B|}(B) \otimes C \otimes A \ar[r]^-{1 \otimes \xi \otimes 1} & \Pi^{|A||B|}(B) \otimes \Pi^{|A||C|}(C) \otimes A \ar[u]_-{\cong}}
\]
where the $\omega_1$ in the top left arrow is $\omega_1(\vert A \vert; \vert B \vert, \vert C \vert)$. On the left side, the upper triangle commutes by the first hexagon axiom for $\beta$. The bottom quadrilateral commutes up to $(-1)^{|A||B||C|}$. On the right side, the bottom triangle commutes up to $(-1)^{|A||B||C|}$ by the sign rule when we compose the bottom path. The top triangle commutes: we can take $\Pi(X \otimes Y)=\Pi(X) \otimes Y$, and $\xi \colon X \otimes Y \to \Pi(X \otimes Y)$ to be $\xi_X \otimes \id_Y$, and we use this convention for the $\Pi(B \otimes C)$ above. Then the top right vertical map is the identity map, the bottom right vertical map is $\xi \otimes \xi \otimes 1$. Hence we see that (H1$'$) holds for $\beta'$ by following the outer paths.

Finally, we check that the hexagon axiom (H2$'$) holds for $\beta'$ by considering the following diagram (again we ignore associators for simplicity)
\[
  \xymatrix@C=4em{
    A \otimes B \otimes C \ar[r]^-{\omega_2 \cdot \beta_{A \otimes B, C}} \ar[d]_-{1 \otimes \beta_{B,C}} & C \otimes A \otimes B \ar[r]^-{\xi \otimes 1 \otimes 1} \ar[dr]^-{\xi \otimes 1\otimes 1} & \Pi^{|A||C|+|B||C|}(C) \otimes A \otimes B \\
    A \otimes C \otimes B \ar[r]^-{1\otimes \xi \otimes 1} \ar[ur]^-{\beta_{A,C} \otimes 1} & A \otimes \Pi^{|B||C|}(C) \otimes B \ar[r] & \Pi^{|B||C|}(C) \otimes A \otimes B \ar[u]^-{\xi \otimes 1 \otimes 1}
  }
\]
where the $\delta$ in the top left arrow is $\omega_2(\vert A \vert, \vert B \vert; \vert C \vert)$, and the bottom right map is $\beta_{A,\Pi^{|B||C|}(C)} \otimes 1$. The left triangle commutes by the second hexagon axiom for $\beta$. The middle quadrilateral commutes up to $(-1)^{|A||B||C|}$ and the right triangle commutes. Hence by considering the outer paths, we see that $\beta'$ satisfies the second hexagon axiom.

Finally, the diagrams above can be used to show that $\beta$ is a type~II superbraiding if $\beta'$ is a type~I superbraiding and hence the stated construction gives a bijection between type~I and type~II superbraidings.

Now we handle the last statement, and suppose that $\beta$ is a type~I supersymmetry. Consider the following diagram:
\[
  \xymatrix@C=4em{
    A \otimes B \ar[r]^-{\beta_{A,B}} \ar[dr]_-{\omega_{\sharp}(|A|,|B|)} &
    B \otimes A \ar[d]^-{\beta_{B,A}} \ar[r]^-{\xi \otimes 1} &
    \Pi^{|A||B|}(B) \otimes A \ar[d]^-{\beta_{\Pi(B),A}} \\
    & A \otimes B \ar[r]^-{1 \otimes \xi} \ar[dr]_-{\xi \otimes \xi} &
    A \otimes \Pi^{|A||B|}(B) \ar[d]^-{\xi \otimes 1} \\
    & & \Pi^{|A||B|}(A) \otimes \Pi^{|A||B|}(B)
    }
  \]
  The upper left triangle commutes since $\beta$ is a type II supersymmetry, and the bottom triangle commutes by definition of composition. The square commutes up to $(-1)^{|A|||B|}$ since $\beta$ is a type II superbraiding, so we conclude that $\beta'$ is a superbraiding with respect to $\fsD \omega$.
\end{proof}

\subsection{Explanation of factor systems} \label{ss:explanation}

We now explain the rationale behind the definition of B-factor systems. We work with type~II superbraidings here, but this discussion can be converted to type~I superbraidings or ordinary braidings as well. Let $\cA$ be a monoidal $\bZ/q$-supercategory with $q$ even. We assume, for simplicity, that the monoidal structure is strict. Suppose that we have an isomorphism $\beta_{X,Y} \colon \otimes \to \otimes \circ \Sigma$ such that $\beta_{X,Y}$ is homogeneous of degree $\vert X \vert \vert Y \vert$ on homogeneous objects, and that (H1) and (H2) commute up to scalars. Precisely, for homogeneous objects $A$, $B$, and $C$ of degrees $a$, $b$, and $c$, we suppose that the diagrams
\begin{displaymath}
\xymatrix{
A \otimes B \otimes C \ar[rr]^{\omega_1(a; b, c) \cdot \beta_{A,B \otimes C}} \ar[rd]_{\beta_{A,B} \otimes \id_C} && B \otimes C \otimes A \\
& B \otimes A \otimes C \ar[ru]_{\id_B \otimes \beta_{A,C}} }
\end{displaymath}
\begin{displaymath}
\xymatrix{
A \otimes B \otimes C \ar[rr]^{\omega_2(a, b; c) \cdot \beta_{A \otimes B, C}} \ar[rd]_{\id_A \otimes \beta_{B,C}} && C \otimes A \otimes B \\
& A \otimes C \otimes B \ar[ru]_{\beta_{A,C} \otimes \id_B} }
\end{displaymath}
commute, where $\omega_1$ and $\omega_2$ are non-zero scalars. Note that these triangles are the hexagon axioms since we have assumed that the associators are the identity.

\begin{subequations}
Let $A$, $B$, $C$, and $D$ be four homogeneous objects of degrees $a$, $b$, $c$, and $d$. Consider the tetrahedron
\begin{equation} \label{eq:diagram1}
\begin{gathered}
\xymatrix@R=16pt{
& ABCD \ar[ld] \ar[rd] \ar[dd] \\
BACD \ar[rd] \ar@{..>}[rr] && BCDA \\
& BCAD  \ar[ru] }
\end{gathered}
\end{equation}
Here we have omitted the $\otimes$ symbols, and each arrow is the unique $\beta$ map that makes sense. This tetrahedron has four faces that each commute up to an $\omega_1$-factor. This yields the identity \eqref{eq:condition1}.

Similarly, we have the commuting tetrahedron
\begin{equation} \label{eq:diagram2}
\begin{gathered}
\xymatrix@R=16pt{
& ABCD \ar[ld] \ar[rd] \ar[dd] \\
ABDC \ar[rd] \ar@{..>}[rr] && DABC \\
& ADBC  \ar[ru] }
\end{gathered}
\end{equation}
that yields the identity \eqref{eq:condition2}.

Now consider the following diagram:
\begin{equation} \label{eq:diagram3}
\begin{gathered}
\xymatrix@R=16pt{
ABCD \ar[rr] \ar[ddd] \ar[rdd] \ar[rdddd] &&
CDAB \\ \\
& ACDB \ar@{..>}[rd] \ar[ruu] \\
ACBD \ar@{..>}[ru] \ar[rd] && CADB \ar[uuu] \\
& CABD \ar[ru] \ar[ruuuu] }
\end{gathered}
\end{equation}
This diagram has six triangles and one square (composed of the bottom four vertices). If $a$, $b$, $c$, and $d$ are all odd then both $\beta_{A,C}$ and $\beta_{B,D}$ are odd, and the square anti-commutes; this is the one place where the specifics of supercategories intervene. Otherwise, at least one of $\beta_{A,C}$ and $\beta_{B,D}$ are even, and the square commutes. Thus, in general, the square commutes up to $(-1)^{abcd}$, and we obtain \eqref{eq:condition3} with $p=1$. (In the case of ordinary braidings, this sign would not be present, and so we would obtain \eqref{eq:condition3} with $p=0$.)

Finally, consider the following diagram:
\begin{equation} \label{eq:diagram4}
\begin{gathered}
\xymatrix@C=4em{
& BAC \ar[r] \ar[rrd] & BCA \ar[rd] \\
ABC \ar[ru] \ar[rd] \ar[rru] \ar[rrd] &&& CBA \\
& ACB \ar[r] \ar[rru] & CAB \ar[ru] }
\end{gathered}
\end{equation}
This diagram has four triangles, two that commute up to $\omega_1$-factors and two that commute up to $\omega_2$-factors. It also has two squares that commute by the naturality of $\beta$. For example, one of the squares is
\begin{displaymath}
\xymatrix@C=4em{
ABC \ar[r]^{\beta_{A,BC}} \ar[d]_{\id \otimes \beta_{B,C}} & BCA \ar[d]^{\beta_{B,C} \otimes \id} \\
ACB \ar[r]^{\beta_{A,CB}} & CBA }
\end{displaymath}
and this commutes up to $(-1)^{a(b+c)bc}=1$. This diagram yields \eqref{eq:condition4}.
\end{subequations}

We thus see that the four equations in the definition of a B-factor system are necessary for consistency of the hexagon axioms. Of course, one might wonder if there are additional equations that are needed that we have overlooked. In fact, this is not the case: this is best explained by the 2-categorical considerations in \S \ref{s:2cat}. (The discussion there will also explain the equations for S-factor systems.)

\subsection{Grothendieck groups}

Let $\cA$ be a monoidal $\bZ/q$-supercategory with $q$ even. Suppose that $\cA^{\circ}$ has an exact structure and $\otimes$ is exact. Then $\otimes$ induces a multiplication on $\rK(\cA^{\circ})$ and $\rK_{\pm}(\cA)$ and gives them each the structure of a superring. If the monoidal structure admits a braiding then all three of these superrings are commutative. If the monoidal structure admits a superbraiding (of any kind) then $\rK_+(\cA)$ is commutative while $\rK_-(\cA)$ is supercommutative. Thus, via $\rK_-$, we can view supersymmetric monoidal supercategories as a categorification of the notion of a supercommutative superalgebra.

\section{Some 2-category theory} \label{s:2cat}

In this section, we explain the 2-categorical underpinnings for supersymmetries and superbraidings. We do not intend for this discussion to be perfectly rigorous, and we allow ourselves to take some shortcuts and make some simplifications. We have mostly followed \cite{crans} for the theory of monoidal 2-categories, and refer there for more details; see also \cite{baez} and \cite{kapranov}.

\subsection{Monoidal 2-categories}

Let $(\sC, \otimes)$ be a semistrict monoidal 2-category. We refer to \cite[\S 2.1]{crans} for the precise definition, and make a few remarks here. We require composition of 1-morphisms in $\sC$ to be strictly associative (this is part of what ``$2$-category'' means). We also require $\otimes$ to be strictly associative (this is part of what semistrict means). We ignore the unit of the monoidal structure in our discussion. An important piece of the monoidal structure is that the square
\begin{displaymath}
\xymatrix@C=4em{
A \otimes B \ar[r]^{\id_A \otimes g} \ar[d]_{f \otimes \id_B} & A \otimes B' \ar[d]^{f \otimes \id_{B'}} \\
A' \otimes B \ar[r]^{\id_{A'} \otimes g} & A' \otimes B' }
\end{displaymath}
commutes up to a given 2-isomorphism $i_{f,g}$; precisely,
\begin{displaymath}
i_{f,g} \colon (f \otimes \id_{B'}) \circ (\id_A \otimes g) \xrightarrow{\sim} (\id_{A'} \otimes g) \circ (f \otimes \id_B).
\end{displaymath}
In what follows, we omit tensor symbols (i.e., we write $AB$ in place of $A \otimes B$), and simply write $A$ for the identity 1-morphism on an object $A$ of $\sC$.

\subsection{Braidings on monoidal 2-categories}

A \textbf{braiding} on $\sC$ consists of a pseudonatural equivalence $R_{A,B} \colon AB \to BA$ together with 2-isomorphisms
\begin{displaymath}
\rho_{A \vert B, C} \colon (BR_{A,C}) \circ (R_{A,B}C) \to R_{A,BC}, \qquad \rho_{A,B \vert C} \colon (R_{A,C}B) \circ (AR_{B,C}) \to R_{AB, C}.
\end{displaymath}
satisfying the five axioms below (we have inverted the convention from \cite[\S 2.2]{crans} using $\rho$ in place of $\rho^{-1}$). Note that the above two 2-isomorphisms make the diagrams
\begin{displaymath}
\xymatrix{
ABC \ar[rr]^{R_{A,BC}} \ar[rd]_{R_{A,B}C} && BCA \\
& BAC \ar[ru]_{BR_{A,C}}
}
\qquad
\xymatrix{
ABC \ar[rr]^{R_{AB,C}} \ar[rd]_{AR_{B,C}} && CAB \\
& ACB \ar[ru]_{R_{A,C}B}
}
\end{displaymath}
commute up to isomorphism, and are essentially the hexagon axioms. (Remember that the associators are identity morphisms, and can thus be omitted.) We recall that if $f \colon A \to A'$ and $g \colon B \to B'$ are 1-morphisms, then there is a given 2-isomorphism
\begin{displaymath}
R_{f,g} \colon (gf)R_{A,B} \xrightarrow{\sim} R_{A',B'} (fg);
\end{displaymath}
this is part of the data of the pseudonatural equivalence $R$. This isomorphism makes the square
\begin{displaymath}
\xymatrix@C=4em{
AB \ar[r]^{R_{A,B}} \ar[d]_{fg} & BA \ar[d]^{gf} \\
A'B' \ar[r]^{R_{A',B'}} & B'A' }
\end{displaymath}
commute up to isomorphism.

\textit{Axiom 1.}
The tetrahedron in \eqref{eq:diagram1} commutes for all objects $A$, $B$, $C$, and $D$, where the edges are the obvious 1-morphisms built from $R$ and the faces are the 2-morphisms provided by $\rho$. Here ``commutes'' is taken in the 2-categorical sense. Explicitly, this means that the diagram
\begin{displaymath}
\xymatrix@C=5em{
(BCR_{A,D})(BR_{A,C}D)(R_{A,B}CD) \ar[r]^-{B\rho_{A \vert C,D} \circ 1} \ar[d]_{1 \circ (\rho_{A \vert B,C}D)} &
(BR_{A,CD})(R_{A,B}CD) \ar[d]^{\rho_{A \vert B, CD}} \\
(BCR_{A,D})(R_{A,BC}D) \ar[r]^-{\rho_{A \vert BC, D}} &
R_{A,BCD} }
\end{displaymath}
commutes.

\textit{Axiom 2.}
The tetrahedron in \eqref{eq:diagram2} commutes. This means that the diagram
\begin{displaymath}
\xymatrix@C=5em{
(R_{A,D}BC)(AR_{B,D}C)(ABR_{C,D}) \ar[r]^-{(\rho_{A,B\vert D}C) \circ 1} \ar[d]_{1 \circ (A \rho_{B,C \vert D})} &
(R_{AB,D}C)(ABR_{C,D}) \ar[d]^{\rho_{AB,C \vert D}} \\
(R_{A,D}BC)(AR_{BC,D}) \ar[r]^-{\rho_{A,BC \vert D}} &
R_{ABC,D} }
\end{displaymath}
commutes.

\textit{Axiom 3.}
The diagram in \eqref{eq:diagram3} commutes; the 2-isomorphism in the square is given by $i$. We write this out explicitly using three diagrams. First, we require that the diagram
\begin{displaymath}
\xymatrix@C=5em{
(CR_{A,D}B)(R_{A,C}DB)(ACR_{B,D})(AR_{B,C}D) \ar[d]_{\rho_{A\vert C,D}B \circ 1 \circ 1} \ar[r]^-{1 \circ 1 \circ A \rho_{B \vert C,D}} \ar@{..>}[rd]^{\phi} &
(CR_{A,D}B)(R_{A,C}DB)(AR_{B,CD}) \ar[d]^{\rho_{A\vert C,D}B \circ 1}\\
(R_{A,CD}B)(ACR_{B,D})(AR_{B,C}D) \ar[r]^-{1 \circ A \rho_{B \vert C,D}} &
(R_{A,CD}B)(AR_{B,CD}) 
}
\end{displaymath}
commutes, and define $\phi$ to be the indicated arrow. Second, we require that the diagram
\begin{displaymath}
  \xymatrix@C=5em{
    (CR_{A,D}B)(CAR_{B,D})(R_{A,C}BD)(AR_{B,C}D) \ar[d]_{1 \circ 1 \circ \rho_{A,B\vert C}D} \ar[r]^-{C \rho_{A,B \vert D} \circ 1 \circ 1} \ar@{..>}[rd]^{\psi} &
(CR_{AB,D})(R_{A,C}BD)(AR_{B,C}D) \ar[d]^{1 \circ \rho_{A,B \vert C} D} \\
(CR_{A,D}B)(CAR_{B,D})(R_{AB,C}D) \ar[r]^-{C\rho_{A,B\vert D} \circ 1} &
(CR_{AB,D})(R_{AB,C}D) }
\end{displaymath}
commutes, and again define $\psi$ to be the indicated arrow. Finally, we require that the diagram
\begin{displaymath}
\tiny
\xymatrix{
(CR_{A,D}B)(R_{A,C}DB)(ACR_{B,D})(AR_{B,C}D) \ar[d]_{\phi} \ar[rr]^{1 \circ i \circ 1} &&
(CR_{A,D}B)(CAR_{B,D})(R_{A,C}BD)(AR_{B,C}D) \ar[d]^{\psi} \\
(R_{A,CD}B)(AR_{B,CD}) \ar[r]^{\rho_{A,B\vert CD}} & R_{AB,CD} &
(CR_{AB,D})(R_{AB,C}D) \ar[l]_{\rho_{AB\vert C,D}} }
\end{displaymath}
commutes, where $i=i_{f,g}$ with $f=R_{A,C}$ and $g=R_{B,D}$.

\textit{Axiom 4.}
The diagram in \eqref{eq:diagram4} commutes; the 2-isomorphisms in the four triangles come from $\rho$, while the 2-isomorphisms in the two squares come from $R$. Explicitly, this means that the diagram
\begin{displaymath}
\xymatrix@C=4em{
(R_{BA,C})(R_{A,B}C) \ar[d]_{R_{R_{A,B},\id_C}} &
(R_{B,C}A)(BR_{A,C})(R_{A,B}C) \ar[l]_-{\rho_{B,A \vert C} \circ 1} \ar[r]^-{1 \circ \rho_{A \vert B,C}} &
(R_{B,C}A)(R_{A,BC}) \ar[d]^{R_{\id_A,R_{B,C}}} \\
(CR_{A,B})(R_{AB,C}) \ar[r]^-{1 \circ \rho^{-1}_{A,B \vert C}} &
(CR_{A,B})(R_{A,C}B)(AR_{B,C}) &
(R_{A,CB})(AR_{B,C}) \ar[l]_-{\rho^{-1}_{A \vert C,B} \circ 1} }
\end{displaymath}
commutes. 

\textit{Axiom 5.}
A number of conditions regarding the unit object; see \cite[\S 2.2]{crans}.

\subsection{Symmetric and sylleptic structures}

There are two ``levels'' of commutativity for monoidal 1-categories: braided and symmetric. For monoidal 2-categories, there are three levels: braided, sylleptic, and symmetric. We now recall what the latter two mean.

Fix a braiding $R$ on $\sC$. A \textbf{syllepsis} is a 2-isomorphism (we have inverted the convention from \cite[\S 4]{crans}, using $v$ in place of $v^{-1}$) 
\begin{displaymath}
v_{A,B} \colon R_{B,A} \circ R_{A,B} \to \id_{A \otimes B}
\end{displaymath}
such that the following diagrams commute (as before, the dotted arrows denote the definition of the common composition in the first and third diagrams).
\begin{displaymath}
\xymatrix@C=4em{
(R_{B,A}C)(BR_{C,A})(BR_{A,C})(R_{A,B}C) \ar[rr]^-{1 \circ 1 \circ \rho_{A|B,C}} \ar[d]_-{\rho_{B,C|A} \circ 1 \circ 1} \ar@{..>}[drr]^-\phi && (R_{B,A}C)(BR_{C,A})(R_{A,BC}) \ar[d]^-{\rho_{B,C|A} \circ 1} \\
(R_{BC,A})(BR_{A,C})(R_{A,B}C) \ar[rr]^-{1 \circ \rho_{A|B,C}} && (R_{BC,A})(R_{A,BC}) },
\end{displaymath}
\begin{displaymath}
\xymatrix@C=4em{
(R_{B,A}C)(BR_{C,A})(BR_{A,C})(R_{A,B}C) \ar[r]^-{\phi} \ar[d]_-{1 \circ Bv_{A,C} \circ 1} & (R_{BC,A})(R_{A,BC}) \ar[d]^-{v_{A,BC}} \\
(R_{B,A}C)(R_{A,B}C) \ar[r]^-{v_{A,B}C} & \id_{ABC} },
\end{displaymath}
\begin{displaymath}
\xymatrix@C=4em{
(AR_{C,B})(R_{C,A}B)(R_{A,C}B)(AR_{B,C}) \ar[rr]^-{1 \circ 1 \circ \rho_{A,B|C}} \ar[d]_-{\rho_{C|A,B} \circ 1 \circ 1} \ar@{..>}[drr]^-\psi && (AR_{C,B})(R_{C,A}B)(R_{AB,C}) \ar[d]^-{\rho_{C|A,B} \circ 1} \\
(R_{C,AB})(R_{A,C}B)(AR_{B,C}) \ar[rr]^-{1 \circ \rho_{A,B|C}} && (R_{C,AB})(R_{AB,C}) },
\end{displaymath}
\begin{displaymath}
\xymatrix@C=4em{
(AR_{C,B})(R_{C,A}B)(R_{A,C}B)(AR_{B,C}) \ar[r]^-{\psi} \ar[d]_-{1 \circ v_{A,C}B \circ 1} & (R_{C,AB})(R_{AB,C}) \ar[d]^-{v_{AB,C}} \\
(AR_{C,B})(AR_{B,C}) \ar[r]^-{Av_{B,C}} & \id_{ABC} }.
\end{displaymath}
There are also conditions involving the unit object that we omit (see \cite[\S 4]{crans}). A syllepsis is a \textbf{symmetry} if $1 \circ v_{A,B} = v_{B,A} \circ 1$ agree as isomorphisms
\begin{displaymath}
R_{A,B}R_{B,A}R_{A,B} \to R_{A,B}.
\end{displaymath}

\subsection{The 2-category of supercategories}

Fix an even positive integer $q$. For the remainder of this section, we let $\sC$ be the 2-category whose objects are $\bZ/q$-supercategories equipped with a $\Pi$-functor. The 1-morphisms of $\sC$ are superfunctors, and the 2-morphisms are even natural transformations. This is a 2-category in the sense that we have been using it, as composition of 1-morphisms in $\sC$ is strictly associative. In this section, we denote supercategories by $A, B, \ldots$, and objects of supercategories by $a,b,\ldots$, to maintain consistency with the rest of this section.

Let $A$ and $B$ be objects of $\sC$. We let $AB=A \otimes B$ be the supercategory $A \boxtimes B$ defined in \S \ref{ss:scat-prod}. For objects $a \in A$ and $b \in B$, we write $a \otimes b$, or simply $ab$, instead of $a \boxtimes b$ in this section. We equip $A \otimes B$ with the $\Pi$-structure $\Pi(a \otimes b)=\Pi(a) \otimes b$. Thus $A \otimes B$ is again an object of $\sC$. This defines a monoidal structure on $\sC$. However, it is \emph{not} semistrict, since this product is not strictly associative. We will ignore this detail, and simply treat $\otimes$ as if it were semistrict. This is the main reason we disclaimed at the start of this section that our discussion would not be entirely rigorous.

In the remainder of this section, we investigate how $\Sigma$ and $\rT$ (as defined in \S \ref{ss:sigma-tau}) define braidings (or syllepses or symmetries) on $\sC$. We focus on the $\rT$ case since it is most relevant to this paper, but the $\Sigma$ case is very similar. To maintain consistency with the rest of this section, we write $R$ in place of $\rT$. Thus, for objects $A$ and $B$ of $\sC$, we have the equivalence $R_{A,B} \colon A \otimes B \to B \otimes A$ defined on objects by $R_{A,B}(a \otimes b)=\Pi^{\vert a \vert \vert b \vert}(b) \otimes a$, and on morphisms by $R_{A,B}(f \otimes g)=(-1)^{\vert f \vert \vert g \vert} \Pi^{\vert a \vert \vert b \vert}(g) \otimes f$.

\subsection{Braiding} \label{ss:scat-braiding}

To give $R$ the structure of a braiding, we first must define the $\rho$ isomorphisms. Fix functions $\omega_1,\omega_2 \colon (\bZ/q)^3 \to \bk^{\times}$. Let $A$, $B$, and $C$ be objects of $\sC$, and let $a$, $b$, and $c$ be homogeneous objects of $A$, $B$, and $C$. We have
\begin{displaymath}
((BR_{A,C})(R_{A,B}C))(a \otimes b \otimes c)=\Pi^{\vert a \vert \vert b \vert}(b) \otimes \Pi^{\vert a \vert \vert c \vert}(c) \otimes a
\end{displaymath}
and
\begin{displaymath}
R_{A,BC}(a \otimes b \otimes c) = \Pi^{\vert a \vert (\vert b \vert+\vert c \vert)}(b) \otimes c \otimes a.
\end{displaymath}
We define
\begin{displaymath}
\rho_{A \vert B,C} \colon \Pi^{\vert a \vert \vert b \vert}(b) \otimes \Pi^{\vert a \vert \vert c \vert}(c) \otimes a \to \Pi^{\vert a \vert(\vert b \vert+\vert c \vert)}(b) \otimes c \otimes a
\end{displaymath}
by
\begin{displaymath}
\rho_{A \vert B,C} = \omega_1(a; b, c) \cdot \xi \otimes \xi \otimes 1,
\end{displaymath}
where we write $\omega_1(a; b, c)$ in place of $\omega_1(\vert a \vert; \vert b \vert, \vert c \vert)$. Here $\xi^{n,m} \colon \Pi^n \to \Pi^m$ is the isomorphism defined in \S \ref{ss:pi}, and we omit the indices to $\xi$ when they are clear. Similarly, we have
\begin{displaymath}
((R_{A,C}B)(AR_{B,C}))(a \otimes b \otimes c) = \Pi^{\vert a \vert \vert c \vert}(\Pi^{\vert b \vert \vert c \vert}(c)) \otimes a \otimes b
\end{displaymath}
and
\begin{displaymath}
R_{AB, C}(a \otimes b \otimes c) = \Pi^{\vert c \vert(\vert a \vert+\vert b \vert)}(c) \otimes a \otimes b.
\end{displaymath}
We define $\rho_{A,B \vert C}$ to be $\omega_2(a,b; c)$ times the identity map; note that $\Pi^{\vert a \vert \vert c \vert} \circ \Pi^{\vert b \vert \vert c \vert} = \Pi^{\vert c \vert (\vert a \vert+\vert b \vert)}$. We have thus defined the necessary data for a braiding on $\sC$. We now examine the axioms.

\textit{Axiom 1.}
The diagram becomes
\[
  \xymatrix@C=8em{\Pi^{|a||b|}(b) \Pi^{|a||c|}(c) \Pi^{|a||d|}(d) a \ar[r]^-{\omega_1(a; c,d) 1 \otimes  \xi \otimes \xi \otimes 1} \ar[d]_-{\omega_1(a; b,c) \xi \otimes \xi \otimes 1 \otimes 1} &
    \Pi^{|a||b|}(b) \Pi^{|a|(|c|+|d|)}(c)da \ar[d]^-{\omega_1(a; b,cd) \xi \otimes \xi \otimes 1 \otimes 1} \\
  \Pi^{|a|(|b|+|c|)}(b)c \Pi^{|a||d|}(d) a \ar[r]^-{\omega_1(a; bc,d) \xi \otimes 1 \otimes \xi \otimes 1} & \Pi^{|a|(|b|+|c|+|d|)}(b)cda }.
\]
The top composition is $\omega_1(a; b,cd) \omega_1(a; c,d) \xi \otimes \xi \otimes \xi \otimes 1$ while the bottom composition is $\omega_1(a;b,c) \omega_1(a;bc,d) \xi \otimes \xi \otimes \xi \otimes 1$. Hence the diagram commutes if and only if
\[
\omega_1(a;b,c) \omega_1(a;bc,d) = \omega_1(a;b,cd) \omega_1(a;c,d)
\]
which is equivalent to \eqref{eq:condition1} (note we have multiplicative notation as opposed to additive notation here since we are using the objects rather than their degrees as input).

\textit{Axiom 2.}
The diagram becomes
\[
  \xymatrix@C=8em{ \Pi^{|a||d|}(\Pi^{|b||d|}(\Pi^{|c||d|}(d))) abc \ar[r]^-{\omega_2(a,b;d)} \ar[d]_-{\omega_2(b,c;d)} & \Pi^{(|a|+|b|)|d|}(\Pi^{|c||d|}(d))abc \ar[d]^-{\omega_2(ab,c;d)} \\
  \Pi^{|a||d|}(\Pi^{(|b|+|c|)|d|}(d))abc \ar[r]^-{\omega_2(a,bc;d)} & \Pi^{(|a|+|b|+|c|)|d|}(d)abc}
\]
where each map is a scalar times the identity. Thus, commutativity of this diagram is equivalent to the identity
\[
  \omega_2(ab,c;d) \omega_2(a,b;d) = \omega_2(b,c;d) \omega_2(a,bc;d)
\]
which is \eqref{eq:condition2}.

\textit{Axiom 3.}
The first diagram comes out to be the following:
\begin{displaymath}
\xymatrix@C=4em{
\Pi^{\vert a \vert \vert c \vert}(\Pi^{\vert b \vert \vert c \vert}(c)) \Pi^{\vert a \vert \vert d \vert}(\Pi^{\vert b \vert \vert d \vert}(d)) ab \ar[r]^{\circled{1}} \ar[d]_{\circled{2}} \ar@{..>}[rd]^{\phi} &
\Pi^{\vert a \vert \vert c \vert}(\Pi^{\vert b \vert(\vert c \vert + \vert d \vert)}(c)) \Pi^{\vert a \vert \vert d \vert}(d) ab \ar[d]^{\circled{3}} \\
\Pi^{\vert a \vert(\vert c \vert+\vert d \vert)}(\Pi^{\vert b \vert \vert c \vert}(c)) \Pi^{\vert b \vert \vert d \vert}(d) ab \ar[r]^{\circled{4}} &
\Pi^{\vert a \vert(\vert c \vert+\vert d \vert)}(\Pi^{\vert b \vert(\vert c \vert+\vert d \vert)}(c)) dab }.
\end{displaymath}
The four numbered maps are given by
\begin{align*}
  \circled{1} &= (-1)^{\vert a \vert \vert b \vert \vert d \vert(|c|+1)} \omega_1(b; c,d) \xi \otimes \xi \otimes 1 \otimes 1\\
\circled{2} &= \omega_1(a; c, d) \xi \otimes \xi \otimes 1 \otimes 1 \\
\circled{3} &= \omega_1(a; c,d) \xi \otimes \xi \otimes 1 \otimes 1 \\
  \circled{4} &= (-1)^{\vert a \vert\vert b \vert \vert d \vert(|c|+1)} \omega_1(b; c, d)  \xi \otimes \xi \otimes 1 \otimes 1 
\end{align*}
To compute the signs, we use the identities $\Pi^k(\xi^{n,m})=(-1)^{k(n+m)} \xi^{n+k,m+k}$ and $\xi^{n,m}_{\Pi^k}=\xi^{n+k,m+k}$. We thus see that the square commutes, and
\begin{displaymath}
\phi = (-1)^{\vert a \vert \vert b \vert \vert c \vert \vert d \vert} \omega_1(b; c,d) \omega_1(a; c,d) \xi \otimes \xi \otimes 1 \otimes 1.
\end{displaymath}
Note that one has to apply the sign rule when computing the compositions, and this cancels the $(-1)^{\vert a \vert \vert b \vert \vert d \vert}$ part of the sign.
The second diagram is the following:
\begin{displaymath}
\xymatrix@C=4em{
\Pi^{\vert a \vert \vert c \vert}(\Pi^{\vert b \vert \vert c \vert}(c)) \Pi^{\vert a \vert \vert d \vert}(\Pi^{\vert b \vert \vert d \vert}(d)) ab \ar[d]_-{\omega_2(a,b;c)} \ar[r]^-{\omega_2(a,b;d)} \ar@{..>}[rd]^{\psi} &
\Pi^{\vert a \vert \vert c \vert}(\Pi^{\vert b \vert \vert c \vert}(c)) \Pi^{\vert d \vert(\vert a \vert+\vert b \vert)}(d)ab \ar[d]^-{\omega_2(a,b;c)} \\
\Pi^{\vert c \vert(\vert a \vert+\vert b \vert)}(c) \Pi^{\vert a \vert \vert d \vert}(\Pi^{\vert b \vert \vert d \vert}(d)) ab \ar[r]^-{\omega_2(a,b;d)} &
\Pi^{\vert c \vert(\vert a \vert+\vert b \vert)}(c) \Pi^{\vert d \vert(\vert a \vert+\vert b \vert)}(d) ab }
\end{displaymath}
All maps are scalar multiples of the identity, so the diagram commutes and
\begin{displaymath}
\psi=\omega_2(a,b; d) \omega_2(a,b; c) \cdot \id.
\end{displaymath}
Finally, we have the third diagram. It is
\begin{displaymath}
\tiny
\xymatrix{
\Pi^{\vert a \vert \vert c \vert}(\Pi^{\vert b \vert \vert c \vert}(c)) \Pi^{\vert a \vert \vert d \vert}(\Pi^{\vert b \vert \vert d \vert}(d)) ab \ar[d]_{\phi} \ar[rr]^{\id} &&
\Pi^{\vert a \vert \vert c \vert}(\Pi^{\vert b \vert \vert c \vert}(c)) \Pi^{\vert a \vert \vert d \vert}(\Pi^{\vert b \vert \vert d \vert}(d)) ab \ar[d]^{\psi} \\
\Pi^{\vert a \vert(\vert c \vert+\vert d \vert)}(\Pi^{\vert b \vert(\vert c \vert+\vert d \vert)}(c)) dab \ar[rd]^{\omega_2(a,b; cd) \id} &&
\Pi^{\vert c \vert(\vert a \vert+\vert b \vert)}(c) \Pi^{\vert d \vert(\vert a \vert+\vert b \vert)}(d) ab \ar[ld]_{\omega_1(ab; c,d) \xi  \xi  1  1} \\
& \Pi^{(\vert a \vert+\vert b \vert)(\vert c \vert+\vert d \vert)}(c)dab }
\end{displaymath}
(We have omitted $\otimes$ symbols in the bottom right map.) This diagram commutes if and only if \eqref{eq:condition3} holds.

\textit{Axiom 4.}
The diagram becomes
\[
  \xymatrix@C=3em{
    \Pi^{|c|(|a|+|b|)}(c) \Pi^{|a||b|}(b) a \ar@{=}[d] & \Pi^{|b||c|}(\Pi^{|a||c|}(c)) \Pi^{|a||b|}(b) a \ar[l]_-{\omega_2(b,a;c)} \ar[r] & \Pi^{|b||c|}(c) \Pi^{|a|(|b|+|c|)}(b) a \ar[d]^-{(-1)^{|a||b||c|} \xi_c \otimes \xi \otimes 1} \\
    \Pi^{(|a|+|b|)|c|}(c) \Pi^{|a||b|}(b) a \ar[r]^{\omega_2(a,b;c)^{-1}} & \Pi^{|a||c|}(\Pi^{|b||c|}(c)) \Pi^{|a||b|}(b) a &
    \Pi^{|a|(|b|+|c|)}(\Pi^{|b||c|}(c)) b a \ar[l] 
  }
\]
where the top right map is $(-1)^{|a||c|(|b|+1)} \omega_1(a;b,c) \xi_c \otimes \xi \otimes 1$ and the bottom right map is $(-1)^{|a||b|(|c|+1)} \omega_1(a;c,b)^{-1} \xi^{-1}_c \otimes \xi^{-1} \otimes 1$. Taking into account signs from composing tensor products of maps,  the right side composition is $\omega_1(a;b,c) \omega_1(a;c,b)^{-1}$ times the identity, and hence the commutativity of this diagram is equivalent to \eqref{eq:condition4}.

We thus see that $R$ and $\rho$ define a braiding if and only if $(\omega_1, \omega_2)$ is a B-factor system. In particular, we obtain braidings using the systems defined in Examples~\ref{ex:type1a} and~\ref{ex:type1b}.

\subsection{Symmetry} \label{ss:scat-sym}

Now we attempt to define a syllepsis on $R$. Fix a function $\omega_{\sharp} \colon (\bZ/q)^2 \to \bk^{\times}$. We have
\begin{displaymath}
R_{B,A}(R_{A,B}(a \otimes b)) = \Pi^{|a||b|}(a) \otimes \Pi^{|a||b|}(b).
\end{displaymath}
We define $v_{A,B} \colon R_{B,A} R_{A,B} \to \id$ to be the isomorphism $\omega_{\sharp}(a,b) \xi \otimes \xi$ on an object $a \otimes b$. We check the commutativity of the four diagrams in order. The first is
\begin{displaymath}
\xymatrix@C=6em{
\Pi^{|a||b|}(\Pi^{(|a||c|)}(a)) \Pi^{|a||b|}(b) \Pi^{|a||c|}(c) \ar[r]^-{\omega_1(a; b,c) 1 \otimes \xi \otimes \xi} \ar[d]_-{\omega_2(b,c; a) \id} \ar@{..>}[dr]^-\phi &
\Pi^{|a||b|}(\Pi^{|a||c|}(a)) \Pi^{|a|(|b|+|c|)}(b) c \ar[d]^-{\omega_2(b,c; a) \id} \\
\Pi^{|a|(|b|+|c|)}(a) \Pi^{|a||b|}(b) \Pi^{|a||c|}(c) \ar[r]^-{\omega_1(a; b,c) 1 \otimes \xi \otimes \xi} &
\Pi^{|a|(|b|+|c|)}(a) \Pi^{|a|(|b|+|c|)}(b)c }.
\end{displaymath}
This formally commutes, and so
\begin{displaymath}
\phi = \omega_1(a; b,c) \omega_2(b,c; a) 1 \otimes \xi \otimes \xi.
\end{displaymath}
The second diagram is
\begin{displaymath}
\xymatrix@C=5em{
\Pi^{|a||b|}(\Pi^{|a||c|}(a)) \Pi^{|a||b|}(b) \Pi^{|a||c|}(c) \ar[r]^-{\phi} \ar[d]_-{(-1)^{|a||b||c|} \omega_{\sharp}(a,c) \xi \otimes 1 \otimes \xi} &
\Pi^{|a|(|b|+|c|)}(a) \Pi^{|a|(|b|+|c|)}(b)c \ar[d]^-{\omega_{\sharp}(a,bc)\xi \otimes \xi \otimes 1} \\
\Pi^{|a||b|}(a) \Pi^{|a||b|}(b) c \ar[r]^-{\omega_{\sharp}(a,b) \xi \otimes \xi \otimes 1} &
abc }.
\end{displaymath}
This diagram commutes if and only if
\begin{displaymath}
\omega_1(a; b,c) \omega_2(b,c; a) \omega_{\sharp}(a,bc) = \omega_{\sharp}(a,b) \omega_{\sharp}(a,c),
\end{displaymath}
which is equivalent to \eqref{eq:condition5}. The third diagram is
\begin{displaymath}
\xymatrix@C=5em{
\Pi^{|a||c|}(a) \Pi^{|b||c|}(b) \Pi^{|a||c|}(\Pi^{|b||c|}(c)) \ar[r]^-{\omega_2(a,b; c) \id} \ar[d]_-{\omega_1(c; a,b) \xi \otimes \xi \otimes 1} \ar@{..>}[dr]^-\psi &
\Pi^{|a||c|}(a) \Pi^{|b||c|}(b) \Pi^{(|a|+|b|)|c|}(c)  \ar[d]^-{\omega_1(c; a,b) \xi \otimes \xi \otimes 1} \\
\Pi^{(|a|+|b|)|c|}(a) b \Pi^{|a||c|}(\Pi^{|b||c|}(c)) \ar[r]^-{\omega_2(a,b; c) \id} &
\Pi^{(|a|+|b|)|c|}(a) b \Pi^{(|a|+|b|)|c|}(c) }.
\end{displaymath}
This formally commutes, and we find that 
\begin{displaymath}
\psi = \omega_1(c; a,b) \omega_2(a,b; c) \xi \otimes \xi \otimes 1.
\end{displaymath}
The fourth diagram is
\begin{displaymath}
\xymatrix@C=5em{
\Pi^{|a||c|}(a) \Pi^{|b||c|}(b) \Pi^{|a||c|}(\Pi^{|b||c|}(c)) \ar[r]^-{\psi} \ar[d]_-{\omega_{\sharp}(a,c) \xi \otimes 1 \otimes \xi} &
\Pi^{(|a|+|b|)|c|}(a) b \Pi^{(|a|+|b|)|c|}(c)  \ar[d]^-{\omega_{\sharp}(ab,c) \xi \otimes 1 \otimes \xi} \\
a \Pi^{|b||c|}(b) \Pi^{|b||c|}(c) \ar[r]^-{\omega_{\sharp}(b,c) 1 \otimes \xi \otimes \xi} &
abc }.
\end{displaymath}
This diagram commutes if and only if
\begin{displaymath}
\omega_1(c; a,b) \omega_2(a,b; c) \omega_{\sharp}(ab,c)=\omega_{\sharp}(a,c) \omega_{\sharp}(b,c),
\end{displaymath}
which is equivalent to \eqref{eq:condition6}. Finally, the two isomorphisms
\begin{displaymath}
(R_{A,B}R_{B,A}R_{A,B})(ab)=\Pi^{2\vert a \vert \vert b \vert}(b) \Pi^{\vert a \vert \vert b}(a) \to \Pi^{\vert a \vert \vert b \vert}(b) a = R_{A,B}(ab)
\end{displaymath}
are
\begin{displaymath}
1 \circ v_{A,B} = \omega_{\sharp}(a,b) \xi \otimes \xi, \qquad
v_{B,A} \circ 1 = \omega_{\sharp}(b,a) \xi \otimes \xi.
\end{displaymath}
These agree if and only if $\omega_{\sharp}(a,b)=\omega_{\sharp}(b,a)$, i.e., $\gamma$ is symmetric. We thus see that $v$ defines a syllepsis if and only if $(\omega_1, \omega_2, \omega_{\sharp})$ is an S-factor system, and it defines a symmetry if and only if additionally $\omega_{\sharp}$ is symmetric.

We thus obtain symmetries from the systems in Examples~\ref{ex:type1a} and~\ref{ex:type1b}.

\begin{remark}
A variant of the above discussion shows that even factor systems can be used to endow $\Sigma$ with a braided, sylleptic, or symmetric structure. 
\end{remark}

\section{Spin symmetric groups} \label{s:spin}

\subsection{S-sets and s-groups}

An {\bf s-set} is a triple $(X, \vert \cdot \vert, c)$ where $X$ is a set, $\vert \cdot \vert \colon X \to \bZ/2$ is a function (the degree function), and $c$ is an involution of $X$ such that $\vert cx \vert=\vert x \vert$ for all $x \in X$. We write $X_i$ for the set of elements of parity $i$.

Let $X$ and $Y$ be s-sets. A {\bf homogeneous map} $X \to Y$ of degree $n$ is a function that maps $X_i$ into $Y_{i+n}$ and is compatible with $c$. We let $\Hom(X,Y)_i$ denote the set of maps of degree $i$, and let $\Hom(X,Y)=\Hom(X,Y)_0 \amalg \Hom(X,Y)_1$, which is naturally an s-set. We thus have a category $\SSet$ of s-sets. We let $\SSet^{\circ}$ be the subcategory where the morphisms are even (degree~0).

Let $X$ and $Y$ be s-sets. We define the {\bf product} of $X$ and $Y$, denoted $X \ttimes Y$, as follows. The underlying set is the quotient of the usual cartesian product $X \times Y$ by the involution $c \times c$. That is, elements of $X \ttimes Y$ are represented by ordered pairs $(x, y)$, and the pairs $(x, y)$ and $(cx, cy)$ represent the same element. The degree function is given by $\vert (x,y) \vert=\vert x \vert+\vert y \vert$. The involution $c$ is given by $c(x,y)=(cx,y)=(x,cy)$. We define an isomorphism $X \ttimes Y \to Y \ttimes X$ by $(x, y) \mapsto c^{\vert x \vert \vert y \vert} (y, x)$. The product $\ttimes$ gives $\SSet^{\circ}$ the structure of a symmetric monoidal category.

An {\bf s-group} is a group object in the monoidal category $\SSet^{\circ}$. Explicitly, an s-group is a triple $(G, \vert \cdot \vert, c)$ where $G$ is a group, $\vert \cdot \vert \colon G \to \bZ/2$ is a group homomorphism, and $c \in G$ is a central element satisfying $c^2=1$ and $\vert c \vert=0$. A {\bf homomorphism} of s-groups $f \colon G \to H$ is a group homomorphism such that $f(c)=c$ and $\vert f(g) \vert=\vert g \vert$ for all $g \in G$.

If $G$ and $H$ are s-groups then $G \ttimes H$ is naturally an s-group, with multiplication defined by $(g,h)(g',h')=c^{\vert h \vert \vert g' \vert} (gg', hh')$. Two elements $g$ and $h$ of an s-group $G$ are said to {\bf supercommute} if $gh=c^{\vert h \vert \vert g \vert} hg$. Similarly, two subsets $A$ and $B$ of $G$ are said to supercommute if every element of $A$ supercommutes with every element of $B$. If $H_1$ and $H_2$ are supercommuting s-subgroups of $G$ then there is a homomorphism $H_1 \ttimes H_2 \to G$ given by $(h_1, h_2) \mapsto h_1 h_2$.

\begin{example}
Let $X$ be an object in a supercategory, and let $\Aut^h(X)$ be the group of all homogeneous automorphisms of $X$. Then $\Aut^h(X)$ is an s-group, taking $c=-1$ and the degree map to be the usual degree map. For a super vector space $V$, we write $\GL^h(V)$ in place of $\Aut^h(V)$.
\end{example}

\subsection{Spin symmetric groups} \label{ss:spin}

Recall that the braid group $\fB_n$ is the group generated by elements $s_1, \ldots, s_{n-1}$ modulo the following two types of relations:
\begin{itemize}
\item the braid relations $s_is_{i+1}s_i=s_{i+1}s_is_{i+1}$ for $1 \le i \le n-2$; and
\item the commutativity relations $s_is_j=s_js_i$ for $\vert i-j \vert \ge 2$.
\end{itemize}
The symmetric group $\fS_n$ is isomorphic to the quotient of $\fB_n$ by the relations $s_i^2=1$; the isomorphism takes $s_i$ to the transposition $(i\;\;i+1)$.

For $\delta \in \bZ/2$, we define $\wt{\fB}^{\delta}_n$ to be the group generated by elements $\tilde{s}_1, \ldots, \tilde{s}_{n-1}, c$ modulo the following relations:
\begin{itemize}
\item the element $c$ is central and $c^2=1$;
\item the braid relations $\tilde{s}_i\tilde{s}_{i+1}\tilde{s}_i=\tilde{s}_{i+1}\tilde{s}_i\tilde{s}_{i+1}$ for $1 \le i \le n-2$; and
\item the (super)commutativity relations $\tilde{s}_i\tilde{s}_j=c^{\delta} \tilde{s}_j\tilde{s}_i$ for $\vert i-j \vert \ge 2$.
\end{itemize}
For $\epsilon \in \bZ/2$, define $\wt{\fS}^{(\delta,\epsilon)}_n$ to be the quotient of $\wt{\fB}^{\delta}_n$ by the relations $\tilde{s}_i^2=c^{\epsilon}$. It is well-known that $c \ne 1$ in $\wt{\fS}^{(\delta,\epsilon)}_n$ (for $\delta=1$ and $\epsilon=0$, this follows from Proposition~\ref{prop:spin-clifford}, for instance), and thus $c \ne 1$ in $\wt{\fB}^{\delta}_n$ as well. We therefore have central extensions
\begin{displaymath}
1 \to \langle c \rangle \to \wt{\fB}^{\delta}_n \to \fB_n \to 1.
\end{displaymath}
and
\begin{displaymath}
1 \to \langle c \rangle \to \wt{\fS}^{(\delta,\epsilon)}_n \to \fS_n \to 1
\end{displaymath}
where $\langle c \rangle$ is cyclic of order two.

We define the degree function on $\wt{\fS}^{(\delta,\epsilon)}_n$ to be the composition $\wt{\fS}^{(\delta,\epsilon)}_n \to \fS_n \stackrel{\ell}{\to} \bZ/2$, where $\ell$ is the length modulo~2 of a permutation, i.e., $0$ for even permutations and 1 for odd permutations. In this way, $\wt{\fS}^{(\delta,\epsilon)}_n$ has the structure of an s-group. We let $\wt{\fS}_n=\wt{\fS}^{(1,0)}_n$, and refer to this as the {\bf spin-symmetric group}.

\begin{remark}
For $n=0,1$, the central extension $\wt{\fS}^{(\delta,\epsilon)}_n$ is trivial for all $\epsilon$ and $\delta$, while for $n=2,3$ it is independent of $\delta$. For $n \ge 4$, the four central extensions are non-isomorphic, and the group $\rH^2(\fS_n,\bZ/2)$ of central extensions is isomorphic to $(\bZ/2)^2$ via $(\delta,\epsilon)$. 

  The calculation of $\rH^2(\fS_n, \bZ/2)$ can be done as follows. First, $\rH_1(\fS_n, \bZ)$ is the abelianization of $\fS_n$, which is just $\bZ/2$ for $n \ge 2$ and trivial otherwise. Next, $\rH_2(\fS_n, \bZ)$ is dual to $\rH^2(\fS_n, \bC^\times)$, which is $\bZ/2$ for $n \ge 4$ and trivial otherwise \cite[Theorems 2.7, 2.9]{HH}. Our earlier claim then follows from the universal coefficients theorem.
\end{remark}

\subsection{The $\tilde{\tau}$ elements} \label{ss:calculations}

Let $\tau_{n,m} \in \fS_{n+m}$ be the permutation defined by
\begin{displaymath}
\tau_{n,m}(i) = \begin{cases}
i+m & \text{if $1 \le i \le n$} \\
i-n & \text{if $n+1 \le i \le n+m$} \end{cases}.
\end{displaymath}
In a sense, $\tau_{n,m}$ induces the natural bijection of sets $[n] \amalg [m] \to [m] \amalg [n]$. Let $\sigma_{n,i} \in \fS_{n+m}$ be the $(n+1)$-cycle $(i\;\;i+1\;\;\cdots\;\;i+n)$. Then we have
\begin{displaymath}
\sigma_{n,i}=s_i s_{i+1} \cdots s_{i+n-1}, \qquad \tau_{n,m}=\sigma_{n,m} \cdots \sigma_{n,1}.
\end{displaymath}
We define
\begin{displaymath}
\tilde{\sigma}_{n,i}=\tilde{s}_i \tilde{s}_{i+1} \cdots \tilde{s}_{i+n-1}, \qquad \tilde{\tau}_{n,m}=\tilde{\sigma}_{n,m} \cdots \tilde{\sigma}_{n,1},
\end{displaymath}
as elements of $\wt{\fS}_{n+m}$. The $\tilde{\tau}_{n,m}$ elements are very important, and are essentially responsible for all examples of supersymmetric monoidal structures that we write down.

\begin{proposition} \label{prop:tau-symmetric}
We have $\tilde{\tau}_{n,m} \tilde{\tau}_{m,n}=c^N$ with $N=\binom{n}{2} \binom{m}{2}$.
\end{proposition}

\begin{proof}
For any $i,j,k$, we have
\begin{align*}
  \tilde{\sigma}_{i,k} \tilde{\sigma}_{j,i+k-1}   &= (\tilde{s}_k \tilde{s}_{k+1} \cdots \tilde{s}_{i+k-1}) (\tilde{s}_{i+k-1} \tilde{s}_{i+k} \cdots \tilde{s}_{i+j+k-2})\\
  &= c^{(i-1)(j-1)} \tilde{\sigma}_{j-1,i+k} \tilde{\sigma}_{i-1,k} 
\end{align*}
where we have canceled the $\tilde{s}_{i+k-1}$ and commuted the resulting $\tilde{\sigma}_{i-1,k}$ past $\tilde{\sigma}_{j-1,i+k}$. By applying this relation successively, this implies that
\[
  \tilde{\sigma}_{i,k} \tilde{\sigma}_{j,i+k-1} \tilde{\sigma}_{j,i+k-2} \cdots \tilde{\sigma}_{j,i} = c^{\binom{i}{2} (j-1)} \tilde{\sigma}_{j-1,i+k} \tilde{\sigma}_{j-1,i+k-1} \cdots \tilde{\sigma}_{j-1,i+1}.
\]
Finally, we successively use this relation to get
\begin{align*}
  \tilde{\tau}_{n,m} \tilde{\tau}_{m,n} &= (\tilde{\sigma}_{n,m} \tilde{\sigma}_{n,m-1} \cdots \tilde{\sigma}_{n,1})( \tilde{\sigma}_{m,n} \tilde{\sigma}_{m,n-1} \cdots \tilde{\sigma}_{m,1})\\
                                        &= c^{(m-1) \binom{n}{2}} (\tilde{\sigma}_{n,m} \tilde{\sigma}_{n,m-1} \cdots \tilde{\sigma}_{n,2})( \tilde{\sigma}_{m-1,n+1} \tilde{\sigma}_{m-1,n} \cdots \tilde{\sigma}_{m-1,1})\\
                                        & \qquad \vdots\\
                                          &= c^{\binom{m}{2} \binom{n}{2} }. \qedhere
\end{align*}
\end{proof}

There is a natural injective homomorphism of s-groups
\begin{align} \label{eqn:spin-emb}
j_{n,m} \colon \wt{\fS}_n \ttimes \wt{\fS}_m \to \wt{\fS}_{n+m}
\end{align}
defined by $(\tilde{s}_i, 1) \mapsto \tilde{s}_i$ and $(1, \tilde{s}_i) \mapsto \tilde{s}_{n+i}$. (To see that this map is injective, note that it is injective modulo $\langle c \rangle$, and so the kernel is contained in $\langle c \rangle$, but $c$ itself is not in the kernel.)

\begin{proposition} \label{prop:spin-conj}
For $g \in \wt{\fS}_n$ and $h \in \wt{\fS}_m$, we have
\begin{displaymath}
\tilde{\tau}_{n,m} j_{n,m}(g,h) \tilde{\tau}_{n,m}^{-1} = c^{nm|g| + nm|h| + |g||h|} j_{m,n}(h,g).
\end{displaymath}
\end{proposition}

\begin{proof}
  First consider the case $h=1$, $m=1$, and $g = \tilde{s}_i$ where $1 \le i \le n-1$. Then $\tilde{\tau}_{n,1} = \tilde{\sigma}_{n,1} = \tilde{s}_1 \tilde{s}_2 \cdots \tilde{s}_n$ and hence
  \begin{align*}
    \tilde{\tau}_{n,1} \tilde{s}_i \tilde{\tau}_{n,1}^{-1} &= \tilde{s}_1 \tilde{s}_2 \cdots \tilde{s}_n \tilde{s}_i \tilde{s}_n \tilde{s}_{n-1} \cdots \tilde{s}_1\\
                                                           &= c^{n-i-1} \tilde{s}_1 \tilde{s}_2 \cdots (\tilde{s}_{i+1} \tilde{s}_i \tilde{s}_{i+1}) \cdots \tilde{s}_1\\
    &= c^{n-2} \tilde{s}_{i+1} = c^n \tilde{s}_{i+1}
  \end{align*}
  where in the second equality, we used the commuting relations $\tilde{s}_j \tilde{s}_i = c \tilde{s}_i \tilde{s_j}$ for $j=n,n-1,\dots,i+2$, and in the third equality, we used the braid relation $\tilde{s}_{i+1} \tilde{s}_i \tilde{s}_{i+1} = \tilde{s}_i\tilde{s}_{i+1} \tilde{s}_i$ and then again the commuting relations for $j=i-1,i-2,\dots,1$. Since conjugation is a group homomorphism, it follows immediately that $\tilde{\tau}_{n,1}j_{n,m}(g,1)\tilde{\tau}^{-1}_{n,1} = c^{n |g|} j_{m,n}(1,g)$ for any $g \in \wt{\fS}_n$ and hence $\tilde{\tau}_{n,m} j_{n,m}(g,1) \tilde{\tau}^{-1}_{n,m} = c^{nm|g|} j_{m,n}(1,g)$. Similarly, $\tilde{\tau}_{n,m}j_{n,m}(1,h)\tilde{\tau}^{-1}_{n,m} = c^{nm |h|} j_{m,n}(h,1)$. Finally, we have
\begin{align*}
\tilde{\tau}_{n,m}j_{n,m}(g,h)\tilde{\tau}^{-1}_{n,m}
&= \tilde{\tau}_{n,m}j_{n,m}(g,1) j_{n,m}(1,h)\tilde{\tau}^{-1}_{n,m} \\
&= c^{nm|g|+nm|h|} j_{m,n}(1,g) j_{m,n}(h,1) \\
&= c^{nm|g|+nm|h|+|g||h|} j_{m,n}(h,g). \qedhere
\end{align*}
\end{proof}

\begin{proposition} \label{prop:tauj}
For any $m,n,p$, we have $j_{n,m+p}(1,\tilde{\tau}_{m,p}) j_{n+m,p}(\tilde{\tau}_{m,n},1)=\tilde{\tau}_{m,n+p}$.
\end{proposition}

\begin{proof}
By definition, we have $j_{n,m+p}(1,\tilde{\sigma}_{m,i}) = \tilde{\sigma}_{m,n+i}$, so 
\begin{align*}
j_{n,m+p}(1,\tilde{\tau}_{m,p}) j_{n+m,p}(\tilde{\tau}_{m,n},1)
&= j_{n,m+p}(1, \tilde{\sigma}_{m,p} \cdots \tilde{\sigma}_{m,1}) j_{n+m,p}(\tilde{\sigma}_{m,n} \cdots \wt{\sigma}_{m,1}, 1) \\
&= \tilde{\sigma}_{m,n+p} \cdots \tilde{\sigma}_{m,n+1} \tilde{\sigma}_{m,n} \cdots \tilde{\sigma}_{m,1}
=\tilde{\tau}_{m,n+p}. \qedhere
\end{align*}
\end{proof}

\subsection{The category $\tilde{\cS}$} \label{ss:catS}

We define a homogeneous $\bZ$-supercategory $\tilde{\cS}$ as follows. There is a single homogeneous object $[n]$ of degree $n$, for each $n \ge 0$, and no homogeneous objects of negative degree. We put $\End_{\tilde{\cS}}([n])=\bk[\wt{\fS}_n]/(c+1)$ and $\Hom_{\tilde{\cS}}([n], [m])=0$ for $n \ne m$. We now define a monoidal structure $\otimes$ on $\tilde{\cS}$. On objects, we define $[n] \otimes [m]=[n+m]$. We define
\begin{displaymath}
\otimes \colon \End_{\tilde{\cS}}([n]) \otimes \End_{\tilde{\cS}}([m]) \to \End_{\tilde{\cS}}([n+m])
\end{displaymath}
to be the linear map induced by $j_{n,m}$. The unit object is $[0]$. The monoidal structure is strict (e.g., the associator maps are the identities). One readily verifies that this satisfies the requisite axioms for a monoidal $\bZ$-supercategory.

For $n,m \in \bN$, define $\beta_{n,m} \colon [n] \otimes [m] \to [m] \otimes [n]$ to be $\tilde{\tau}_{n,m}$.

\begin{proposition}
$\beta$ defines a type~IIa supersymmetry on $\tilde{\cS}$.
\end{proposition}

\begin{proof}
Proposition~\ref{prop:spin-conj} exactly shows that $\beta$ is a natural transformation (see the first paragraph of \S \ref{ss:type2}). Proposition~\ref{prop:tauj} exactly shows that the hexagon (H1) commutes. Proposition~\ref{prop:tau-symmetric} shows that $\beta$ is symmetric, i.e., $\beta_{m,n}\beta_{n,m}=(-1)^k$ with $k=\binom{n}{2} \binom{m}{2}$, which also implies the second hexagon axiom. This completes the proof.
\end{proof}

Now suppose that $\zeta_{16} \in \bk$. For $n \in \bZ$, define $\epsilon(n)=-1$ if $n \equiv 3 \pmod{4}$ and $\epsilon(n)=1$ otherwise. Define $\beta'_{n,m} \colon [n+m] \to [n+m]$ to be the map $\epsilon(n)^m \zeta_4^{-\binom{n}{2} m} \zeta_{16}^{nm} \tilde{\tau}_{n,m}$.

\begin{proposition}
$\beta'$ defines a type~IIb supersymmetry on $\tilde{\cS}$.
\end{proposition}

\begin{proof}
This follows from Proposition~\ref{prop:type2-ab}.
\end{proof}

Let $q$ be a positive integer. We can then regard $\tilde{\cS}$ as $\bZ/q$-graded, and the monoidal structure carries over. If $q$ is divisible by~4 then $\beta$ defines a type~IIa supersymmetry, while if $q$ is any even integer then $\beta'$ defines a type~IIb supersymmetry (the key point here is that the scalar $\beta'_{n,m} \beta'_{m,n}$ only depends on the parities of $n$ and $m$; this is the purpose of the somewhat complicated scalar factor in the definition of $\beta'$). 

\begin{remark}
In the category $\tilde{\cS}$, the object $[n]$ admits an odd degree automorphism if and only if $n \ge 2$. Thus $\tilde{\cS}$ does not admit a $\Pi$-structure, since the objects $[0]$ and $[1]$ do not.
\end{remark}

\begin{remark}
In our discussion of supersymmetric monoidal categories, we have worked exclusively with linear categories since the examples we care about are linear. However, this is not necessary. One can define a ``non-linear supercategory'' to be a category enriched in $\SSet^{\circ}$, and then go on to define non-linear analogs of many of the other notions we have discussed. The category $\tilde{\cS}$ defined above is in fact the linearization of what we would call the groupoid of spin-sets, which is likely an object of fundamental importance.
\end{remark}

\subsection{Actions of spin-symmetric groups} \label{ss:actions}

Let $\cA$ be a type~IIa supersymmetric monoidal $\bZ/q$-supercategory with $4 \mid q$. In what follows, we assume that $\cA$ is strict (i.e., the associativity isomorphisms are identity maps) for simplicity. (In fact, this comes at no loss of generality due to a version of Mac Lane's coherence theorem in this setting; see the discussion in \cite{brundan} following Definition~1.4.) Let $X$ be a homogeneous object of degree $p \in \bZ/q$ and fix a positive integer $n$. For $1 \le i \le n-1$, define $\sigma_i \colon X^{\otimes n} \to X^{\otimes n}$ to be the map
\begin{displaymath}
X^{\otimes n} =
X^{\otimes (i-1)} \otimes X^{\otimes 2} \otimes X^{\otimes (n-i-1)} \xrightarrow{\id \otimes \beta_{X,X} \otimes \id}
X^{\otimes (i-1)} \otimes X^{\otimes 2} \otimes X^{\otimes (n-i-1)} =
X^{\otimes n}
\end{displaymath}
Note that $\sigma_i$ has degree $p$ (since $p^2=p$ modulo~2). We then have:

\begin{proposition}
The automorphisms $\sigma$ satisfy the following two relations:
\begin{enumerate}[\rm \indent (a)]
\item $\sigma_i \sigma_j=(-1)^p \sigma_j \sigma_i$ for $\vert i-j \vert>1$.
\item $\sigma_i \sigma_{i+1} \sigma_i=(-1)^p \sigma_{i+1} \sigma_i \sigma_{i+1}$.
\item $\sigma_i^2 = (-1)^{\binom{p}{2}}$.
\end{enumerate}
\end{proposition}

\begin{proof}
Relation (a) follows immediately from how composition works in product categories and the fact that each $\sigma_i$ has degree $p$. We now verify (b). It suffices to treat the case $n=3$. We must show that the following diagram commutes up to $(-1)^p$:
\begin{displaymath}
\xymatrix@C=4em{
  X \otimes X \otimes X \ar[d]_{\sigma_1} \ar[r]^{\sigma_2} & X \otimes X \otimes X \ar[r]^{\sigma_1} & X \otimes X \otimes X \ar[d]^{\sigma_2}\\
  X \otimes X \otimes X \ar[r]^{\sigma_2} & X \otimes X \otimes X \ar[r]^{\sigma_1} & X \otimes X \otimes X 
}
\end{displaymath}
By the first hexagon axiom, $\sigma_1 \sigma_2=\beta_{X \otimes X, X}$. It thus suffices to show that the following diagram commutes up to $(-1)^p$:
\begin{displaymath}
\xymatrix@C=5em{
  X \otimes X \otimes X \ar[d]_{\beta_{X,X} \otimes \id_X} \ar[r]^{\beta_{X \otimes X, X}} & X \otimes X \otimes X \ar[d]^{\id_X \otimes \beta_{X,X}} \\
  X \otimes X \otimes X \ar[r]^{\beta_{X \otimes X, X}} & X \otimes X \otimes X }
\end{displaymath}
This follows from the definition of $\beta$ (see \S \ref{ss:type2}), and the fact that $\vert X \vert=\vert \beta_{X,X} \vert=p$ and $\vert \id_X \vert=0$. Finally, (c) follows from the definition of supersymmetry.
\end{proof}

We introduce the following notation:
\begin{displaymath}
\wt{\fS}^{0}_n = \wt{\fS}^{(0,0)}_n, \quad
\wt{\fS}^{1}_n = \wt{\fS}^{(1,0)}_n, \quad
\wt{\fS}^{2}_n = \wt{\fS}^{(0,1)}_n, \quad
\wt{\fS}^{3}_n = \wt{\fS}^{(1,1)}_n.
\end{displaymath}
The index $p$ in $\wt{\fS}^p_n$ is regarded as an element of $\bZ/4$. The above proposition shows that $\wt{\fS}^p_n$ acts on $X^{\otimes n}$ if we let the generator $\tilde{s}_i$ act by $(-1)^{ip}$ and $c$ act by $(-1)^p$. In fact:

\begin{proposition} \label{cor:spin-sym-action}
The group $\wt{\fS}^p_n$ naturally acts on the functor $\cA_p \to \cA_{pn}$ given by $X \mapsto X^{\otimes n}$.
\end{proposition}

\begin{proof}
We just need to verify naturality. Given a homogeneous morphism $f \colon X \to Y$ between homogeneous objects, we get a morphism $f^{\otimes n} \colon X^{\otimes n} \to Y^{\otimes n}$, and we check that it supercommutes with the action of $\sigma_i$. We have
  \begin{align*}
    \sigma_i f^{\otimes n} &= (1 \otimes \beta_{Y,Y} \otimes 1) (f^{\otimes i} \otimes f^{\otimes 2} \otimes f^{\otimes (i-2)})\\
                           &= (-1)^{i|f||Y|} (f^{\otimes i} \otimes \beta_{Y,Y} f^{\otimes 2} \otimes f^{\otimes (i-2)})\\
                           &= (-1)^{i|f||Y| + |f|}(f^{\otimes i} \otimes f^{\otimes 2} \beta_{X,X} \otimes f^{\otimes (i-2)})\\
                           &= (-1)^{|f|}(f^{\otimes i} \otimes f^{\otimes 2} \otimes f^{\otimes (i-2)})(1 \otimes \beta_{X,X} \otimes 1)\\
    &= (-1)^{|f|} f^{\otimes n} \sigma_i. \qedhere
\end{align*}
\end{proof}

\begin{remark}
There are similar results to the above in the superbraiding case and the type~IIb case.
\end{remark}

\subsection{Universality of $\tilde{\cS}$}

Fix an integer $q$ divisible by~4 (we allow $q=0$). In this section, we regard the category $\tilde{\cS}$ as a type~IIa supersymmetric monoidal $\bZ/q$-supercategory via the supersymmetry $\beta$ defined in \S \ref{ss:catS}. The following theorem shows that $\tilde{\cS}$ is the universal supersymmetric monoidal $\bZ/q$-supercategory equipped with an object of degree~1:

\begin{theorem} \label{thm:Suniv}
Let $\cA$ be a type~IIa supersymmetric monoidal $\bZ/q$-supercategory. Let $\cB$ be the category of supersymmetric monoidal $\bZ/q$-superfunctors $\tilde{\cS} \to \cA$. Then the functor $\Phi \colon \cB \to \cA^{\circ}_1$ given by $\Phi(F)=F([1])$ is an equivalence of $\bk$-linear categories.
\end{theorem}

\begin{proof}
Suppose that $X$ is a degree~1 object of $\cA$. We have seen that $\wt{\fS}_n$ naturally acts on $X^{\otimes n}$, and so we can build a functor $\Psi_X \colon \tilde{\cS} \to \cA$ by $\Psi_X([n])=X^{\otimes n}$, which is easily seen to be a supersymmetric monoidal $\bZ/q$-functor. The assignment $X \mapsto \Psi_X$ is functorial by Corollary~\ref{cor:spin-sym-action}. The resulting functor $\Psi \colon \cA^{\circ}_1 \to \cB$ is easily seen to be quasi-inverse to $\Phi$.
\end{proof}

\begin{remark}
Given $p \in \bZ/q$, one can build a supersymmetric monoidal $\bZ/q$-supercategory $\tilde{\cS}^p$ using the $\wt{\fS}^p$ groups, specializing to $\tilde{\cS}$ when $p=1$. The category $\tilde{\cS}^p$ is the universal supersymmetric monoidal $\bZ/q$-category with an object of degree $p$.
\end{remark}

\subsection{Type IIa$'$ supersymmetries} \label{ss:abcoh}

Let $\cA$ be a type~IIa monoidal $\bZ/4$-supercategory. We have seen that $\fS^p_n$ acts on $X^{\otimes n}$ when $X$ is homogeneous of degree $p \in \bZ/4$. This may seem strange since it identifies $\rH^2(\fS_n, \bZ/2)$ with $\bZ/4$ (as sets). In fact, this is due to a choice we made, namely, the choice of factor system.

To make this clear, define a type~IIa$'$ supersymmetry to be a type~II supersymmetry with respect to $\fsA \fsC \fsD$, where $\fsA$, $\fsC$, and $\fsD$ are as in Examples~\ref{ex:theta}, \ref{ex:eta}, and \ref{ex:kappa}. Suppose $\cA$ is a type~IIa$'$ supersymmetric monoidal $\bZ/4$-category and $X$ is an object of degree $p \in \bZ/4$. Then one can show that $\wt{\fS}^{-p}_n$ acts on $X^{\otimes n}$.

We note that if $\zeta_4 \in \bk$ then one can pass between type~IIa and type~IIa$'$ superbraidings: indeed, if $\beta$ is a type~IIa superbraiding then putting $\beta'_{X,Y}=\zeta_4^{\vert X \vert \vert Y \vert} \beta_{X,Y}$, one finds that $\beta'$ is a type~IIa$'$ superbraiding. (The $\zeta_4$ factor in this expression realizes $\fsC \fsD$ as a coboundary.)

\section{Linear spin-species} \label{s:spinrep}

\subsection{S-representations} \label{ss:trirep}

Let $G$ be an s-group. An {\bf s-representation} of $G$ on a superspace $V$ is a homomorphism of s-groups $\rho \colon G \to \GL^h(V)$. More explicitly, an s-representation $\rho$ associates to each $g \in G$ an automorphism $\rho(g)$ of $V$ such that
\begin{enumerate}
\item $\rho(c)=-1$,
\item $\rho(g)$ is homogeneous of degree $\deg(g)$, and 
\item $\rho(gh)=\rho(g) \rho(h)$.
\end{enumerate}
Note that an s-representation is simply a representation of $G$ on $V$ in the usual sense satisfying certain conditions. A {\bf morphism} of s-representations $f \colon V \to W$ of parity $p$ is a map of super vector spaces of parity $p$ such that $f(gv)=(-1)^{p \vert g \vert} gf(v)$ for all $v \in V$. We thus have an $\SVec$-category $\Rep(G)$ of representations of $G$, which is abelian and admits a $\Pi$-structure (as we see below).

Let $G$ and $H$ be s-groups and let $V$ and $W$ be s-representations of them. We define an s-representation of $G \ttimes H$ on $V \otimes W$ by the formula
\begin{displaymath}
(g, h)(v \otimes w)=(-1)^{\vert v \vert \vert h \vert} gv \otimes hw.
\end{displaymath}
Let $j \colon G \ttimes H \to H \ttimes G$ be the usual isomorphism of s-groups, and regard $W \otimes V$ as an s-representation of $G \ttimes H$ via $j$; explicitly,
\begin{displaymath}
(g,h)(w \otimes v)=(-1)^{\vert g \vert \vert h \vert+\vert w \vert \vert g \vert} hw \otimes gv.
\end{displaymath}
Then the usual isomorphism $\tau \colon V \otimes W \to W \otimes V$ (as defined in \S \ref{ss:super-vector-spaces}) is an isomorphism of $G \ttimes H$ s-representations.

\begin{remark}
  Note that when $V$ and $W$ are two s-representations of $G$ there is no natural s-representation of $G$ on $V \otimes W$, due to the lack of a diagonal morphism $G \to G \ttimes G$. However, one can regard $V \otimes W$ as a representation of the ordinary group $G$ (or even $G/\langle c \rangle$) by $g (v \otimes w) = gv \otimes gw$. Note that there is no sign here, even if $g$ and $v$ are both of odd degree. This action of $G/\langle c \rangle$ is by even transformations.
\end{remark}

\subsection{$\Pi$-structures}

Let $V$ be a super vector space, and let $\pi \in \bk[1]$ be the element $1 \in \bk$. There are (at least) three natural choices of $\Pi$-structure on $V$:
\begin{itemize}
\item Let $\Pi_n(V)=V[1]$ and let $\xi_n \colon V \to \Pi_n(V)$ be the map $v \mapsto v$.
\item Let $\Pi_\ell(V)=\bk[1] \otimes V$, and let $\xi_\ell \colon V \to \Pi_\ell(V)$ be the map $v \mapsto \pi \otimes v$.
\item Let $\Pi_r(V)=V \otimes \bk[1]$, and let $\xi_r \colon V \to \Pi_r(V)$ be the map $v \mapsto (-1)^{\vert v \vert} v \otimes \pi$.
\end{itemize}
We refer to $\Pi_{\ell}$ and $\Pi_r$ as the ``left'' and ``right'' $\Pi$-structures, and $\Pi_n$ as the ``neutral'' $\Pi$-structure (since it does not favor left or right). The neutral $\Pi$-structure may seem to be the most natural choice, but there is a practical issue with it: if $v$ is an element of $V$ then it is also an element of $V[1]$, and thus the symbol $v$ becomes ambiguous. The reason for the sign in the right $\Pi$-structure is so that the following diagram commutes:
\begin{displaymath}
\xymatrix{
V \ar@{=}[r] \ar[d]_{\xi_r} & V \otimes \bk \ar[d]^{\id \otimes \xi_r} \\
\Pi_r(V) \ar@{=}[r] & V \otimes \Pi_r(\bk) }
\end{displaymath}
Here $\xi_r \colon \bk \to \Pi_r(\bk)=\bk[1]$ is the map $1 \mapsto \pi$. Since this is an odd map, we have $(\id \otimes \xi_r)(v \otimes 1) = (-1)^{\vert v \vert} v \otimes \pi$. The sign in $\xi_r$ also ensures that the canonical even degree isomorphism $\Pi_{\ell}(V) \to \Pi_r(V)$ is just the symmetry $\tau$ of the tensor product.

Suppose that $V$ is an s-representation of an s-group $G$. We define s-representations of $G$ on $\Pi_{\ell}(V)=\bk[1] \otimes V$ and $\Pi_r(V)=V \otimes \bk[1]$ by $g (\pi \otimes v) = (-1)^{\vert g \vert} \pi \otimes gv$ and $g (v \otimes \pi)=gv \otimes \pi$. One readily verifies that the maps $\xi_r \colon V \to \Pi_r(V)$ and $\xi_{\ell} \colon V \to \Pi_{\ell}(V)$ are odd degree isomorphisms of s-representations. Furthermore, the canonical even degree isomorphism $\Pi_{\ell}(V) \to \Pi_r(V)$ is one of s-representations. (We do not bother defining $\Pi_n$ for s-representations.) We thus have two $\Pi$-structures for s-representations: $\Pi_r$ and $\Pi_{\ell}$.

\subsection{Induction} \label{ss:tri-ind}

Let $H \subseteq G$ be s-groups and let $V$ be an s-representation of $H$. Consider the induced representation $W=\Ind_H^G(V)=\bk[G] \otimes_{\bk[H]} V$. We give $W$ the structure of a super vector space by declaring $g \otimes v$ to be homogeneous of degree $\vert g \vert + \vert v \vert$ whenever $v \in V$ is homogeneous. In this way, $W$ is an s-representation of $G$. As usual, induction is the left adjoint of restriction.

Consider the following situation:
\begin{itemize}
\item $H_1$ and $H_2$ are s-subgroups of $G$.
\item $V_1$ (resp.\ $V_2$) is an s-representation of $H_1$ (resp.\ $H_2$).
\item $\tau$ is an element of $G$ such that $\tau H_1 \tau^{-1}=H_2$.
\item $i \colon V_1 \to V_2^{\tau}$ is an even-degree isomorphism of s-representations of $H_1$, where $V_2^{\tau}$ is the s-representation of $H_1$ on $V_2$ defined by $h \cdot v=(\tau h \tau^{-1}) v$. In other words, we have $i(hv)=(\tau h \tau^{-1}) i(v)$ for $h \in H_1$.
\end{itemize}
Define 
\begin{displaymath}
\phi \colon \Ind_{H_1}^G(V_1) \to \Pi_{r}^{\vert \tau \vert} \Ind_{H_2}^G(V_2), \qquad
\phi(g \otimes v) = g \tau^{-1} \otimes i(v) \otimes \pi^{\vert \tau \vert}
\end{displaymath}
We then have:

\begin{proposition}
The map $\phi$ is an even degree isomorphism of s-representations of $G$. 
\end{proposition}

\begin{proof}
We first check that $\phi$ is well-defined, i.e., that $\phi(gh \otimes v)=\phi(g \otimes hv)$ for $h \in H_1$. Omitting the factors of $\pi$ for space, we have
\begin{align*}
\phi(gh \otimes v)
  &= gh\tau^{-1} \otimes i(v) 
 = g \tau^{-1} \otimes \tau h \tau^{-1} i(v) 
  = g \tau^{-1} \otimes i(hv)
  = \phi(g \otimes hv).
\end{align*}
Thus $\phi$ is well-defined. For $g' \in G$ we have
\begin{displaymath}
  g' \phi(g \otimes v) = g'(g \tau^{-1} \otimes i(v) \otimes \pi^{\vert \tau \vert}) = \phi(g'g \otimes v),
\end{displaymath}
and so $\phi$ is a map of s-representations. Finally, interchanging the roles of $\tau$ and $\tau^{-1}$ gives a map in the opposite direction that is the inverse of $\phi$. Thus $\phi$ is an isomorphism.
\end{proof}

\subsection{The category of linear spin-species}

A {\bf representation of $\wt{\fS}_{\ast}$} is a sequence $(V_n)_{n \ge 0}$, where $V_n$ is an s-representation of the s-group $\wt{\fS}_n$ over $\bk$. A {\bf morphism} of $\wt{\fS}_{\ast}$-representations $V \to W$ is a sequence $f=(f_n)_{n \ge 0}$ where $f_n \colon V_n \to W_n$ is a (not necessarily homogeneous) morphism of $\wt{\fS}_n$-representations. We say that an $\wt{\fS}_{\ast}$-representation $V$ is {\bf even} (resp.\ {\bf odd}) if $V_n=0$ for all odd (resp.\ even) values of $n$. This gives $\Rep(\wt{\fS}_{\ast})$ the structure of a supercategory. For an integer $q$, we can regard $\Rep(\wt{\fS}_{\ast})$ to be an (inhomogeneous) $\bZ/q$-supercategory by declaring an object $(V_n)_{n \ge 0}$ to have degree $d \in \bZ/q$ if $V_n=0$ for all $n \not\equiv d \pmod{q}$. Every object decomposes as a direct sum of homogeneous objects, and this decomposition is finite if $q \ne 0$.

\subsection{Monoidal structure} \label{ss:spin-monoidal}

Let $n$ and $m$ be non-negative integers. Suppose that $V$ and $W$ are two representations of $\wt{\fS}_{\ast}$. We define $V \otimes W$ to be the representation of $\wt{\fS}_{\ast}$ given by
\begin{displaymath}
(V \otimes W)_n = \bigoplus_{i+j=n} \Ind_{\wt{\fS}_i \ttimes \wt{\fS}_j}^{\wt{\fS}_n} (V_i \otimes W_j),
\end{displaymath}
where here $\wt{\fS}_i \ttimes \wt{\fS}_j$ is regarded as a subgroup of $\wt{\fS}_n$ via the map $j_{n,m}$ defined in \eqref{eqn:spin-emb}. Given a third $\wt{\fS}_{\ast}$-representation $U$, we have a natural identification
\begin{displaymath}
((U \otimes V) \otimes W)_n = \bigoplus_{i+j+k=n} \Ind_{\wt{\fS}_i \ttimes \wt{\fS}_j \ttimes \wt{\fS}_k}^{\wt{\fS}_n}(U_i \otimes V_j \otimes W_k),
\end{displaymath}
and similarly for $U \otimes (V \otimes W)$. This yields a natural associativity isomorphism $(U \otimes V) \otimes W \to U \otimes (V \otimes W)$ that satisfies the pentagon axiom. The $\wt{\fS}_{\ast}$-representation that is $\bk$ in degree~0 and~0 in other degrees is naturally a unit object for $\otimes$. Thus $\Rep(\wt{\fS}_{\ast})$ has the structure of a monoidal $\bZ/q$-supercategory.

\subsection{Supersymmetry} \label{ss:spin-symmetry}

In this section, we assume $q$ is divisible by~4 and construct a supersymmetry on $\Rep(\wt{\fS}_*)$. We begin with the following observation:

\begin{proposition}
Let $V$ and $W$ be s-representations of $\wt{\fS}_n$ and $\wt{\fS}_m$. Regard $\wt{\fS}_n \ttimes \wt{\fS}_m$ as a subgroup of $\wt{\fS}_{n+m}$ via $j_{n,m}$, and similarly with $n$ and $m$ reversed. Then 
\begin{align*}
  i_{V,W} \colon V \otimes W &\to (W \otimes V)^{\tilde{\tau}_{n,m}} \\
  v \otimes w &\mapsto (-1)^{|v||w| + nm|v| + nm|w|} w \otimes v
\end{align*}
is an isomorphism of representations of $\wt{\fS}_n \ttimes \wt{\fS}_m$, where the superscript denotes the conjugate representation. Precisely, this means
  \[
    i_{V,W}(j_{n,m}(g,h) (v\otimes w)) = \tilde{\tau}_{n,m} j_{n,m}(g,h) \tilde{\tau}_{n,m}^{-1} i_{V,W}(v \otimes w)
  \]
  for all $v \in V$ and $w \in W$ and all $g \in \wt{\fS}_n$ and $h \in \wt{\fS}_m$.
\end{proposition}

\begin{proof}
We have
\begin{align*}
i_{V,W}(j_{n,m}(g,h) (v \otimes w)) &= (-1)^{|h||v|} i_{V,W}(gv \otimes hw)\\
&= (-1)^{|h||v| + (|g|+|v|)(|h|+|w|) + nm(|g|+|v|) + nm(|h|+|w|)} hw \otimes gv \\
&= (-1)^{|g||h| + nm|g| + nm|h| + |v||w| + nm|v| + nm|w| + |g||w|}  hw \otimes gv\\
&= \tilde{\tau}_{n,m} j_{n,m}(g,h) \tilde{\tau}_{n,m}^{-1} i_{V,W}(v \otimes w).
\end{align*}
In the last equality, we used Proposition~\ref{prop:spin-conj}, and that $c$ acts by $-1$.
\end{proof}

As discussed in \S \ref{ss:tri-ind}, we have an isomorphism
\begin{equation} \label{eq:ind-iso}
\Ind_{\wt{\fS}_n \ttimes \wt{\fS}_m}^{\wt{\fS}_{n+m}} (V \otimes W) \to \Pi_r^{nm} \Ind_{\wt{\fS}_m \ttimes \wt{\fS}_n}^{\wt{\fS}_{n+m}} (W \otimes V)
\end{equation}
coming from conjugation by $\tilde{\tau}_{n,m}$, and the isomorphism $i_{V,W}$ (maintaining the notation from the proposition).

\begin{definition}
Given homogeneous objects $V$ and $W$ of $\Rep(\wt{\fS}_*)$, we define
\begin{displaymath}
\beta_{V,W} \colon V \otimes W \to \Pi_{r}^{|V||W|}(W \otimes V)
\end{displaymath}
to be the isomorphism induced by the isomorphisms \eqref{eq:ind-iso} using the elements $\tilde{\tau}_{n,m}$ and the isomorphisms $i$. Explicitly, for $v \in V_n$ and $w \in W_m$ and $g \in \wt{\fS}_{n+m}$, we have
\begin{displaymath}
\beta_{V,W}(g \otimes v \otimes w) = (-1)^{nm(\vert v \vert+\vert w \vert)+\vert v \vert \vert w \vert} g\tilde{\tau}_{n,m}^{-1} \otimes w \otimes v \otimes \pi^{nm},
\end{displaymath}
where here $g \otimes v \otimes w$ denotes an element of $\Ind_{\wt{\fS}_n \ttimes \wt{\fS}_m}^{\wt{\fS}_{n+m}}(V \otimes W)$.
\end{definition}

\begin{proposition} \label{prop:spin-spec-braid}
$\beta$ defines a type~Ia supersymmetry on $\otimes$.
\end{proposition}

\begin{proof}
  First, we check that $\beta$ is a natural transformation. Let $f \colon V_n \to V'_n$ and $g \colon W_m \to W'_m$ be nonzero homogeneous morphisms. Pick $v \in V_n$, $w \in W_m$ and $\sigma \in \wt{\fS}_{n+m}$. Then we have
  \begin{align*}
    &\qquad \Pi^{|V||W|}(g \otimes f)\beta_{V_n,W_m}(\sigma \otimes v \otimes w)\\
    &=(-1)^{nm(\vert v \vert+\vert w \vert)+\vert v \vert \vert w \vert + |f|(|\sigma|+nm+|w|) + |g|(|\sigma|+nm)} \sigma \tilde{\tau}_{n,m}^{-1} \otimes g(w) \otimes f(v) \otimes \pi^{nm}.
  \end{align*}
  On the other hand, we have
  \begin{align*}
    &\qquad \beta_{V'_n,W'_m}(f \otimes g)(\sigma \otimes v \otimes w) \\
    &= \beta_{V'_n,W'_m} ((-1)^{|g|(|\sigma|+|v|)+ |f||\sigma|} \sigma \otimes f(v) \otimes g(w))\\
    &= (-1)^{|g|(|\sigma|+|v|)+ |f||\sigma| + nm(|f|+|v|+|g|+|w|) + (|f|+|v|)(|g|+|w|)} \sigma \tilde{\tau}_{n,m}^{-1} \otimes g(w) \otimes f(v) \otimes \pi^{nm}.
  \end{align*}
  This agrees with the first expression up to $(-1)^{|f||g|}$, which is what we wanted to show.
  
Next, it follows from Proposition~\ref{prop:tau-symmetric} that $\beta_{V,W} \beta_{W,V} = (-1)^{\binom{n}{2}\binom{m}{2}}$. Hence we only need to check that one of the hexagon axioms hold, and we check (H1$'$). So pick non-negative integers $m,n,p$ and let $U,V,W$ be representations of $\wt{\fS}_m$, $\wt{\fS}_n$, and $\wt{\fS}_p$, respectively. We need to check that $(1_V \otimes \beta_{U,W})(\beta_{U,V} \otimes 1_W) = \beta_{U,V \otimes W}$. (Here we identify $U$ with the $\fS_*$-representation that is $U$ in degree $m$ and~0 in other degrees, and similarly for $V$ and $W$. We note that it suffices to check the axiom on such objects.)
 
The upper path in the hexagon is the following composition:
\begin{align*}
& \Ind_{m+n,p}^{m+n+p}(\Ind_{m,n}^{m+n}(U \otimes V) \otimes W) \\
\xrightarrow{\alpha_{U,V,W}^{-1}} & \Ind_{m,n+p}^{m+n+p}(U \otimes \Ind_{n,p}^{n+p}(V \otimes W)) \\
\xrightarrow{\beta_{U,V \otimes W}} & \Ind_{n+p,m}^{m+n+p}(\Ind_{n,p}^{n+p}(V \otimes W) \otimes U) \otimes \bk[m(n+p)] \\
\xrightarrow{\alpha_{V,W,U}^{-1}} & \Ind_{n,m+p}^{m+n+p}(V \otimes \Ind^{p+m}_{p,m}(W \otimes U)) \otimes \bk[m(n+p)]
\end{align*}
Here we have written $n$ in place of $\wt{\fS}_n$ in the inductions for readability. We compute the composition explicitly. Let $g \in \wt{\fS}_{m+n+p}$ and let $u \in U$, $v \in V$, and $w \in W$ be homogeneous elements. We have
\begin{align*}
& g \otimes (1 \otimes u \otimes v) \otimes w \\
\mapsto & g \otimes u \otimes (1 \otimes v \otimes w) \\
\mapsto & (-1)^{m(n+p)(\vert u \vert+\vert v \vert+\vert w \vert)+\vert u \vert(\vert v \vert + \vert w \vert)} g\tilde{\tau}_{m,n+p}^{-1} \otimes (1 \otimes v \otimes w) \otimes u \otimes \pi^{m(n+p)} \\
\mapsto & (-1)^{m(n+p)(\vert u \vert+\vert v \vert+\vert w \vert)+\vert u \vert(\vert v \vert + \vert w \vert)} g\tilde{\tau}_{m,n+p}^{-1} \otimes  v \otimes (1 \otimes w \otimes u) \otimes \pi^{m(n+p)}
\end{align*}
The lower path in the hexagon is the following composition:
\begin{align*}
& \Ind_{m+n,p}^{m+n+p}(\Ind_{m,n}^{m+n}(U \otimes V) \otimes W) \\
\xrightarrow{\beta_{U,V} \otimes \id_W} & \Ind_{m+n,p}^{m+n+p}(\Ind_{n,m}^{m+n}(V \otimes W) \otimes \bk[mn] \otimes U) \\
\xrightarrow{\sim} & \Ind_{m+n,p}^{m+n+p}(\Ind_{n,m}^{m+n}(V \otimes W) \otimes U) \otimes \bk[mn] \\
\xrightarrow{\alpha_{V,W,U}^{-1}} & \Ind_{n,m+p}^{m+n+p}(V \otimes \Ind_{n,p}^{n+p}(U \otimes W)) \otimes \bk[mn] \\
\xrightarrow{\id_V \otimes \beta_{U,W}} & \Ind_{n,m+p}^{m+n+p}(V \otimes \Ind_{p,m}^{m+p}(W \otimes U) \otimes \bk[mp]) \otimes \bk[mn] \\
\xrightarrow{\sim} & \Ind_{n,m+p}^{m+n+p}(V \otimes \Ind_{p,m}^{m+p}(W \otimes U)) \otimes \bk[m(n+p)]
\end{align*}
We have
\begin{align*}
& g \otimes (1 \otimes u \otimes v) \otimes w \\
\mapsto & (-1)^{mn(\vert u \vert+\vert v \vert) + \vert u \vert \vert v \vert} g \otimes (\tau_{m,n}^{-1} \otimes v \otimes u \otimes \pi^{mn}) \otimes w \\
\mapsto & (-1)^{mn(\vert u \vert+\vert v \vert) + \vert u \vert \vert v \vert + mn \vert w \vert} g\tilde{\tau}_{m,n}^{-1} \otimes (1 \otimes v \otimes u) \otimes w \otimes \pi^{mn} \\
\mapsto & (-1)^{mn(\vert u \vert+\vert v \vert) + \vert u \vert \vert v \vert + mn \vert w \vert} g\tilde{\tau}_{m,n}^{-1} \otimes v \otimes (1 \otimes u \otimes w) \otimes \pi^{mn} \\
\mapsto & (-1)^{mn(\vert u \vert+\vert v \vert) + \vert u \vert \vert v \vert + mn \vert w \vert+mp(\vert u \vert+\vert w \vert)+\vert u \vert \vert w \vert} g\tilde{\tau}_{m,n}^{-1} \otimes v \otimes (\tilde{\tau}_{m,p}^{-1} \otimes w \otimes u \otimes \pi^{mp}) \otimes \pi^{mn} \\
\mapsto & (-1)^{mn(\vert u \vert+\vert v \vert) + \vert u \vert \vert v \vert + mn \vert w \vert+mp(\vert u \vert+\vert w \vert)+\vert u \vert \vert w \vert+mp \vert v \vert} g\tilde{\tau}_{m,n}^{-1}j(1,\tilde{\tau}_{m,p}^{-1}) \otimes v \otimes (1 \otimes w \otimes u) \otimes \pi^{m(n+p)}
\end{align*}
where here $j=j_{n,m+p}$. In the final step, we used the identity
\begin{displaymath}
v \otimes (\tilde{\tau}^{-1}_{m,p} \otimes w \otimes u)
=(-1)^{mp \vert v \vert} (1,\tilde{\tau}^{-1}_{m,p})(v \otimes (1 \otimes w \otimes u))
\end{displaymath}
in $V \otimes \Ind_{p,m}^{m+p}(W \otimes U)$, and then moved $(1,\tilde{\tau}^{-1}_{m,p})$ into the first tensor factor via $j$. Since $\tilde{\tau}^{-1}_{m,n} j(1,\tilde{\tau}^{-1}_{m,p})=\tilde{\tau}^{-1}_{m,n+p}$ (Proposition~\ref{prop:tauj}) and the signs in the two compositions agree, the result follows.
\end{proof}

\subsection{Universal property}

Let $\Rep(\wt{\fS}_*)^{\rm pf}$ be the full subcategory of $\Rep(\wt{\fS}_*)$ on finite length projective objects. If $\chr(\bk)=0$ then the projective condition is automatic. The category $\Rep(\wt{\fS}_*)^{\rm pf}$ is ``nice'' in the sense defined in \S \ref{ss:nice}: it is additive, Karoubian, and admits a $\Pi$-structure. Let $\bV$ be the ``standard'' representation of $\wt{\fS}_{\ast}$: it is $\bk$ in degree~1 and~0 in other degrees. 

\begin{theorem} \label{thm:S-univ}
The pair $(\Rep(\wt{\fS}_*)^{\rm pf}, \bV)$ is the universal nice type Ia supersymmetric monoidal $\bZ$-supercategory with an object of degree~$1$, in the following sense. Suppose $\cC$ is a nice supersymmetric monoidal $\bZ$-supercategory and $M$ is an object of $\cC$ of degree~$1$. Then there exists a supersymmetric monoidal superfunctor $\Phi \colon \Rep(\wt{\fS}_*)^{\rm pf} \to \cC$ and an isomorphism $\iota \colon \Phi(\bV) \to M$. Furthermore, $(\Phi, \iota)$ is unique up to isomorphism.
\end{theorem}

\begin{proof}
One can prove this directly; however, we simply note that $\Rep(\wt{\fS}_*)^{\rm pf}$ is the ``nice envelope'' of $\tilde{\cS}$, and so inherits the universal property of $\tilde{\cS}$ (Theorem~\ref{thm:Suniv}).
\end{proof}

\subsection{Exterior powers of categories} \label{ss:exterior}

One can generalize the above discussion to use coefficients in a supercategory $\cA$. This leads to the notion of the exterior power of $\cA$, as defined in \cite{ganter}. The following discussion is intended as an extended comment, and so we omit many details.

Let $\cA$ be a nice supercategory (see \S \ref{ss:nice}), and let $\cA^{\hattimes n}$ be its $n$-fold self-product with respect to the tensor product $\hattimes$ defined in \S \ref{ss:envelopes} below. The symmetric group $\fS_n$ acts in a natural way on $\cA^{\hattimes n}$, and so $\wt{\fS}_n$ does too through the quotient map $\wt{\fS}_n \to \fS_n$. A {\bf $\Sym^n$-object} of $\cA$ is an object $M$ of $\cA^{\hattimes n}$ equipped with even isomorphisms $\alpha_{\sigma} \colon \sigma^*(M) \to M$ for $\sigma \in \fS_n$ satisfying the usual cocycle conditions. A {\bf $\lw^n$-object} of $\cA$ is an object $M$ of $\cA^{\hattimes n}$ equipped with isomorphisms $\beta_{\sigma} \colon \sigma^*(M) \to M$ for $\sigma \in \wt{\fS}_n$ satisfying the cocycle conditions and such that $\beta_c=-1$ and $\beta_{\sigma}$ is homogeneous of the same degree as $\sigma$. We let $\Sym^n(\cA)$ and $\lw^n(\cA)$ denote the categories of $\Sym^n$- and $\lw^n$-objects of $\cA$. We define $\Sym(\cA)$ and $\lw(\cA)$ to be the $\bZ$-supercategories with these homogeneous pieces.

One can endow $\Sym(\cA)$ with a symmetric monoidal structure and $\lw(\cA)$ with a supersymmetric monoidal structure, similar to the definitions used above for $\Rep(\wt{\fS}_*)$. We note that $\Rep(\wt{\fS}_*)$ itself can be viewed as $\lw(\SVec)$. These categories satisfy natural universal properties: e.g., given a nice supersymmetric monoidal $\bZ$-supercategory $\cB$, giving a supersymmetric monoidal $\bZ$-functor $\lw(\cA) \to \cB$ is the same as giving a superfunctor $\cA \to \cB$.

\section{Queer vector spaces} \label{s:queer}

\textit{We assume $\bk$ contains $\zeta_{16}$ throughout this section, and put $\sqrt{2} = \zeta_8(1-\zeta_4)$.}

\subsection{Queer vector spaces}

Let $V$ be a super vector space. A {\bf queer structure} on $V$ is an odd isomorphism $\nu \colon V \to V$ such that $\nu^2=1$. A {\bf queer vector space} is a vector space equipped with a queer structure. Suppose that $(U,\mu)$ and $(V,\nu)$ are queer vector spaces. A homogeneous morphism of queer spaces $(U,\mu) \to (V,\nu)$ is a homogeneous linear map $f \colon U \to V$ such that $f \mu = (-1)^{\vert f \vert} \nu f$. In this way, we have a supercategory $\QVec$ of queer vector spaces.

\subsection{The half tensor product} \label{ss:half}

Let $(U,\mu)$ and $(V,\nu)$ be two queer vector spaces. Then $\mu \otimes \nu$ defines an even automorphism of $U \otimes V$ squaring to $-1$. We define the {\bf half tensor product} of $U$ and $V$, denoted $2^{-1}(U \otimes V)$, to be the $\zeta_4$-eigenspace of $\mu \otimes \nu$. We note that the $(-\zeta_4)$-eigenspace is isomorphic to the $\zeta_4$-eigenspace, via $\mu \otimes 1$ or $1 \otimes \nu$. The half tensor product is an ordinary super vector space; it does not come with a preferred queer structure.

\subsection{The queer monoidal category}

Motivated by the above discussion, we define the {\bf queer category} $\Queer$ to be the $\bZ/2$-supercategory with $\Queer_0=\SVec$ and $\Queer_1=\QVec$. We now define a monoidal operation $\odot$ on $\Queer$. We define $\odot$ on objects as follows:
\begin{itemize}
\item Given $U,V \in \Queer_0$, we let $U \odot V=U \otimes V$.
\item Let $U \in \Queer_0$ and $V \in \Queer_1$, and let $\nu$ be the queer structure on $V$. Then $U \otimes V$ carries the queer structure $1 \otimes \nu$. We define $U \odot V$ to be the queer vector space $(U \otimes V, 1 \otimes \nu)$. We similarly define $V \odot U$.
\item Given $U,V \in \Queer_1$, we define $U \odot V$ to be the half tensor product $2^{-1}(U \otimes V)$.
\end{itemize}
The definition of $\odot$ on morphisms is straightforward. In this way, we have defined a superfunctor
\begin{displaymath}
\odot \colon \Queer \boxtimes \Queer \to \Queer.
\end{displaymath}
The unit object for this monoidal structure is the unit object of $\Queer_0$, i.e., the super vector space $\bk$. The isomorphisms $\lambda$ and $\rho$ are similar to $\SVec$, and we omit the details.

We now discuss the associativity constraint $\alpha$ for $\odot$. For this, we will require the following result:

\begin{proposition} \label{prop:queer-assoc}
Let $(U,\mu)$, $(V,\nu)$, and $(W,\sigma)$ be queer vector spaces. Then $\mu \sigma-1$ is an automorphism of $U \otimes V \otimes W$ that maps $U \otimes 2^{-1}(V \otimes W)$ isomorphically to $2^{-1}(U \otimes V) \otimes W$.
\end{proposition}

\begin{proof}
Since $(\mu\sigma)^2=-1$, the inverse of $\mu \sigma-1$ is $-\tfrac{1}{2} (\mu \sigma+1)$. Now suppose that $x \in U \otimes 2^{-1}(V \otimes W)$. Thus $\nu \sigma x = \zeta_4 x$; applying $\sigma$ to each side, we find $\nu x = -\zeta_4 \sigma x$. We thus have
\begin{displaymath}
\mu \nu (\mu \sigma-1) x = -\nu \sigma x - \mu \nu x = -\zeta_4 x+\zeta_4 \mu \sigma x = \zeta_4 (\mu \sigma-1) x,
\end{displaymath}
which shows that $(\mu \sigma-1) x$ belongs to $2^{-1}(U \otimes V) \otimes W$. By symmetry, $\sigma \mu-1$ maps $2^{-1}(U \otimes V) \otimes W$ to $U \otimes 2^{-1}(V \otimes W)$. Since $\mu \sigma-1$ and $\sigma \mu-1$ are inverses (up to a factor of~2), these maps are isomorphisms.
\end{proof}

Let $U,V,W \in \Queer$ be homogeneous objects. We define an isomorphism
\begin{displaymath}
\alpha_{U,V,W} \colon U \odot (V \odot W) \to (U \odot V) \odot W
\end{displaymath}
as follows. We identify the domain and target above with subspaces of $U \otimes V \otimes W$. If at least one of $U$, $V$, or $W$ has degree~0, then the domain and target are identified with the same subspace, and we take $\alpha$ to be the identity map. Suppose now that $U$, $V$, and $W$ are each degree~1, and let $\mu$, $\nu$, and $\sigma$ be the queer structures on them. We then take $\alpha_{U,V,W}=\tfrac{1}{\sqrt{2}}(\mu \sigma-1)$, which is an appropriate isomorphism by Proposition~\ref{prop:queer-assoc}.

\begin{proposition}
The pentagon axiom holds.
\end{proposition}

\begin{proof}
Let $U,V,W,X \in \Queer$ be homogeneous objects. We let $\mu, \nu, \sigma, \tau$ denote queer structures on these objects, respectively, whenever they have degree 1. We must show that the diagram
\begin{displaymath}
\xymatrix{
& (U \odot V) \odot (W \odot X) \ar[rd]^-{\alpha_{U \odot V,W,X}} & \\
U \odot (V \odot (W \odot X)) \ar[d]_-{\id_U \odot \alpha_{V,W,X}} \ar[ru]^-{\alpha_{U,V,W\odot X}} & &
((U \odot V) \odot W) \odot X \\
U \odot ((V \odot W) \odot X) \ar[rr]^-{\alpha_{U,V \odot W,X}} &&
(U \odot (V \odot W)) \odot X \ar[u]_-{\alpha_{U,V,W} \odot \id_X}
}
\end{displaymath}
commutes. We identify each space above with a subspace of $U \otimes V \otimes W \otimes X$. If at most two of $U$, $V$, $W$, and $X$ are degree~1 then all maps in the diagram are the identity map, and so it commutes. Now suppose exactly three of the objects have degree~1. Then in each of the two paths in the pentagon, exactly one $\alpha$ is not the identity, and these two $\alpha$'s are equal; for example, if $\vert U \vert=0$ then $\alpha_{U \odot V,W,X}=\id_U \odot \alpha_{V,W,X} = \tfrac{1}{\sqrt{2}} (\nu \tau-1)$. Finally, suppose that all four objects have degree~1. Then all edges in the pentagon except the bottom left and bottom right are the identity. We must show $(\mu\sigma-1)(\nu \tau-1)=2$ on $U \odot (V \odot (W \odot X))$. Suppose $x$ belongs to this space. Then we have $\mu \nu x = \zeta_4 x$, and so $\nu x=\zeta_4 \mu x$; similarly, $\sigma \tau x = \zeta_4 x$, and so $\tau x = \zeta_4 \sigma x$. We thus find
\begin{displaymath}
\nu \tau x = \zeta_4 \nu \sigma x = -\zeta_4 \sigma \nu x = \sigma \mu x = - \mu \sigma x.
\end{displaymath}
Thus
\begin{displaymath}
(\mu \sigma-1)(\nu \tau-1) x=(\mu \sigma-1)(-\mu \sigma-1)x = 2x,
\end{displaymath}
since $(\mu \sigma)^2=-1$. This completes the proof.
\end{proof}

\subsection{Supersymmetry}

We now aim to define a supersymmetry on $\odot$. We let $\tau$ be the standard symmetry on $\SVec$. We begin with the following observation.

\begin{proposition} \label{prop:queer-commute}
Let $(U,\mu)$ and $(V,\nu)$ be queer vector spaces. Then the diagram
\begin{displaymath}
\xymatrix@C=4em{
U \otimes V \ar[r]^{\mu \otimes \nu} \ar[d]_{\tau} & U \otimes V \ar[d]^{\tau} \\
V \otimes U \ar[r]^{\nu \otimes \mu} & V \otimes U }
\end{displaymath}
commutes up to $-1$. In particular, $\tau$ maps the $\zeta_4$-eigenspace of $\mu \otimes \nu$ on $U \otimes V$ to the $(-\zeta_4)$-eigenspace of $\nu \otimes \mu$ on $V \otimes U$.
\end{proposition}

\begin{proof}
The first statement follows from the fact that $\tau$ is a natural transformation and $\mu$ and $\nu$ are odd (see the first paragraph of \S \ref{ss:braiding}). The second statement follows from the first.
\end{proof}

Let $U,V \in \Queer$. We define an isomorphism
\begin{displaymath}
\beta_{U,V} \colon U \odot V \to V \odot U
\end{displaymath}
as follows. If at least one of $U$ or $V$ has degree~0 then $\beta_{U,V}=\tau_{U,V}$. Now suppose that both $U$ and $V$ have degree~1, and let $\nu$ be the queer structure on $V$. Then $\beta_{U,V}=\zeta_{16}^{-1} (\nu \otimes 1) \tau_{U,V}$. This maps $U \odot V$ into $V \odot U$ by Proposition~\ref{prop:queer-commute} and properties of the half tensor product (see \S \ref{ss:half}).

\begin{proposition}
$\beta$ defines a type~IIb supersymmetry on $\odot$.
\end{proposition}

\begin{proof}
One easily sees that $\beta$ is a natural transformation, and it is evident that $\beta_{U,V}$ has degree $\vert U \vert \vert V \vert$ when $U$ and $V$ are homogeneous objects. We will use the following convention in the remainder of the proof: for an endomorphism $\mu$ of a super vector space $U$, we write $\mu$ still for the induced endomorphism of $U \otimes V$ or $V \otimes U$. Note that, under this convention, $\mu$ commutes with $\tau_{U,V}$.

Now, let $U,V \in \Queer$ be homogeneous objects. If either $U$ or $V$ has degree~0 then $\beta_{V,U}\beta_{U,V}$ is the identity. Now suppose that both $U$ and $V$ have degree~1, and let $\mu$ and $\nu$ be the queer structures. Then
\begin{displaymath}
\beta_{V,U} \beta_{U,V} = \zeta_8^{-1} \mu \tau_{V,U} \nu \tau_{U,V} = \zeta_8^{-1}  \mu \nu \tau_{V,U} \tau_{U,V}=\zeta_8^{-1} \mu \nu = \zeta_8.
\end{displaymath}
In the final step, we used the fact that $\mu\nu=\zeta_4$ on $U \odot V$. We thus see that $\beta_{V,U}\beta_{U,V}=\fsB_{\sharp}(\vert U \vert, \vert V \vert) \id_{U \otimes V}$ in all cases, which verifies the requisite symmetry condition for $\beta$.

We now verify the hexagon axiom (H1). Let $U,V,W \in \Queer$ be homogeneous objects, and let $\mu,\nu,\sigma$ be the queer structures on the objects of degree~1. We must show that the diagram
\begin{displaymath}
\xymatrix@C=3em{
& U \odot (V \odot W) \ar[r]^-{\beta_{U,V\odot W}} &
(V \odot W) \odot U \ar[dr]^-{\alpha_{V,W,U}^{-1}} \\
(U \odot V) \odot W \ar[ur]^-{\alpha_{U,V,W}^{-1}} \ar[dr]_-{\beta_{U,V} \odot 1_C} & & &
V \odot (W \odot U)\\
& (V \odot U) \odot W \ar[r]_-{\alpha^{-1}_{V,U,W}} &
V \odot (U \odot W) \ar[ur]_-{1_B \odot \beta_{U,W}}
}
\end{displaymath}
commutes. (Recall that $\fsB_1=1$ identically, which is why the diagram must commute exactly.) If at most one of the objects has degree~1, then this is the usual hexagon in $\SVec$, which commutes. Now suppose exactly two objects have degree~1. Then all the associators are the identity (or, more precisely, induced by the standard associator on the ordinary tensor products). If $\vert U \vert=0$ then each of the $\beta$'s is just a $\tau$, and the diagram commutes. Otherwise, one of $\beta_{U,V}$ or $\beta_{U,W}$ is the usual $\tau$, and the other involves the same queer structure as $\beta_{U, V \odot W}$; for instance, if $\vert V \vert=0$ then
\begin{displaymath}
\beta_{U,V\odot W}=\zeta_{16}^{-1} \sigma \tau_{U,V \odot W}, \qquad \beta_{U,V}=\tau_{U,V}, \qquad \beta_{U,W}=\zeta_{16}^{-1} \sigma \tau_{U,W}.
\end{displaymath}
Thus the diagram commutes (the queer structure can be moved past all the $\tau$'s and then one has the ordinary hexagon axiom in $\SVec$).

Finally, suppose that all objects have degree~1. The top path is (induced by) the map
\begin{displaymath}
\tfrac{1}{2} (\mu\nu-1) \tau_{U,V \otimes W} (\sigma \mu-1) = \tfrac{1}{2} \tau_{U,V \otimes W} (\mu\nu-1) (\sigma \mu-1).
\end{displaymath}
The bottom path is (induced by) the map
\begin{displaymath}
\tfrac{1}{\sqrt{2}} \zeta_8^{-1} \sigma \tau_{U,W} (\sigma \nu-1) \nu \tau_{U,V}
=\tfrac{1}{\sqrt{2}} \zeta_8^{-1} \tau_{U,W} \tau_{U,V} \sigma (\sigma \nu-1) \nu
\end{displaymath}
Since $\tau_{U,V \otimes W} = \tau_{U,W} \tau_{U,V}$, it suffices to show that
\begin{displaymath}
\tfrac{1}{2} (\mu \nu-1)(\sigma \mu-1) = \tfrac{1}{\sqrt{2}} \zeta_8^{-1} \sigma (\sigma \nu-1) \nu
\end{displaymath}
as functions on $(U \odot V) \odot W$. This equation is equivalent to
\begin{displaymath}
\nu \sigma-\mu\nu-\sigma \mu +1=\sqrt{2} \zeta_8^{-1} (1-\sigma \nu).
\end{displaymath}
Since $\mu\nu=\zeta_4$ on the source, we have $\mu=-\zeta_4 \nu$, and so the left side of the above equation is
\begin{displaymath}
-\sigma \nu-\zeta_4+\zeta_4 \sigma \nu+1 = (1-\zeta_4)(1-\sigma \nu),
\end{displaymath}
and so the result follows, as $1-\zeta_4=\sqrt{2} \zeta_8^{-1}$.

Since $\beta$ satisfies the symmetry condition, the second hexagon axiom is a consequence of the first one. Thus the result is proved.
\end{proof}

\subsection{Further comments} \label{ss:qcom}

The pentagon and hexagon axioms in $\Queer$ are rather complicated. In fact, there is an alternate way of defining the category $\Queer$ that makes these axioms more transparent. Recall that the Clifford algebra $\Cl_n$ is generated by $n$ supercommuting odd elements that each square to~1. One formulation of Bott periodicity (discussed in Appendix~\ref{s:bott}) states that $\Cl_n$ and $\Cl_{n+2}$ are Morita equivalent. We can identify $\QVec$ with the category of $\Cl_1$-modules. The half tensor product is then the composition
\begin{displaymath}
\Mod_{\Cl_1} \times \Mod_{\Cl_1} \to \Mod_{\Cl_2} \to \SVec,
\end{displaymath}
where the first functor takes $(V,W)$ to $V \otimes W$, with obvious $\Cl_2$-structure, and the second functor is the Morita equivalence provided by Bott periodicity. It is the second step that introduces the complications.

To obtain a ``cleaner'' version of the queer category, we can simply omit the Morita equivalence above. More precisely, we could redefine $\Queer_i$ to be the 2-colimit of the categories $\Mod_{\Cl_n}$ taken over $n$ of the same parity as $i$. The tensor product would then be more natural and the axioms would be more transparent.

This new approach also points the way to a vast generalization: one can start with an arbitrary $\bZ/2$-supercategory $\cA$ and construct a ``queer version'' by considering $\Cl_{\rm even}$-modules in $\cA_0$ and $\Cl_{\rm odd}$-modules in $\cA_1$. In fact, one can also work over more general base fields where the periodicity in Clifford algebras may be~4 or~8 instead of~2. We call this generalization ``Clifford eversion,'' and carry out the details in the subsequent section.

\section{Clifford eversion} \label{s:clifford}

\subsection{Nice supercategories} \label{ss:nice}

Let $\cC$ be a supercategory. We say that $\cC$ is {\bf additive} if all finite direct sums exist. (The direct sum of two objects is an object equipped with \emph{even} inclusion and projection maps satisfying the usual conditions.) We say that $\cC$ is {\bf Karoubian} if every \emph{even} idempotent $e \colon A \to A$ in $\cC$ has an image. We say that $\cC$ is {\bf nice} if it is additive, Karoubian, and admits a $\Pi$-structure. We say that a $\Lambda$-supercategory is nice if each graded piece is. We will mostly be concerned with nice supercategories in this section. We note that if $M$ is an object in a nice supercategory and $V$ is a finite dimensional super vector space then we can make sense of $V \otimes M$, e.g., $\bk^{r \vert s} \otimes M=M^{\oplus r} \oplus \Pi(M)^{\oplus s}$. (In fact, this construction does not require the Karoubian condition.)

\subsection{The Clifford algebra}

The {\bf Clifford algebra} $\Cl_n$ is the $\bk$-superalgebra generated by odd-degree variables $\alpha_1, \ldots, \alpha_n$ subject to the relations $\alpha_i^2=2$ and $\alpha_i \alpha_j=-\alpha_j \alpha_i$ for $i \ne j$. If~2 is a square in $\bk$ then one can renormalize $\alpha_i$ so that it squares to~1 (as we did in \S \ref{ss:qcom}), which is a more common convention. We use our convention so that the following proposition works without need of $\sqrt{2}$:

\begin{proposition} \label{prop:spin-clifford}
There is a group homomorphism $\wt{\fS}_n \to \Cl_n^{\times}$ given by
\begin{displaymath}
\tilde{s}_i \mapsto \tfrac{1}{2} (\alpha_{i+1}-\alpha_i), \qquad
c \mapsto -1.
\end{displaymath}
\end{proposition}

\begin{proof}
Let $\sigma_i=\frac{1}{2}(\alpha_{i+1}-\alpha_i)$. We have
\begin{displaymath}
\sigma_i^2 = \tfrac{1}{4} (\alpha_{i+1}^2-\alpha_{i+1} \alpha_i-\alpha_i \alpha_{i+1}+ \alpha_i^2) = 1,
\end{displaymath}
as the middle terms cancel and the outer terms are each equal to~2. It is clear that $\sigma_i$ and $\sigma_j$ anti-commute if $\vert i-j \vert>1$. Finally, we have
\[
  \sigma_i\sigma_{i+1} \sigma_i = \tfrac{1}{2}(\alpha_i - \alpha_{i+2}) = \sigma_{i+1} \sigma_i \sigma_{i+1},
\]
so the braid relation holds. We thus see that the $\sigma$'s satisfy the relations defining the spin-symmetric group, and the claim follows.
\end{proof}

\subsection{Clifford modules}

Let $\cA$ be a supercategory. A {\bf $\Cl_n$-module} in $\cA$ is a pair $(M, \rho)$ consisting of an object $M$ of $\cA$ and a morphism of $\bk$-superalgebras $\rho \colon \Cl_n \to \End(M)$. Explicitly, a $\Cl_n$-module consists of an object $M$ together with odd-degree supermorphisms $\alpha_1, \ldots, \alpha_n \colon M \to M$ such that $\alpha_i^2=2$ and $\alpha_i \alpha_j=-\alpha_j \alpha_i$ for $i \ne j$. Suppose that $M$ and $N$ are two $\Cl_n$-modules in $\cA$. A {\bf morphism} $f \colon M \to N$ of parity $p$ is a morphism in $\cA$ of parity $p$ compatible with the $\Cl_n$-structures, i.e., such that $f \alpha_i=(-1)^p \alpha_i f$. We let $\Cl_n(\cA)$ be the supercategory of $\Cl_n$-modules in $\cA$.

One easily sees that if $\cA$ is additive or Karoubian then so is $\Cl_n(\cA)$. Moreover, if $\cA$ admits a $\Pi$-structure then so does $\Cl_n(\cA)$, via
\begin{displaymath}
\Pi(M; \alpha_1, \ldots, \alpha_n)=(\Pi(M); \Pi(\alpha_1), \ldots, \Pi(\alpha_n)).
\end{displaymath}
We thus see that if $\cA$ is nice then so is $\Cl_n(\cA)$. The following statement is clear, and we omit its proof.

\begin{proposition} \label{prop:cliff-iter}
  We have an equivalence $\Cl_m(\Cl_n(\cA)) \to \Cl_{n+m}(\cA)$ defined by
\begin{displaymath}
((M; \alpha_1, \ldots, \alpha_n); \alpha'_1, \ldots, \alpha'_m) \mapsto (M; \alpha_1, \ldots, \alpha_n, \alpha'_1, \ldots, \alpha'_m).
\end{displaymath}
\end{proposition}

Given a $\Cl_n$-module $(M; \alpha_1, \ldots, \alpha_n)$ in $\cA$ and an element $\sigma \in \fS_n$, put
\begin{displaymath}
\sigma_*(M; \alpha_1, \ldots, \alpha_n) = (M; \alpha_{\sigma(1)}, \ldots, \alpha_{\sigma(n)}).
\end{displaymath}
This construction defines a left action of $\fS_n$ on the category $\Cl_n(\cA)$. We let $\wt{\fS}_n$ act via the homomorphism $\wt{\fS}_n \to \fS_n$. The following proposition shows that every object of $\Cl_n(\cA)$ is canonically equivariant for the $\wt{\fS}_n$ action:

\begin{proposition} \label{prop:cliff-perm}
Let $(M; \alpha_1, \ldots, \alpha_n)$ be a $\Cl_n$-module. Then $\wt{\fS}_n$ naturally acts on $M$, and permutes the $\alpha$'s via the surjection $\wt{\fS}_n \to \fS_n$. In particular, for $\sigma \in \wt{\fS}_n$, we have a natural isomorphism of $\Cl_n$-modules
\begin{displaymath}
\sigma \colon (M; \alpha_1, \ldots, \alpha_n) \to (M; \alpha_{\sigma(1)}, \ldots, \alpha_{\sigma(n)})
\end{displaymath}
that has the same parity as $\sigma$.
\end{proposition}

\begin{proof}
  The action comes from Proposition~\ref{prop:spin-clifford} using the homomorphism $\phi \colon \wt{\fS}_n \to \Cl_n^\times$. We claim that $\phi(\sigma) \alpha_i = (-1)^{|\sigma|} \alpha_{\sigma^{-1}(i)} \phi(\sigma)$ for all $\sigma \in \wt{\fS}_n$. It is enough to check this when $\sigma = \tilde{s}_j$. Then we have (for the third line, $k \notin \{j,j+1\}$)
  \begin{align*}
    \phi(\tilde{s}_j) \alpha_j &= \tfrac12 (\alpha_{j+1} - \alpha_j) \alpha_j = \tfrac12 \alpha_{j+1} \alpha_j - 1 = -\tfrac12 \alpha_{j+1}(\alpha_{j+1} - \alpha_j) = -\alpha_{j+1} \phi(\tilde{s}_j),\\
    \phi(\tilde{s}_j) \alpha_{j+1} &= \tfrac12 (\alpha_{j+1} - \alpha_j) \alpha_{j+1} = 1 - \tfrac12 \alpha_j \alpha_{j+1}  = -\tfrac12 \alpha_{j}(\alpha_{j+1} - \alpha_j) = -\alpha_{j} \phi(\tilde{s}_j),\\
    \phi(\tilde{s}_j) \alpha_{k} &= \tfrac12 (\alpha_{j+1} - \alpha_j) \alpha_{k} = - \tfrac12 \alpha_k(\alpha_{j+1} - \alpha_j) = -\alpha_{k} \phi(\tilde{s}_j). \qedhere
  \end{align*}
\end{proof}

\subsection{Periodicity}

Throughout this section, we fix a value $p$ such that $\Cl_p$ is isomorphic as a superalgebra to a matrix superalgebra. (By Theorem~\ref{thm:bott}, we know that $p=8$ is possible for any $\bk$ and that we can take $p=2$ if $\zeta_4 \in \bk$.) Let $U_p$ be a $\Cl_p$-module such that the functor $\SVec \to \Mod_{\Cl_p}$ given by $V \mapsto V \otimes U_p$ is an equivalence, and let $\epsilon_p \in \Cl_p$ be an even idempotent such that the functor $\Mod_{\Cl_p} \to \SVec$ given by $M \mapsto \epsilon_p M$ is a quasi-inverse.

Let $\cA$ be a nice supercategory. The following is clear.

\begin{proposition} \label{prop:morita}
  The functor $\cA \to \Cl_p(\cA)$ given by $M \mapsto M \otimes U_p$ is an equivalence.
\end{proposition}

We let $\Phi_n \colon \Cl_n(\cA) \to \Cl_{n+p}(\cA)$ be the equivalence induced by the identification $\Cl_{n+p}(\cA)=\Cl_p(\Cl_n(\cA))$. For a non-negative integer $i$, define $\Cl_{(i;p)}(\cA)$ to be the 2-colimit of the system
\begin{displaymath}
\xymatrix{
\Cl_i(\cA) \ar[r]^-{\Phi_i} & \Cl_{i+p}(\cA) \ar[r]^-{\Phi_{i+p}} & \Cl_{i+2p}(\cA) \ar[r] & \cdots }
\end{displaymath}
Explicitly, an object of $\Cl_{(i; p)}(\cA)$ is a $\Cl_n$-module in $\cA$, for some integer $n \ge 0$ such that $n \equiv i \pmod p$. If $M$ is a $\Cl_m$-module and $N$ is a $\Cl_n$-module, with $m \equiv n \equiv i \pmod p$, then
\begin{displaymath}
\Hom_{\Cl_{(i; p)}(\cA)}(M, N) = \varinjlim \Hom_{\Cl_r(\cA)}(\Phi^{(r-m)/p}(M), \Phi^{(r-n)/p}(N)),
\end{displaymath}
where the colimit is taken as $r \to \infty$ over integers $r$ such that $r\equiv i \pmod p$. It is clear that the natural functor $\Cl_{i}(\cA) \to \Cl_{(i; p)}(\cA)$ is an equivalence, and so, in a sense, these limit categories are nothing new. However, the categories $\Cl_{(i; p)}(\cA)$  will sometimes be easier to work with, especially when dealing with products.

\subsection{Clifford eversion}

As before, let $p$ be such that $\Cl_p$ is a matrix superalgebra, and fix an integer $q$ divisible by $p$ (we allow $q=0$). Note that $q$ is necessarily even.

\begin{definition}
The {\bf Clifford eversion} of a nice $\bZ/q$-supercategory $\cA$, denoted $\Cl(\cA)$, is the nice $\bZ/q$-supercategory given by $\Cl(\cA)_i=\Cl_{(i;p)}(\cA_i)$.
\end{definition}

\begin{remark}
We defined $\Cl(\cA)_i$ to be $\Cl_{(i;p)}(\cA_i)$. In some ways, it is more natural to use $\Cl_{(i;q)}(\cA_i)$, but there is one problem with this: it is not defined if $q=0$ and $i<0$.
\end{remark}

\begin{remark}
The motivation for the definition of Clifford eversion comes from the queer category; see \S \ref{ss:qcom}.
\end{remark}

We let $\Cl^k$ be the $k$-fold iterate of $\Cl$. To study this, it is convenient to introduce a slight variant. Define $\Cl^{(k)}(\cA)$ to be the $\bZ/q$-supercategory given by $\Cl^{(k)}(\cA)_i=\Cl_{(ki;p)}(\cA)$. Thus $\Cl^{(1)}=\Cl$. The equivalence in Proposition~\ref{prop:cliff-iter} induces an equivalence $\Cl_{(i;p)}(\Cl_{(j;p)}(\cB)) \to \Cl_{(i+j;p)}(\cB)$ for any nice supercategory $\cB$. It follows that there is a natural equivalence $\Cl^{(k)}(\Cl^{(\ell)}(\cA)) \to \Cl^{(k+\ell)}(\cA)$. In particular, we see that there is a natural equivalence $\Cl^k(\cA) \to \Cl^{(k)}(\cA)$ for any $k \ge 0$.

\begin{proposition} \label{prop:evert-fin-order}
We have a canonical equivalence $\cA \cong \Cl^p(\cA)$.
\end{proposition}

\begin{proof}
Combine the natural equivalences $\cA \to \Cl^{(p)}(\cA)$ and $\Cl^p(\cA) \to \Cl^{(p)}(\cA)$.
\end{proof}

Suppose $\cA$ is a nice monoidal $\bZ/q$-supercategory. We define a natural monoidal structure on $\Cl(\cA)$ as follows. Let $r,s \in \bZ/q$ and let $r'$ and $s'$ be non-negative integers congruent to $r$ and $s$ modulo $p$. We have a functor $\Cl_{r'}(\cA_r) \boxtimes \Cl_{s'}(\cA_s) \to \Cl_{r'+s'}(\cA_{r+s})$ given by
\begin{displaymath}
(M; \mu_1, \ldots, \mu_{r'}) \otimes (N; \nu_1, \ldots, \nu_{s'}) \mapsto (M \otimes N; \mu_1, \ldots, \mu_{r'}, \nu_1, \ldots, \nu_{s'}).
\end{displaymath}
Here $\mu_i$ is shorthand for $\mu_i \otimes 1$, and similarly for $\nu_j$. Note that $M$ is an object of $\cA_r$ and $N$ is an object of $\cA_s$, so $M \otimes N$ is an object of $\cA_{r+s}$. Thus the above object does indeed belong to $\Cl(\cA)_{r+s}$. The diagram
\begin{displaymath}
\xymatrix@C=4em{
\Cl_{r'}(\cA_r) \boxtimes \Cl_{s'}(\cA_s) \ar[r]^{1 \otimes \Phi_{s'}} \ar[d]
& \Cl_{r'}(\cA_r) \boxtimes \Cl_{s'+p}(\cA_s) \ar[d] \\
\Cl_{r'+s'}(\cA_{r+s}) \ar[r]^{\Phi_{r'+s'}} & \Cl_{r'+s'+p}(\cA_{r+s}) }
\end{displaymath}
commutes, since in both cases the extra endomorphisms for $\Cl_p$ are appended to the end of the list. The diagram
\begin{displaymath}
\xymatrix@C=4em{
\Cl_{r'}(\cA_r) \boxtimes \Cl_{s'}(\cA_s) \ar[r]^{\Phi_{s'} \otimes 1} \ar[d]
& \Cl_{r'+p}(\cA_r) \boxtimes \Cl_{s'}(\cA_s) \ar[d] \\
\Cl_{r'+s'}(\cA_{r+s}) \ar[r]^{\Phi_{r'+s'}} & \Cl_{r'+s'+p}(\cA_{r+s}) }
\end{displaymath}
does not commute, but does commute up to canonical even isomorphism. Precisely, let $(M; \mu_{\bullet})$ and $(N; \nu_{\bullet})$ be objects of $\Cl_{r'}(\cA_r)$ and $\Cl_{s'}(\cA_s)$, and let $(U_q; \lambda_1, \ldots, \lambda_p)$ be the $\Cl_p$-module inducing periodicity. Then the top path in the above diagram yields $(M \otimes N; \mu_{\bullet}, \lambda_{\bullet}, \nu_{\bullet})$, while the bottom path yields $(M \otimes N; \mu_{\bullet}, \nu_{\bullet}, \lambda_{\bullet})$. These two objects are isomorphic via Proposition~\ref{prop:cliff-perm}, using $j_{r',p+s'}(1,\tilde{\tau}_{p,s'}) \in \wt{\fS}_{r'+s'+p}$. Note that this isomorphism is even since $p$ is even.

It follows from the above discussion that we have a well-defined functor $\Cl(\cA) \boxtimes \Cl(\cA) \to \Cl(\cA)$. The associativity and unit constraints on $\Cl(\cA)$ are defined in the obvious manner, and one readily verifies that the axioms of a monoidal $\bZ/q$-supercategory are satisfied.

We now investigate how the above construction interacts with braidings. In what follows, we assume $4 \mid q$. Suppose we have an isomorphism $\beta \colon \otimes \to \otimes \circ \Sigma$. Given objects $(M;\alpha_{\bullet})$ and $(N;\alpha'_{\bullet})$ as above, we define
\begin{displaymath}
\beta'_{(M;\alpha_\bullet),(N;\alpha'_\bullet)} \colon (M;\alpha_{\bullet}) \otimes (N;\alpha'_{\bullet}) \to (N;\alpha'_{\bullet}) \otimes (M;\alpha_{\bullet})
\end{displaymath}
to be the composition
\begin{displaymath}
(M \otimes N; \alpha_{\bullet}, \alpha'_{\bullet}) \xrightarrow{\beta_{M,N}}
(N \otimes M; \alpha_{\bullet}, \alpha'_{\bullet}) \xrightarrow{\tilde{\tau}^{-1}_{r,s}}
(N \otimes M; \alpha'_{\bullet}, \alpha_{\bullet}).
\end{displaymath}
The second map uses the action of $\wt{\fS}_{r+s}$ on $N \otimes M$, as in Proposition~\ref{prop:cliff-perm}. We note that $\vert \beta' \vert=\vert \beta \vert+rs$. For the following, we use the notation from \S\ref{ss:factor}.

\begin{proposition} \label{prop:evert-monoidal}
Maintain the above notation.
\begin{enumerate}[\rm \indent (a)]
\item If $\beta$ is a symmetry (or braiding) with respect to $\omega \in \Gamma_0$ then $\beta'$ is a type~II supersymmetry (or superbraiding) with respect to $\fsD \fsA \omega$.
\item If $\beta$ is a type~II supersymmetry (or superbraiding) with respect to $\omega \in \Gamma_1$ then $\beta'$ is a symmetry (or braiding) with respect to $\fsA \omega$.
\end{enumerate}
\end{proposition}

\addtocounter{equation}{-1}
\begin{subequations}
\begin{proof} 
Suppose that $\beta$ is a braiding with respect to $\omega=(\omega_1, \omega_2)$. We show that $\beta'$ is a type~II superbraiding with respect to $\fsA \omega$. We have already observed that $\beta'_{X,Y}$ is homogeneous of degree $\vert X \vert \vert Y \vert$ if $X$ and $Y$ are homogeneous objects. We now verify that $\beta'$ is a natural transformation. Suppose that $f \colon (M;\mu_1, \ldots, \mu_r) \to (M'; \mu'_1, \ldots, \mu'_r)$ and $g \colon (N;\nu_1, \ldots, \nu_s) \to (N'; \nu'_1, \ldots, \nu'_s)$ are homogeneous morphisms of homogeneous objects of $\Cl(\cA)$. Consider the diagram:
\begin{displaymath}
\xymatrix@C=4em{
(M \otimes N; \mu_{\bullet}, \nu_{\bullet}) \ar[r]^{\beta_{M,N}} \ar[d]_{f \otimes g} &
(N \otimes M; \mu_{\bullet}, \nu_{\bullet}) \ar[r]^{\tilde{\tau}_{r,s}^{-1}} \ar[d]^{g \otimes f} &
(N \otimes M; \nu_{\bullet}, \mu_{\bullet}) \ar[d]^{g \otimes f} \\
(M' \otimes N'; \mu'_{\bullet}, \nu'_{\bullet}) \ar[r]^{\beta_{M',N'}} &
(N' \otimes M'; \mu'_{\bullet}, \nu'_{\bullet}) \ar[r]^{\tilde{\tau}^{-1}_{r,s}} &
(N' \otimes M'; \nu'_{\bullet}, \mu'_{\bullet}) }
\end{displaymath}
The first square commutes up to $(-1)^{\vert f \vert \vert g \vert}$, as this is simply what it means that $\beta$ is a natural transformation, while the second square commutes up to $(-1)^{(\vert f \vert+\vert g \vert) rs}$. This shows that $\beta'$ is natural.

We now look at the hexagon axioms. Let $(M;\mu_{\bullet})$, $(N;\nu_{\bullet})$, and $(P, \pi_{\bullet})$ be homogeneous objects of $\Cl(\cA)$ of degrees $r$, $s$, and $t$. Consider the following diagram:
\begin{displaymath}
\xymatrix@C=4em{
(MNP; \mu, \nu, \pi) \ar[r]^{\beta_{M,NP}} \ar[d]_{\beta_{N,M}} &
(NPM; \mu, \nu, \pi) \ar[r]^{x} &
(NPM; \nu, \pi, \mu) \\
(NMP; \mu, \nu, \pi) \ar[rd]_{y} &&
(NPM; \nu, \mu, \pi) \ar[u]_{z} \\
& (NMP; \nu, \mu, \pi) \ar[ru]_{\beta_{M,P}} }
\end{displaymath}
Here we have omitted $\otimes$ and $\bullet$ symbols, and $x$, $y$, and $z$ are the appropriate elements of $\wt{\fS}_{\ast}$ appearing in the definition of $\beta'$. This is the hexagon (H1) for $\beta'$, where we have omitted the associators. Put
\begin{displaymath}
(\alpha_1, \ldots, \alpha_{r+s+t})=(\mu_1, \ldots, \mu_r, \nu_1, \ldots, \nu_s, \pi_1, \ldots, \pi_t).
\end{displaymath}
We regard each $\alpha_i$ as acting on $MNP$, or any of the other tensor products in the above diagram, in the natural manner. We let $\wt{\fS}_{r+s+t}$ act on $MNP$ via the $\alpha$'s (as in Proposition~\ref{prop:spin-clifford}). Then
\begin{displaymath}
x=\tilde{\tau}^{-1}_{r,s+t}, \qquad
y=j_{r+s,t}(\tilde{\tau}^{-1}_{r,s},1), \qquad
z=(-1)^{rst} \cdot y \cdot j_{s,r+t}(1,\tilde{\tau}^{-1}_{r,t}) \cdot y^{-1}.
\end{displaymath}
The expressions for $x$ and $y$ are clear. By definition, $z$ is $\id_N \otimes \tilde{\tau}_{r,t}^{-1}$, but where this $\tilde{\tau}_{r,t}$ element is built from the $\mu$'s and $\pi$'s, i.e., the elements $\alpha_1, \ldots, \alpha_r$ and $\alpha_{r+s+1}, \ldots, \alpha_{r+s+t}$. This agrees with the above formula; note that
\begin{displaymath}
y \alpha_i y^{-1} = (-1)^{rs} \alpha_{\tau^{-1}_{r,s}(i)}
\end{displaymath}
by Proposition~\ref{prop:cliff-perm}. One readily verifies that $\beta_{M,P}$ commutes with the $\alpha_i$'s (in the evident sense), and thus with $y$ as well. Since $\beta_{M,P} \beta_{N,M} = \omega_1(r; s,t) \beta_{M,NP}$ (this is the first hexagon axiom for $\beta$) and $j_{s,r+t}(1,\tilde{\tau}_{r,t}) \cdot j_{r+s,t}(\tilde{\tau}_{r,s},1) = \tilde{\tau}_{r,s+t}$ (Proposition~\ref{prop:tauj}), we see that the above diagram commutes up to $(-1)^{rst} \omega_1(r; s,t)$, which verifies the first hexagon axiom for $\beta'$ with respect to $\fsD \fsA \omega$. The second hexagon axiom is similar. We thus see that $\beta'$ is a type~II superbraiding.

Now suppose that $\beta$ is symmetric, and let us show that $\beta'$ is supersymmetric. Maintaining the above notation, we have
\begin{displaymath}
\beta'_{(N;\nu),(M;\mu)} \beta'_{(M;\mu),(N;\nu)}
= y \beta_{N,M} x \beta_{M,N}
\end{displaymath}
where $x=\tilde{\tau}^{-1}_{r,s}$, and $y$ is defined like $\tilde{\tau}^{-1}_{s,r}$ but with respect to $\alpha_{r+1}, \ldots, \alpha_{r+s}, \alpha_1, \ldots, \alpha_r$; in other words, $y=(-1)^{rs} x\tilde{\tau}^{-1}_{s,r}x^{-1}=(-1)^{rs} \tilde{\tau}^{-1}_{s,r}$. Since $\beta_{N,M}$ and $x$ commute, we have
\begin{displaymath}
\beta'_{(N;\nu),(M;\mu)} \beta'_{(M;\mu),(N;\nu)}
= (-1)^{rs} \tilde{\tau}^{-1}_{s,r}  \tilde{\tau}^{-1}_{r,s} \beta_{N,M} \beta_{M,N}
=(-1)^{\binom{r}{2} \binom{s}{2} +rs} \omega_{\sharp}(\vert M \vert, \vert N \vert) \id_{M \otimes N},
\end{displaymath}
where we have used Proposition~\ref{prop:tau-symmetric}. We thus see that $\beta'$ is a supersymmetry with respect to $\fsD \fsA \omega$.

We have thus proved statement (a). Statement (b) is similar, though some additional signs come in: in the hexagon axiom, we pick up an additional $(-1)^{rst}$ when commuting $\beta_{M,P}$ and $y$, and in the symmetry axiom we pick up an additional $(-1)^{rs}$ when commuting $\beta_{N,M}$ and $\tilde{\tau}_{r,s}$.
\end{proof}
\end{subequations}

\begin{remark}
If one only cares about (super)braidings, one can phrase the above proposition using B-factor systems instead of S-factor systems.
\end{remark}

The above proposition can be summarized pictorially as follows: for $\omega \in \Gamma_0$, Clifford eversion acts by
\begin{displaymath}
\omega \to \fsD \fsA \omega \to \fsD \omega \to \fsA \omega \to \omega \to \cdots.
\end{displaymath}
That is, given a (super)symmetric monoidal category for one factor system, the Clifford eversion will be a (super)symmetric monoidal category for the next factor system. The list of factor system repeats with periodicity~4, while Clifford eversion itself repeats with period $p$ (which is a divisor of~8).

We have seen that there is a canonical equivalence $\cA \to \Cl^p(\cA)$. If $\cA$ is (super)braided or (super)symmetric then so is $\Cl^p(\cA)$, and with respect to the same factor system (recall that $q$ is divisible by~$4$). We now show that this equivalence is appropriately monoidal.

\begin{proposition}
The equivalence $\cA \to \Cl^p(\cA)$ is monoidal, (super)braided monoidal, or (super)symmetric monoidal, depending on what structure $\cA$ has.
\end{proposition}

\begin{proof}
The definition of the monoidal structure and (super)braiding on $\Cl(\cA)$ generalizes in the obvious manner to $\Cl^{(k)}(\cA)$ for any $k$. The equivalences $\Cl^{(k)}(\Cl^{(\ell)}(\cA)) \to \Cl^{(k+\ell)}(\cA)$ are easily seen to be compatible with these structures; it follows that the equivalence $\Cl^k(\cA) \to \Cl^{(k)}(\cA)$ is as well. The equivalence $\cA \to \Cl^{(p)}(\cA)$ induced by Clifford periodicity is also compatible with these structures.
\end{proof}

Combining the above results, we find:

\begin{proposition} \label{prop:cliff-eversion}
Suppose $q$ is divisible by~$4$ and let $\omega \in \Gamma_0$. Then the $2$-category of symmetric monoidal $\bZ/q$-supercategories with respect to $\omega$ is equivalent to the $2$-category of type~II supersymmetric monoidal $\bZ/q$-supercategories with respect to $\fsA \omega$.
\end{proposition}

Specializing to the canonical choice $\omega=1$, we find:

\begin{corollary} \label{cor:sym-super}
The $2$-category of symmetric monoidal $\bZ/q$-supercategories is equivalent to the $2$-category of type~IIa supersymmetric monoidal $\bZ/q$-supercategories, via Clifford eversion.
\end{corollary}

\begin{remark}
We have assumed $4 \mid q$ throughout the above discussion. There are analogous results if $p=2$ and $2 \mid q$, assuming one has sufficiently many roots of unity in $\bk$.
\end{remark}

\subsection{Envelopes} \label{ss:envelopes}

Let $\cC$ be a supercategory. We define its {\bf additive envelope} $\cC^\add$ as follows. The objects are symbols (formal direct sums) $A_1 \oplus \cdots \oplus A_r$ where the $A_i$ are objects of $\cC$. A morphism $f \colon A_1 \oplus \cdots \oplus A_r \to B_1 \oplus \cdots \oplus B_s$ is an element $(f_{i,j}) \in \prod_{i,j} \Hom_\cC(A_i, B_j)$, thought of as a matrix. We say that $f$ is homogeneous of parity $p$ if each $f_{i,j}$ is homogeneous of parity $p$. Given another morphism $g \colon B_1 \oplus \cdots \oplus B_s \to C_1 \oplus \cdots \oplus C_t$, the composition $gf$ is defined by $(gf)_{i,j} = \sum_k g_{k,j} f_{i,k}$. One readily verifies that $\cC^\add$ is an additive supercategory. We have a fully faithful functor $i \colon \cC \to \cC^\add$ that satisfies a universal property: given a superfunctor $j \colon \cC \to \cA$ to an additive supercategory $\cA$, there is a unique (up to isomorphism) superfunctor $j' \colon \cC^\add \to \cA$ such that $j = j'i$. In particular, $i$ is an equivalence if $\cC$ is additive. 

The {\bf Karoubian envelope} of a supercategory $\cA$, denoted $\cA^\kar$, consists of objects $(A,e)$ where $A$ is an object of $\cC$ and $e$ is an even idempotent of $A$. A morphism $(A,e) \to (B,f)$ in $\cA^\kar$ is a morphism $\phi \colon A \to B$ such that $f\phi = \phi = \phi e$. (The identity of $(A,e)$ is given by $e \colon A \to A$.) One easily sees that $\cA^{\kar}$ is a Karoubian supercategory. We have a fully faithful superfunctor $i \colon \cA \to \cA^\kar$ defined by $i(A) = (A,1)$ that satisfies a universal property: given an superfunctor $j \colon \cA \to \cB$ to a Karoubian supercategory $\cB$, there is a unique (up to isomorphism) superfunctor $j' \colon \cA^\kar \to \cB$ such that $j=j'i$. In particular, $i$ is an equivalence if $\cA$ is Karoubian.

\begin{lemma} \label{lem:fully-faithful}
Let $F \colon \cA \to \cB$ be a fully faithful superfunctor of supercategories. Then $F^\add \colon \cA^\add \to \cB^\add$ and $F^\kar \colon \cA^\kar \to \cB^\kar$ are also fully faithful.
\end{lemma}

\begin{proof}
This is clear for $F^\add$ by the explicit description for morphisms given above. Similarly, a morphism $(A,e) \to (A',e')$ in the Karoubian envelope is just a morphism $A \to A'$ satisfying some conditions which are preserved by $F$, so $F^\kar$ is also fully faithful.
\end{proof}

Let $\cA$ and $\cB$ be nice $\bZ/q$-supercategories. The product $\cA \boxtimes \cB$ will typically not be additive or Karoubian, and thus not nice. We define $\cA \hattimes \cB$ to be the Karoubian envelope of the additive envelope of $\cA \boxtimes \cB$. This is again a nice $\bZ/q$-supercategory, and $\hattimes$ endows the $2$-category of nice $\bZ/q$-supercategories with a monoidal structure. We note that if $\cA$ is a nice monoidal $\bZ/q$-supercategory then the monoidal product $\cA \boxtimes \cA \to \cA$ canonically extends to a superfunctor $\cA \hattimes \cA \to \cA$, by formal properties of envelopes.

\subsection{2-categorical aspects}

We now explain some 2-categorical aspects of Clifford eversion. As in \S \ref{s:2cat}, we do not intend this discussion to be completely rigorous. Fix a positive integer $q$ divisible by~4, and let $\sC$ be the 2-category of nice $\bZ/q$-supercategories, which we regard as monoidal under the product $\hattimes$. Let $\omega \in \Gamma_0$ be an even S-factor system. Then $\Sigma$ and $\omega$ induce a symmetric monoidal structure on $\sC$, just as $\rT$ and an odd S-factor system do (see \S \ref{s:2cat}). We write $(\sC, \Sigma, \omega)$ to denote the resulting symmetric monoidal 2-category. Of course, we similarly get a symmetric monoidal 2-category from $\rT$ and an odd S-factor system. The following is the main statement we want to explain, though we do not give a complete proof:

\begin{theorem} \label{thm:evert-2cat}
  Clifford eversion defines an equivalence of symmetric monoidal $2$-categories
  \[
    \Cl \colon (\sC, \Sigma, \omega) \to (\sC, \rT, \fsA \omega).
  \]
\end{theorem}

\begin{lemma}
Let $\cA$ and $\cB$ be nice $\bZ/q$-supercategories. We have a canonical equivalence
\[
  i_{\cA,\cB} \colon \Cl(\cA) \hattimes \Cl(\cB) \to \Cl(\cA \hattimes \cB).
\]
\end{lemma}

\begin{proof}
Let
\begin{displaymath}
j^{n,m}_{\cA,\cB} \colon \Cl_n(\cA) \boxtimes \Cl_m(\cB) \to \Cl_{n+m}(\cA \hattimes \cB)
\end{displaymath}
be the functor defined by
\begin{displaymath}
(M; \mu_1, \ldots, \mu_n) \boxtimes (N; \nu_1, \ldots, \nu_m) \mapsto (M \boxtimes N; \mu_1, \ldots, \mu_n, \nu_1, \ldots, \nu_m),
\end{displaymath}
where here we write $\mu_i$ in place of $\mu_i \otimes \id$ on the right side, and similarly for $\nu_j$.

We claim that $j^{n,m}_{\cA,\cB}$ is fully faithful. Faithfullness is clear. Let $A,A'$ be objects of $\Cl_n(\cA)$ and let $B,B'$ be objects of $\Cl_m(\cB)$. A homogeneous morphism $A \otimes B \to A' \otimes B'$ in $\Cl_{n+m}(\cA \hattimes \cB)$ is an element $f = \sum_k g_k \otimes h_k \in \Hom_\cA(A,A') \otimes \Hom_\cB(B,B')$, where the $g_k$ are homogeneous and linearly independent, and similarly for the $h_k$, such that
\[
  f (\mu_i \otimes 1) = (-1)^{|f|} (\mu'_i \otimes 1) f \qquad \text{and} \qquad f(1 \otimes \nu_j) = (-1)^{|f|} (1 \otimes \nu'_j) f
\]
for all $i,j$. The first condition becomes $\sum_k (-1)^{|h_k|} g_k \mu_i \otimes h_k = \sum_k (-1)^{|f|} \mu'_i g_k \otimes h_k$, so by linear independence of the $h_k$, we deduce that $g_k\mu_i = (-1)^{|g_k|} \mu'_i g_k$ for all $i,k$. (Note: $|f|=|g_k|+|h_k|$.) Similarly, we deduce that $h_k \nu_j = (-1)^{|h_k|} \nu'_j h_k$ for all $j,k$. Hence $j^{n,m}_{\cA,\cB}$ is full. Thus each $g_k$ defines a morphism $A \to A'$ in $\Cl_n(\cA)$ and each $h_k$ defines a morphism $B \to B'$ in $\Cl_m(\cB)$, and so $f$ is in the image of $j^{n,m}_{\cA,\cB}$. This proves the claim.

One easily sees (using an argument we have already made) that the $j^{n,m}_{\cA,\cB}$ are compatible with the periodicity equivalences. They therefore induce a superfunctor
\begin{displaymath}
j_{\cA,\cB} \colon \Cl(\cA) \boxtimes \Cl(\cB) \to \Cl(\cA \hattimes \cB)
\end{displaymath}
that is also fully faithful. Let
\begin{displaymath}
i_{\cA,\cB} \colon \Cl(\cA) \hattimes \Cl(\cB) \to \Cl(\cA \hattimes \cB)
\end{displaymath}
be the functor induced by $j_{\cA,\cB}$ by taking envelopes; the target does not change since it is already additive and Karoubian. By Lemma~\ref{lem:fully-faithful}, we see that $i_{\cA,\cB}$ is fully faithful. Iterating this construction, we obtain a fully faithful functor
\begin{displaymath}
i^k_{\cA,\cB} \colon \Cl^k(\cA) \hattimes \Cl^k(\cB) \to \Cl^k(\cA \hattimes \cB)
\end{displaymath}
for any $k \ge 0$. Consider the diagram
\begin{displaymath}
\xymatrix{
\Cl^{p-1}(\Cl(\cA)) \hattimes \Cl^{p-1}(\Cl(\cB)) \ar[rd]_-{i^{p-1}_{\Cl(\cA),\Cl(\cB)}} \ar[rr]^{i^p_{\cA,\cB}} &&
\Cl^{p-1}(\Cl(\cA \hattimes \cB)) \\
& \Cl^{p-1}(\Cl(\cA) \hattimes \Cl(\cB)) \ar[ru]_-{\Cl^{p-1}(i_{\cA,\cB})} }.
\end{displaymath}
This diagram commutes, essentially by definition. The two categories in the top row are each identified with $\cA \hattimes \cB$ by Proposition~\ref{prop:evert-fin-order}, and $i^p_{\cA,\cB}$ is easily seen to be identified with the identity functor; thus the top line is an equivalence. Since the two diagonal functors are fully faithful and their composite is an equivalence, each is an equivalence. Since $\Cl$ is an equivalence of 2-categories, again by Proposition~\ref{prop:evert-fin-order}, and $\Cl^{p-1}(i_{\cA,\cB})$ is an equivalence, it follows that $i_{\cA,\cB}$ is an equivalence.
\end{proof}

\begin{lemma}
The diagram
\begin{displaymath}
\xymatrix@C=6em{
\Cl(\cA) \hattimes \Cl(\cB) \ar[r]^{i_{\cA,\cB}} \ar[d]_{\rT_{\Cl(\cA),\Cl(\cB)}} & \Cl(\cA \hattimes \cB) \ar[d]^{\Cl(\Sigma_{\cA,\cB})} \\
\Cl(\cB) \hattimes \Cl(\cA) \ar[r]^{i_{\cB,\cA}} & \Cl(\cB \hattimes \cA) }
\end{displaymath}
commutes up to a canonical even isomorphism. 
\end{lemma}

\begin{proof}
Let $(M; \mu_{\bullet})$ be $(N; \nu_{\bullet})$ be objects of $\Cl_m(\cA)$ and $\Cl_n(\cB)$. Then
\begin{displaymath}
i_{\cA,\cB}((M;\mu_{\bullet}) \otimes (N; \nu_{\bullet})) = (M \otimes N; \mu_{\bullet}, \nu_{\bullet})
\end{displaymath}
where here we have written $\mu_i$ in place of $\mu_i \otimes 1$ and $\nu_j$ in place of $1 \otimes \nu_j$. The functor $\Sigma_{\cA,\cB} \colon \cA \hattimes \cB \to \cB \hattimes \cA$ induces a functor $\Cl_{m+n}(\Sigma_{\cA,\cB}) \colon \Cl_{m+n}(\cA \hattimes \cB) \to \Cl_{m+n}(\cB \hattimes \cA)$ that takes the above object to $(N \otimes M; \mu_{\bullet}, \nu_{\bullet})$, where now $\mu_i$ means $1 \otimes \mu_i$ and $\nu_j$ means $\nu_j \otimes 1$. Note that the order of the $\mu$ and $\nu$ does not change, since $\Cl_{m+n}(\Sigma_{\cA,\cB})$ simply applies $\Sigma_{\cA,\cB}$ to all components of the input. On the other hand,
\begin{displaymath}
\rT((M; \mu_{\bullet}) \otimes (N; \nu_{\bullet})) = \Pi^{mn}(N; \nu_{\bullet}) \otimes (M; \mu_{\bullet}).
\end{displaymath}
Applying $i_{\cB,\cA}$ to this yields
\begin{displaymath}
\Pi^{mn}(N \otimes M; \nu_{\bullet}, \mu_{\bullet}).
\end{displaymath}
We have a natural even degree isomorphism
\begin{displaymath}
(N \otimes M; \mu_{\bullet}, \nu_{\bullet}) \to \Pi^{mn}(N \otimes M; \nu_{\bullet}, \mu_{\bullet})
\end{displaymath}
coming from Proposition~\ref{prop:cliff-perm} using the element $\tilde{\tau}^{-1}_{m,n}$ defined in \S\ref{ss:calculations}: note that $\tilde{\tau}^{-1}_{m,n}$ has parity $mn$.
\end{proof}

We know that Clifford eversion is an equivalence of 2-categories $\sC \to \sC$. The main content of Theorem~\ref{thm:evert-2cat} is that it admits a symmetric monoidal structure. A definition of braided monoidal 2-functor is given in \cite[\S 4]{baez}, though this definition is not the most general possible, and only concerns braidings (not symmetries or syllepses). The above lemmas define much of the structure required for $\Cl$ to be a braided monoidal 2-functor (though we have not discussed the modification $\alpha$ in \cite[Definition~16]{baez}). To complete the proof that $\Cl$ is braided monoidal, we would need to verify the quite complicated conditions in loc.\ cit. We leave this (together with the problem of upgrading from braided monoidal to symmetric monoidal) as an open problem. We note that the particular factor systems appearing in Theorem~\ref{thm:evert-2cat} have not appeared in the above discussion; however, we know these factor systems must be the correct ones due to Proposition~\ref{prop:evert-monoidal}: precisely, if $\cA$ is a symmetric monoidal object in $(\sC, \Sigma, \omega)$, i.e., a symmetric monoidal supercategory with respect to $\omega$, then by Proposition~\ref{prop:evert-monoidal}, $\Cl(\cA)$ is a type~II supersymmetric monoidal supercategory with respect to $\fsA \fsD \omega$, and therefore (can be converted to) a type~I supersymmetric monoidal supercategory with respect to $\fsA \omega$ (via Proposition~\ref{prop:type12}), i.e., a symmetric monoidal object of $(\sC, \rT, \fsA \omega)$.

\subsection{Grothendieck groups and eversion}

Fix a $\Cl_1$-module $(\bk^{1|1}, \xi)$ in $\SVec$. Explicitly, let $e_0$ be a basis vector for $\bk^{1|0}$ and let $e_1$ be a basis vector for $\bk^{0|1}$ and let $\xi$ be the map given by $e_0 \mapsto e_1$ and $e_1 \mapsto 2e_0$. 

Let $\cA$ be a nice supercategory. For an object $M$ of $\cA$, let $F(M)$ be the $\Cl_1$-module $(M \otimes \bk^{1 \vert 1}, 1 \otimes \xi)$ in $\cA$; this defines a functor $F \colon \cA \to \Cl_1(\cA)$. Let $G \colon \Cl_1(\cA) \to \cA$ be the forgetful functor, i.e., $G(M, \alpha) = M$. 

\begin{proposition} \label{prop:Cl1}
We have natural even isomorphisms $G \circ F = \id \oplus \Pi$ and $F \circ G = \id \oplus \Pi$.
\end{proposition}

\begin{proof}
The statement for $G \circ F$ is clear. Suppose that $(M,\alpha)$ is a $\Cl_1$-module in $\cA$. Then $M \otimes \bk^{1|1}$ is naturally a $\Cl_2$-module via $\alpha=\alpha \otimes 1$ and $\xi=1 \otimes \xi$. One easily verifies that $2-\xi \alpha$ defines an even isomorphism $(M \otimes \bk^{1|1}, \alpha) \to (M \otimes \bk^{1|1}, \xi)$ of $\Cl_1$-modules; this is a variant of Proposition~\ref{prop:cliff-perm}. Since $(M \otimes \bk^{1|1}, \alpha)$ is naturally isomorphic to $(M, \alpha) \oplus \Pi(M,\alpha)$ in $\Cl_1(\cA)$ by an even degree isomorphism, the result follows.
\end{proof}

Now suppose that $\cA$ is equipped with an exact structure. Then $\Cl_n(\cA)$ has a natural exact structure as well.

\begin{proposition} \label{prop:K-2}
  The forgetful functor $\Cl_1(\cA) \to \cA$ induces an isomorphism
  \[
    \rK_+(\Cl_1(\cA))[\tfrac12] \to \rK_+(\cA)[\tfrac12].
  \]
\end{proposition}

\begin{proof}
This follows from Proposition~\ref{prop:Cl1} since $\id \oplus \Pi$ induces multiplication by~2 on $\rK_+$.
\end{proof}

Now suppose that $\cA$ is a $\bZ/2$-supercategory with an exact structure and that $\zeta_4 \in \bk$. We take $p=2$ in Clifford periodicity, so that the Clifford eversion $\Cl(\cA)$ is also a $\bZ/2$-supercategory. We define a $\bZ[1/\sqrt{2}]$-linear map
\[
  i \colon \rK_+(\Cl(\cA)) \otimes \bZ[1/\sqrt{2}] \to \rK_+(\cA) \otimes \bZ[1/\sqrt{2}]
\]
as follows. If $A$ is an object of $\Cl(\cA)_0=\cA_0$, then $i([A]) = [A]$. If $(A,\alpha)$ is an object of $\Cl_1(\cA)=\Cl_1(\cA_1)$, then we define $i([(A,\alpha)]) = [A] / \sqrt{2}$. Thus, on the odd piece we are just scaling the map induced by the forgetful functor by $1/\sqrt{2}$. It follows from Proposition~\ref{prop:K-2} that $i$ is an isomorphism of abelian groups. The following proposition explains the reason for the $\sqrt{2}$ factor:

\begin{proposition} \label{prop:K-ring}
If $\cA$ has a monoidal structure then $i$ is a ring isomorphism.
\end{proposition}

\begin{proof}
  Let $A$, $A'$ be objects of $\cA_0$ and let $(M,\alpha)$, $(M',\alpha')$ be objects of $\Cl_1(\cA_1)$. Then $A \otimes A'$ is the same when computed in $\Cl(\cA)$ or $\cA$, so $i([A][A']) = i([A])i([A'])$. Similarly, $A \otimes (M,\alpha) = (A \otimes M, 1 \otimes \alpha)$, so again $i([A][(M,\alpha)]) = i([A])i([(M,\alpha)])$.

  Finally, via $\Cl_2(\cA) \cong \cA$, $(M,\alpha) \otimes (M',\alpha')$ is identified with the $\zeta_4$-eigenspace $N$ of $\alpha \otimes \alpha'$. The $\zeta_4$-eigenspace and $-\zeta_4$-eigenspace of $M \otimes M'$ are interchanged by $\alpha \otimes 1$, and hence we see that $(M, \alpha) \otimes (M', \alpha') \cong N \oplus \Pi(N)$, so
  \[
i([(M,\alpha)][(M',\alpha')]) = [N] = \frac{1}{2} [M][M'] =    i([(M,\alpha)]) i([(M',\alpha')]). \qedhere
\]
\end{proof}

\section{Schur--Sergeev duality} \label{s:sergeev}

{\it We work over $\bk=\bC$ in this section for simplicity.}

\subsection{Polynomial representations}

For an integer $n \ge 0$, we let $\fq(n) \subset \fgl(n|n)$ denote the super Lie subalgebra of matrices of the form $\begin{bmatrix} A & B \\ B & A \end{bmatrix}$ where $A,B$ are $n \times n$ matrices. This is the {\bf queer Lie superalgebra}. Alternatively, if $V_n$ is a vector space of dimension $n$, $\fq(V_n)$ is the commutator of the $\Cl_1$-module structure on $V_{n|n}=V_n \otimes \bk^{1|1}$ in $\fgl(V_{n|n})$. There are natural embeddings $V_{n|n} \to V_{n+1|n+1}$ which induce embeddings $\fq(n) \to \fq(n+1)$, and we define $\fq(\infty)$ to be the direct limit.

Let $\bV = \varinjlim_n V_{n|n}$; this is the {\bf standard representation} of $\fq(\infty)$. The tensor powers $V_{n|n}^{\otimes d}$ are semisimple, and the centralizer of $\fq(n)$ is the Hecke--Clifford algebra $\cH_d$ (defined below) when $n \gg d$ \cite[Theorem 4.7]{WW}. The inclusion maps $V_{n|n}^{\otimes d} \to V_{n+1|n+1}^{\otimes d}$ are $\cH_d$-linear and hence the tensor powers $\bV^{\otimes d}$ are semisimple. The simple constituents of this representation are naturally indexed by strict partitions of $d$; we let $L_{\lambda}$ denote the simple corresponding to the partition $\lambda$ \cite[Theorem 4.8]{WW}. We say that a representation of $\fq(\infty)$ is {\bf polynomial} if it is a direct sum of $L_{\lambda}$'s. The category of polynomial representations is denoted $\Rep^\pol(\fq)$. It is naturally a symmetric monoidal $\bZ$-supercategory. The object $L_{\lambda}$ is homogeneous of degree $\vert \lambda \vert$.

\begin{remark}
Even though $\Pi(\bV) \cong \bV$, we do not have $\Pi(V) \cong V$ for all objects $V$ in $\Rep^\pol(\fq)$. More precisely, the endomorphism algebra of the irreducible $L_\lambda$, where $\lambda$ is a strict partition, is $\Cl_1$ if and only if $\ell(\lambda)$ is odd \cite[Lemma 1.40]{chengwang}.
\end{remark}

\subsection{The Hecke--Clifford algebra}

The symmetric group $\fS_n$ acts on the algebra $\Cl_n$ by $g \cdot \alpha_i = \alpha_{g(i)}$. The {\bf Hecke--Clifford algebra} is $\cH_n = \bC[\fS_n] \ltimes \Cl_n$. This is a superalgebra via the grading $\deg(g) = 0$ for $g \in \fS_n$ and $\deg(\alpha_i)=1$. We define a category $\Rep(\cH_\ast)$ whose objects are sequences $(V_n)_{n \ge 0}$ where $V_n$ is a finite-dimensional super vector space with a homogeneous action of $\cH_n$ such that $V_n=0$ for $n \gg 0$. We define $\Pi \colon \Rep(\cH_\ast) \to \Rep(\cH_\ast)$ by $(\Pi V)_n = \Pi(V_n)$. There is a natural identification of $\cH_{n,m}=\cH_n \otimes \cH_m$ with a subalgebra of $\cH_{n+m}$, and we define a monoidal structure on $\Rep(\cH_\ast)$ via
\[
  (V \otimes W)_n = \bigoplus_{i=0}^n \cH_n \otimes_{\cH_{i,n-i}} (V_i \boxtimes W_{n-i}).
\]
We define a symmetry $\beta'$ on this tensor product by
\begin{displaymath}
a \otimes v \otimes w \mapsto (-1)^{|v| |w|} a \tau_{i,n-i}^{-1} \otimes w \otimes v
\end{displaymath}
where $a \in \cH_n$, $v \in V_i$, and $w \in W_{n-i}$. We omit the straightforward check that this is well-defined.

The Hecke--Clifford algebra is closely related to the queer superalgebra. Let $\xi \colon \bV \to \bV$ be the given odd endomorphism. The algebra $\Cl_n$ has a right action on $\bV^{\otimes n}$ via
\[
(v_1 \otimes \cdots \otimes v_n)   \alpha_i= (-1)^{|v_{i+1}|+\cdots + |v_{n}|} (v_1 \otimes \cdots \xi(v_i) \cdots \otimes v_n),
\]
where the $v_i$ are homogeneous elements of $\bV$. The symmetric group $\fS_n$ also has a right action on $\bV^{\otimes n}$, in the usual manner. Together, these define a right action of the Hecke--Clifford algebra $\cH_n$ on $\bV^{\otimes n}$. The following is a typical formulation of Schur--Sergeev duality:

\begin{proposition} \label{prop:old-sergeev}
  The functor $\Rep(\cH_\ast) \to \Rep^{\pol}(\fq)$ given by $(N_n)_n \mapsto \bigoplus_n \bV^{\otimes n} \otimes_{\cH_n} N_n$ is an equivalence of symmetric monoidal $\bZ$-supercategories.
\end{proposition}

\begin{proof}
Let $\Phi$ denote the functor under discussion. The explicit decomposition of $\bV^{\otimes n}$ given in \cite[Theorem 4.8]{WW} shows that $\Phi$ is an equivalence of abelian categories. Let $M$ and $N$ be $\cH_m$ and $\cH_n$-modules, respectively. Then
  \begin{align*}
    \Phi(M \otimes N) &= \bV^{\otimes (m+n)} \otimes_{\cH_{m+n}} (\cH_{m+n} \otimes_{\cH_{m,n}} (M \otimes N) )  \\
                      &\cong (\bV^{\otimes m} \otimes \bV^{\otimes n}) \otimes_{\cH_{m,n}} (M \otimes N)\\
                      &\cong (\bV^{\otimes m} \otimes_{\cH_m} M) \otimes (\bV^{\otimes n} \otimes_{\cH_n} N) \\
    &= \Phi(M) \otimes \Phi(N),
  \end{align*}
  which shows that $\Phi$ is monoidal. Finally, we verify the symmetry. Let $\beta$ be the symmetry on $\Rep^\pol(\fq)$. Pick $v \otimes x \in \Phi(M)$ and $v' \otimes y \in \Phi(N)$ where $x \in M$, $y \in N$, $v \in \bV^{\otimes m}$, and $v' \in \bV^{\otimes n}$ are homogeneous. Then
  \begin{align*}
    \beta((v \otimes x)  \otimes (v' \otimes y)) &= (-1)^{(|v|+|x|)(|v'|+|y|)} (v' \otimes y) \otimes (v \otimes x)\\
    &= (-1)^{|v||v'| + |x||v'| + |x||y|} (v' \otimes v) \otimes (y \otimes x).
  \end{align*}
  On the other hand, under the identification $\Phi(M) \otimes \Phi(N) \cong \Phi(M \otimes N)$, $(v \otimes x) \otimes (v' \otimes y)$ maps to $(-1)^{|v'||x|} (v \otimes v') \otimes (1 \otimes x \otimes y)$ where $1$ is the unit of $\cH_{m+n}$. If we apply the symmetry $\beta'$ in $\Rep(\cH_*)$ to the second factor, then we get
  \begin{align*}
    (-1)^{|v'||x|+|x||y|} (v \otimes v') \otimes (\tau_{m,n}^{-1} \otimes y \otimes x) =     (-1)^{|v'||x|+|x||y|+|v||v'|} (v' \otimes v) \otimes (1 \otimes y \otimes x) ,
  \end{align*}
  which agrees with the previous expression, so $\Phi$ is symmetric.
\end{proof}

\subsection{Sergeev--Schur duality}

Let $\beta$ be the type~II supersymmetry on $\Rep(\wt{\fS}_*)$ obtained from the type~I supersymmetry described in \S \ref{ss:spin-symmetry} by the general process in \S \ref{ss:equivalence-type}. Thus, explicitly, if $M$ is a representation of $\wt{\fS}_m$ and $N$ is a representation of $\wt{\fS}_n$ then
\begin{displaymath}
\beta_{M,N} \colon \Ind_{\wt{\fS}_m \ttimes \wt{\fS}_n}^{\wt{\fS}_{n+m}}(M \otimes N) \to \Ind_{\wt{\fS}_n \ttimes \wt{\fS}_m}^{\wt{\fS}_{n+m}}(N \otimes M)
\end{displaymath}
is given by
\begin{displaymath}
g \otimes v \otimes w \mapsto (-1)^{\vert v \vert \vert w \vert+mn \vert g \vert + mn} g\tilde{\tau}^{-1}_{m,n} \otimes w \otimes v.
\end{displaymath}
By results in the aforementioned sections, $\beta$ is a type~II supersymmetry with respect to the factor system $\fsD \fsA$. Define $\beta^*$ by
\begin{displaymath}
\beta^*_{M,N}=(-1)^{\binom{mn}{2}} \zeta_4^{mn} \beta_{M,N}
\end{displaymath}
where $m=\vert M \vert$ and $n=\vert N \vert$. Note that $(-1)^{\binom{mn}{2}} \zeta_4^{mn} = \fnD(m,n)$ where $\fnD$ is as defined in Example~\ref{ex:cobound2}, and so $\beta^*$ is a type~II supersymmetry with respect to $\fsA$.

The following is our reformulation of Sergeev--Schur duality, and the main result of this section:

\begin{theorem} \label{thm:sergeev}
We have a natural equivalence of symmetric monoidal $\bZ$-supercategories
\begin{displaymath}
\Rep^{\pol}(\fq) \cong \Cl(\Rep(\wt{\fS}_*)),
\end{displaymath}
where here $\Rep(\wt{\fS}_*)$ is equipped with the supersymmetry $\beta^*$.
\end{theorem}

The proof will take the rest of the section. We omit the proof for the following fact, which is a straightforward calculation.

\begin{lemma} \label{lem:hecke-sym}
  We have an isomorphism of superalgebras
  \[
    \bC[\wt{\fS}_n]/(c+1) \otimes \Cl_n \to \cH_n
  \]
  given by $1 \otimes \alpha_i \mapsto \alpha_i$ and $\tilde{s}_j \otimes 1 \mapsto \frac{\zeta_4}{2} s_j(\alpha_j - \alpha_{j+1})$.
\end{lemma}

\begin{proposition} \label{prop:SH-CT}
  We have an equivalence of symmetric monoidal $\bZ$-supercategories
  \[
    \Cl(\Rep(\wt{\fS}_*)) \cong \Rep(\cH_\ast).
  \]
\end{proposition}

\begin{proof}
Via the isomorphism in Lemma~\ref{lem:hecke-sym}, an $\cH_n$-module $M$ is the same thing as a $\Cl_n$-module in $\Rep(\wt{\fS}_n)$. This gives an equivalence $\Phi \colon \Rep(\cH_*) \to \Cl(\Rep(\wt{\fS}_*))$ which preserves the $\bZ$-grading. We claim that this is a monoidal equivalence. Let $M$ be an $\cH_m$-module and let $N$ be an $\cH_n$-module. We have $\Cl_m \otimes \Cl_n \cong \Cl_{m+n}$, so that $\cH_{m+n} \otimes_{\cH_{m,n}} (M \otimes N)$ is isomorphic, as a vector space, to $\Ind^{\wt{\fS}_{m+n}}_{\wt{\fS}_m \ttimes \wt{\fS}_n}(M \otimes N)$. The action of $\Cl_m$ is via $M$ and the action of $\Cl_n$ is via $N$, so that the action of $\Cl_{m+n}$ is via the tensor product action, and hence $\Phi$ is a monoidal equivalence.

We now check symmetry. We have homomorphisms
\begin{align*}
\phi &\colon \wt{\fS}_{n+m} \to \cH_{n+m}, & \tilde{s}_i &\mapsto \tfrac{\zeta_4}{2} s_i (\alpha_{i+1}-\alpha_i) \\
\psi &\colon \wt{\fS}_{n+m} \to \cH_{n+m}, & \tilde{s}_i &\mapsto \tfrac{1}{2} (\alpha_{i+1}-\alpha_{i}).
\end{align*}
The existence of $\psi$ comes from Proposition~\ref{prop:spin-clifford}, while the existence of $\phi$ comes from Lemma~\ref{lem:hecke-sym}. Observe that $\phi(\tilde{s}_i)$ skew-commutes with $\alpha_j$, for all $i$ and $j$, and so $\phi(\tilde{s}_i)$ and $\psi(\tilde{s}_j)$ skew-commute for all $i$ and $j$. We also have $\phi(\tilde{s}_i)= \zeta_4 s_i \psi(\tilde{s}_i)$, and so $\phi(\tilde{s}_i) \psi(\tilde{s}_i) = \zeta_4 s_i$. It follows that we have
\begin{displaymath}
\tau_{m,n} = (-1)^{\binom{mn}{2}} \zeta_4^{mn} \phi(\tilde{\tau}_{m,n}) \psi(\tilde{\tau}_{m,n}).
\end{displaymath}
The vector space isomorphism
\[
  \Ind^{\wt{\fS}_{m+n}}_{\wt{\fS}_m \ttimes \wt{\fS}_n}(M \otimes N) \to \cH_{m+n} \otimes_{\cH_{m,n}} (M \otimes N)
\]
is given by $g \otimes v \otimes w \mapsto \phi(g) \otimes v \otimes w$. We thus have
\begin{align*}
\beta'_{M,N}(\phi(g) \otimes v \otimes w)
  &= (-1)^{\vert v \vert \vert w \vert} \phi(g) \tau_{m,n}^{-1} \otimes w \otimes v \\
  &= (-1)^{\vert v \vert \vert w \vert+\binom{mn}{2}} \zeta_4^{-mn} \phi(g) \psi(\tilde{\tau}_{m,n})^{-1} \phi(\tilde{\tau}_{m,n})^{-1} \otimes w \otimes v \\
  &= (-1)^{\vert v \vert \vert w \vert+\binom{mn}{2}+mn \vert g \vert} \zeta_4^{-mn} \psi(\tilde{\tau}_{m,n})^{-1} \phi(g \tilde{\tau}_{m,n}^{-1}) \otimes w \otimes v \\
  &= \psi(\tilde{\tau}_{m,n})^{-1} \beta^*(g \otimes v \otimes w).
\end{align*}
The expression on the final line is exactly the symmetry on $\Cl(\Rep(\wt{\fS}_*))$, and so the result follows.
\end{proof}

Theorem~\ref{thm:sergeev} now follows from Propositions~\ref{prop:SH-CT} and~\ref{prop:old-sergeev}.

\subsection{Symmetric functions}

In this section, $\lambda = (\lambda_1 \ge \cdots \ge \lambda_r)$ denotes an integer partition, $\ell(\lambda)$ denotes the number of nonzero entries of $\lambda$, and $|\lambda| = \sum_i \lambda_i$.

Let $\Lambda$ be the ring of symmetric functions in an alphabet $x_1, x_2, \dots$. Define symmetric functions $q_k(x)$ by the identity
\[
\sum_{k \ge 0} q_k(x) t^k = \prod_{i \ge 1} \frac{1 + x_i t}{1 - x_i t}.
\]
Let $\Gamma$ be the subring of $\Lambda$ generated by $q_1, q_2, \dots$. They satisfy the relations, for $n>0$,
\[
\sum_{i=0}^n (-1)^i q_i q_{n-i} = 0.
\]
So $\Gamma \otimes \bZ[\frac{1}{2}]$ is generated by $q_1, q_3, q_5, \dots$, which are algebraically independent. For $a,b$, define
\[
Q_{(a,b)} = q_aq_b + 2 \sum_{i=1}^b (-1)^i q_{a+i} q_{b-i}
\]
(with the convention that $q_i=0$ if $i<0$). Note that $Q_{(b,a)} = -Q_{(a,b)}$. Let $\lambda = (\lambda_1 > \cdots > \lambda_{2r} \ge 0)$ be a strict partition (here $\ell(\lambda)=2r$ if $\lambda_{2r}>0$ and otherwise $\ell(\lambda)=2r-1$) we define
\[
Q_\lambda = \Pf(Q_{(\lambda_i, \lambda_j)})_{1 \le i, j \le 2r}
\]
where $\Pf$ denotes the Pfaffian of a skew-symmetric matrix. The $Q_\lambda$ form a $\bZ$-basis for $\Gamma$ \cite[(8.9)]{macdonald} and $\Gamma \otimes \bZ[1/\sqrt{2}]$ is a model for the Grothendieck groups $\Rep^\pol(\fq)$, $\Rep(\wt{\fS}_*)$, and $\Rep(\cH_*)$ (in each case, the ring structure comes from the monoidal product) as we now explain. In particular, in each case, the simple objects are indexed by strict partitions.

For $\Rep^\pol(\fq)$ and $\Rep(\cH_\ast)$, the classes of the simple objects $L_\lambda$ are $2^{-\lfloor \ell(\lambda)/2 \rfloor} Q_\lambda$ \cite[Theorem 3.48]{chengwang} and for $\Rep(\wt{\fS}_\ast)$, the classes of the simple objects $\bN_\lambda$ are $2^{(\eps(\lambda)-\ell(\lambda))/2} Q_\lambda$ where $\eps(\lambda) \in \{0,1\}$ is the reduction modulo 2 of $|\lambda|-\ell(\lambda)$ \cite[Theorem 8.6]{HH}.

In particular, when $|\lambda|$ is even, we have $[L_\lambda] = [\bN_\lambda]$, and when $|\lambda|$ is odd, we have
\[
  [L_\lambda] = \begin{cases} \frac{1}{\sqrt{2}} [\bN_\lambda] & \text{if $\ell(\lambda)$ is even} \\
    \sqrt{2} [\bN_\lambda] & \text{if $\ell(\lambda)$ is odd} 
  \end{cases}.
\]
Recall from Proposition~\ref{prop:SH-CT} that we have an equivalence $\Rep(\cH_n) \simeq \Cl_1(\Rep(\wt{\fS}_n))$ when $n$ is odd. Hence, from Proposition~\ref{prop:K-ring}, we have an isomorphism
\begin{align*}
  \rK_+(\Cl_1(\Rep(\wt{\fS}_n))) \otimes \bZ[1/\sqrt{2}] &\to \rK_+(\Rep(\wt{\fS}_n)) \otimes \bZ[1/\sqrt{2}]\\
  [(A,\alpha)] &\mapsto [A]/\sqrt{2}.
\end{align*}

To match this with the Q-functions, we note that $\bN_\lambda$ is a queer object if and only if $|\lambda|-\ell(\lambda)$ is odd \cite[Corollary 4.19]{jozefiak}. Hence when $|\lambda|$ is odd, the simple objects of $\Cl_1(\Rep(\wt{\fS}_n))$ are
\[
  \{\bN_\lambda \mid \ell(\lambda) \text{ is even}\} \cup \{\bN_\lambda \otimes \bk^{1|1} \mid \ell(\lambda) \text{ is odd}\}
\]
and the correspondence becomes
\[
  [L_\lambda] \mapsto \begin{cases} \frac{1}{\sqrt{2}} [\bN_\lambda] & \text{if $\ell(\lambda)$ is even} \\
    \frac{1}{\sqrt{2}} [\bN_\lambda \otimes \bk^{1|1}] & \text{if $\ell(\lambda)$ is odd} 
  \end{cases}.
\]

\appendix

\section{Bott periodicity} \label{s:bott}

Let $\bk$ be a field of characteristic different from $2$. In this paper, the Clifford algebra $\Cl_k$ is the $\bk$-superalgebra generated by odd elements $\alpha_1,\dots,\alpha_k$ subject to the relations $\alpha_i^2 = 2$ for all $i$ and $\alpha_i \alpha_j = -\alpha_j\alpha_i$ for all $i \ne j$. Our goal here is to establish a version of Bott periodicity (Theorem~\ref{thm:bott}) for these Clifford algebras. While this is well-known, we have provided the details since most references work over $\bk = \bR$ and our definition of the algebra is not standard. Here we follow the proof in \cite{trimble}.

Given a superalgebra $A$, the opposite superalgebra $A^\op$ has the same underlying vector space but multiplication of homogeneous elements is given by $a \cdot_\op b = (-1)^{|a||b|}b \cdot a$. In particular, the presentation of $\Cl^\op_k$ is the same as $\Cl_k$ except that the $\alpha_i$ square to $-2$ instead of 2. Also, note that $\Cl_k \otimes \Cl_\ell \cong \Cl_{k+\ell}$ where we emphasize that we are use the tensor product of superalgebras, so that on homogeneous elements $(x_1 \otimes y_1)(x_2 \otimes y_2) = (-1)^{|x_2||y_1|} x_1x_2 \otimes y_1y_2$.

\begin{lemma} \label{lem:cliff1}
We have an isomorphism of superalgebras  $\Cl_1 \otimes \Cl^\op_1 \cong \End(\bk^{1|1})$.
\end{lemma}

\begin{proof}
  Let $e_0,e_1$ be a homogeneous basis for $\bk^{1|1}$. We represent $\alpha_1 \otimes 1$ by the operator $e_0 \mapsto e_1$, $e_1 \mapsto 2e_0$ and $1 \otimes \alpha_1$ by the operator $e_0 \mapsto e_1$, $e_1 \mapsto -2e_0$. 
\end{proof}

Let $H$ denote the quaternion algebra concentrated in degree $0$, i.e., $H$ is the associative $\bk$-algebra generated by even elements $\bi$ and $\bj$ such that $\bi\bj=-\bj\bi$ and $\bi^2=\bj^2=-1$.

\begin{lemma} \label{lem:cliff2}
  We have an isomorphism of superalgebras $\Cl_3 \cong H \otimes \Cl^\op_1$.
\end{lemma}

\begin{proof}
  We define $H \otimes \Cl^\op_1 \to \Cl_3$ by $\bi \otimes 1 \mapsto \frac12 \alpha_1 \alpha_2$, $\bj \otimes 1 \mapsto \frac12 \alpha_1\alpha_3$, and $1 \otimes \alpha_1 \mapsto \frac12 \alpha_1 \alpha_2 \alpha_3$. It is readily verified to be surjective and both algebras have dimension $8$ and hence it is an isomorphism.
\end{proof}

\begin{lemma} \label{lem:cliff3}
We have an isomorphism of superalgebras  $\Cl_4 \cong \Cl_4^\op$.
\end{lemma}

\begin{proof}
  By Lemma~\ref{lem:cliff1} and \ref{lem:cliff2}, we have
  \[
    \Cl_4 \cong \Cl_3 \otimes \Cl_1 \cong H \otimes \Cl^\op_1 \otimes \Cl_1 \cong H \otimes \End(\bk^{1|1}).
  \]
  Next, we have $H \cong H^\op$ via $\bi \mapsto -\bi$ and $\bj \mapsto -\bj$ and $\End(\bk^{1|1}) \cong \End(\bk^{1|1})^\op$ via $\begin{bmatrix} a & b \\ c & d \end{bmatrix} \mapsto \begin{bmatrix} a & -c \\ b & d \end{bmatrix}$ where the matrices are given with respect to a homogeneous basis of $\bk^{1|1}$.
\end{proof}

\begin{theorem}[Bott periodicity] \label{thm:bott}
  We have an isomorphism of superalgebras $\Cl_8 \cong \End(\bk^{8|8})$.

  Furthermore, if $\bk$ contains a primitive $4$th root of unity $\zeta_4$, then $\Cl_2 \cong \End(\bk^{1|1})$.
\end{theorem}

\begin{proof}
  By Lemma~\ref{lem:cliff1} and \ref{lem:cliff3}, we have
  \[
  \Cl_8 \cong \Cl_4 \otimes \Cl_4 \cong \Cl_4 \otimes \Cl_4^\op \cong \End(\bk^{1|1})^{\otimes 4} \cong \End(\bk^{8|8}). 
\]
The second statement follows from  Lemma~\ref{lem:cliff1} since we have $\Cl_1 \cong \Cl^\op_1$ via $\alpha_1 \mapsto \zeta_4 \alpha_1$.
\end{proof}

The periodicity comes when we consider Morita equivalence. Since $\End(\bk^{m|n})$ is a simple superalgebra for any $m,n$, its category of (super)modules is equivalent to the category of super vector spaces, and hence we get $\Mod_{\Cl_{k+8}} \simeq \Mod_{\Cl_k}$ for any $k$ and $\Mod_{\Cl_{k+2}}\simeq \Mod_{\Cl_k}$ when $\zeta_4 \in \bk$.


\begin{thebibliography}{KKTT}

\bibitem[BN]{baez} John C.~Baez, Martin~Neuchl. Higher-dimensional algebra I: braided monoidal categories. \textit{Adv.\ Math.} \textbf{121} (1996), 196--244. \arxiv{q-alg/9511013v1}

\bibitem[BGH${}^+$]{bruillard} Paul Bruillard, Cesar Galindo, Tobias Hagge, Siu-Hung Ng, Julia Yael Plavnik, Eric C. Rowell, Zhenghan Wang. Fermionic modular categories and the 16-fold way. {\it J. Math. Phys.} {\bf 58} (2017), no.~4, 041704, 31 pp. \arxiv{1603.09294v3} 

\bibitem[BE]{brundan} Jonathan Brundan, Alexander P. Ellis. Monoidal supercategories. {\it Comm. Math. Phys.} {\bf 351} (2017), no.~3, 1045--1089. \arxiv{1603.05928v3} 

\bibitem[CW]{chengwang} Shun-Jen Cheng, Weiqiang Wang. {\it Dualities and Representations of Lie Superalgebras}. Graduate Studies in Mathematics {\bf 144}, American Mathematical Society, Providence, RI, 2012.

\bibitem[Cr]{crans} Sjoerd E. Crans. Generalized centers of braided and sylleptic monoidal 2-categories. \textit{Adv.\ Math.} \textbf{136}(2) (1998), 183--223.

\bibitem[EL]{ellis} Alexander P.~Ellis and Aaron D.~Lauda. An odd categorification of $U_q(\fsl_2)$. {\it Adv.\ Math.} {\bf 265} (2014), 169--240. \arxiv{1307.7816v2}

\bibitem[FH]{fultonharris} William Fulton, Joe Harris. {\it Representation Theory: A First Course}. Graduate Texts in Mathematics {\bf 129}, Springer-Verlag, New York, 1991.

\bibitem[GK]{ganter} Nora Ganter, Mikhail Kapranov. Symmetric and exterior powers of categories. {\it Transform. Groups} {\bf 19} (2014), no.~1, 57--103. \arxiv{1110.4753v1}

\bibitem[Ho]{hoffman} Peter Hoffman. A projective analogue of Schur's tensor power construction. {\it Comm. Algebra} {\bf 21} (1993), no.~7, 2211--2249.

\bibitem[HH]{HH} P.~N. Hoffman, J.~F. Humphreys. {\it Projective Representations of the Symmetric Groups. Q-functions and Shifted Tableaux.} Oxford Mathematical Monographs. Oxford Science Publications. The Clarendon Press, Oxford University Press, New York, 1992.

\bibitem[JS1]{JoyalStreet} Andr\'e~Joyal, Ross~Street. Braided monoidal categories. Macquarie Math Reports \textbf{860081} (1986). Available at: \url{http://maths.mq.edu.au/~street/JS1.pdf}

\bibitem[JS2]{JoyalStreet2} A.~Joyal, R.~Street. Braided tensor categories. \textit{Adv.\ Math.} \textbf{102} (1993), no.~1, 20--78.

\bibitem[Jo]{jozefiak} Tadeusz J\'ozefiak. Schur Q-functions and cohomology of isotropic Grassmannians. {\it Math. Proc. Cambridge Philos. Soc.} {\bf 109} (1991), no.~3, 471--478.

\bibitem[KKT]{KKT} Seok-Jin Kang, Masaki Kashiwara, Shunsuke Tsuchioka. Quiver Hecke superalgebras. {\it J. Reine Angew. Math.} {\bf 711} (2016), 1--54. \arxiv{1107.1039v2}

\bibitem[Ka]{supersym} Mikhail~Kapranov. Supergeometry in mathematics and physics. \arxiv{1512.07042v2}

\bibitem[KV]{kapranov} M.~M.~Kapranov, V.~A.~Voevodsky. 2-categories and Zamolodchikov tetrahedra equations. In Proc.\ Symp.\ Pure Math.\ \textbf{56} Part~2 (1994), AMS, Providence, pp.~177-259. \\ Available at \url{https://www.math.ias.edu/vladimir/node/71}

\bibitem[Ke]{kelly} G. M. Kelly. {\it Basic Concepts of Enriched Category Theory.} Reprints in Theory and
Applications of Categories, No. 10, 2005. 

\bibitem[Ma]{macdonald} I.~G.~Macdonald. {\it Symmetric Functions and Hall Polynomials}, second edition, Oxford Mathematical Monographs, Oxford, 1995.

\bibitem[Mc]{MacLane} S.~Mac~Lane. Cohomology theory of abelian groups. \textit{Proc.\ International
Congress of Mathematicians, Volume II} (1950), 8--14.
  
\bibitem[MP]{morrison} Scott Morrison, David Penneys. Monoidal categories enriched in braided monoidal categories.  {\it Int. Math. Res. Not. IMRN} (2019), no.~11, 3527--3579. \arxiv{1701.00567v1} 

\bibitem[Na]{nazarov} M.~L. Nazarov. Young's orthogonal form of irreducible projective representations of the symmetric group. {\it J. London Math. Soc. (2)} {\bf 42} (1990), no.~3, 437--451. 

\bibitem[Tr]{trimble} Todd Trimble, The super Brauer group and super division algebras. Available at \url{http://math.ucr.edu/home/baez/trimble/superdivision.html}
  
\bibitem[Us]{usher} Robert Usher. Fermionic $6j$-symbols in superfusion categories.  {\it J. Algebra} {\bf 503} (2018), 453--473. \arxiv{1606.03466v1} 

\bibitem[WW]{WW} Jinkui Wan, Weiqiang Wang. Lectures on spin representation theory of symmetric groups. {\it Bull. Inst. Math. Acad. Sin. (N.S.)} {\bf 7} (2012), no.~1, 91--164. \arxiv{1110.0263v2}

\end{thebibliography}
\end{document}